\documentclass[11pt]{amsart}
\usepackage{graphpap}
\usepackage{amsmath}
\usepackage{cases}
\usepackage{amssymb}
\usepackage{hyperref}
\usepackage{xcolor}
\usepackage{color}
\newcommand{\noi}{\noindent}
\newcommand{\wt}{\widetilde}

\newtheorem{theorem}{Theorem}

\newtheorem{proposition}{Proposition}
\newtheorem{lemme}[proposition]{Lemma}

\newtheorem{corollaire}[proposition]{Corollary}

\newtheorem{remarque}[proposition]{Remark}

\numberwithin{equation}{section}
\numberwithin{proposition}{section}

\def\ov{\overline} 
\def\11{{\rm 1~\hspace{-1.4ex}l} }
\def\R{\mathbb R}
\def\C{\mathbb C}
\def\Z{\mathbb Z}
\def\N{\mathbb N}

\def\T{\mathbb T}

\begin{document}
	\title[Gibbs measure dynamics for the fractional NLS ]
	{
Gibbs measure dynamics for the fractional NLS
	}
	\author{Chenmin Sun, Nikolay Tzvetkov}
	\address{
		Universit\'e de Cergy-Pontoise,  Cergy-Pontoise, F-95000,UMR 8088 du CNRS}
	\email{nikolay.tzvetkov@u-cergy.fr}
	\email{chenmin.sun@u-cergy.fr}
	\begin{abstract} 
We construct global solutions on a full measure set with respect to the Gibbs measure for the one dimensional cubic fractional nonlinear Schr\"odinger equation (FNLS) with weak dispersion $(-\partial_x^2)^{\alpha/2}$,  $\alpha<2$ by quite different methods, depending on the value of $\alpha$. We show that if $\alpha>\frac{6}{5}$, the sequence of smooth solutions for FNLS with truncated initial data converges almost surely, and the obtained limit has recurrence properties as the time goes to infinity. The analysis requires to go beyond the available deterministic theory of the equation.  When $1<\alpha\leq \frac{6}{5}$,  we are not able so far to get the recurrence properties but we succeeded to use a method of Bourgain-Bulut to prove the convergence of the solutions of the FNLS equation with regularized both data and nonlinearity.  Finally, if $\frac{7}{8}<\alpha\leq 1$ we can construct global solutions in a much weaker sense by a classical compactness argument. 
\end{abstract}

	%\subjclass{ 35Q55, 35BXX, 37K05, 37L50, 81Q20 }
	%\keywords{}
	\maketitle
	%\tableofcontents

\section{Introduction}
\subsection{Motivation}
Invariant Gibbs measures for Hamiltonian PDE's were extensively studied in the last 35 years.
These studies aim to provide macroscopic properties for these PDE's. They have several perspectives. One of them (see the introduction of the seminal paper \cite{F}) is the extension of the recurrence properties of the solutions of Hamiltonian PDE's from integrable to non integrable models. Another (see \cite{B, bourgain, BB1,BB2,BB3,BTT-AIF,BTT-Toulouse,BT-IMRN,dS,Dem,Deng, FOSW, LRS,NORS,Oh1,Oh2,OTh,R,Tz-AIF,Z1,Z2,Z3}) is the construction of low regularity solutions.  As a consequence of the above mentioned works,  when considering the initial value problem of a Hamiltonian PDE for initial data on the support of the Gibbs measure, we now have methods to get weak solutions, to prove uniqueness of weak solutions and to get strong solutions (leading to recurrence properties).  It turns out that all these methods can be naturally applied in the context of the fractional NLS which is the goal of this article. It will be revealed that the strength of the dispersion will crucially influence on the nature of the obtained solutions. Our results leave the picture incomplete, several interesting problems remain to be understood.
%%%%%
\subsection{The fractional nonlinear Schr\"odinger equation}
We are interested in the one dimensional defocusing cubic fractional nonlinear Schr\"odinger equation (FNLS)
\begin{equation}\label{main-NLS}
i\partial_t u+|D_x|^\alpha u+|u|^{2}u=0, \quad (t,x)\in\R\times \T,
\end{equation}
where $u$ is complex-valued and $|D_x|^\alpha =(-\partial_x^2)^{\alpha/2}$ is defined as the Fourier-multiplier 
$\widehat{|D_x|^\alpha f}(n)=|n|^\alpha\widehat{f}(n)$. 
The parameter $\alpha$ measures the strength of the dispersion.
The equation  \eqref{main-NLS} is a Hamiltonian system with conserved energy functional
$$ H(u)=\int_{\T}||D_x|^{\frac{\alpha}{2}}u|^2dx+\frac{1}{2}\int_{\T}|u|^4dx\,.
$$ 
Moreover, the mass
$ M(u)=\int_{\T}|u|^2dx
$ 
is also conserved along the flow of \eqref{main-NLS}. 
The fractional Schr\"odinger equations was introduced in the
theory of the fractional quantum mechanics where the Feynmann path integrals
approach is generalized to $\alpha$-stable L\'evy process \cite{Laskin}. Also, it appears in the water
wave models (see \cite{Ionescu-Pusateri} and references therein). 
Finally, we refer to \cite{KLS} where the  fractional NLS on the line appears as a limit of the discrete NLS with long range interactions. 
%%%
\subsection{Construction of the Gibbs measure}
Roughly speaking, our aim in this article is to study how much dispersion $\alpha$ is needed to construct an invariant Gibbs measure for \eqref{main-NLS}. There are two aspects of the analysis. The first is the construction of the Gibbs measure, and the second is the construction of a dynamics on the support of the measure, leading to invariance of the Gibbs measure. In this subsection, we discuss the measure construction.
\\

Let $(g_n)_{n\in\Z}$ be a sequence of independent, standard complex-valued Gaussian random variables on a probability space 
$(\Omega,\mathcal{F},\mathbb{P})$. Let us consider the Gaussian measure $\mu$ on $H^{\frac{\alpha-1}{2}-\epsilon}(\T)$ for any $\epsilon>0$, induced by the map
\begin{equation}\label{induced-Gaussian}
\omega\longmapsto \sum_{n\in\Z} \frac{g_n(\omega)}{[n]^{\frac{\alpha}{2}}}e^{inx},
\end{equation}
where $[n]^{\frac{\alpha}{2}}=(1+|n|^{\alpha})^{\frac{1}{2}}$. Set $E_N=\textrm{span}\{e^{inx}: |n|\leq N
 \}$.
 %, $E_N^{\perp}=\text{span}\{e^{inx}:|n|>N \}$. 
 We denote by 
 $$
 \Pi_N:  H^{\frac{\alpha-1}{2}-\epsilon}(\T) \longrightarrow E_N
 $$ 
 the corresponding projection.
% We denote by $\mu_N$ the Gaussian measure on $E_N$ induced by 
 %$$ \omega\longmapsto \sum_{|n|\leq N}\frac{g_n(\omega)}{[n]^{\frac{\alpha}{2}}}e^{inx}. 
 %$$
 %Similarly, we define $\mu^{N}$. Then the Gaussian measure $\mu$ can be written as as $\mu=\mu_N\otimes\mu^{N}$ 
 %for any decomposition $E_N\otimes E_N^{\perp}$.
\\

If $\alpha>1$, it is well-known that for $0\leq \sigma<\frac{\alpha-1}{2}$, $\||D|^{\sigma}u\|_{L^{\infty}(\T)}$ is $\mu$-almost surely finite. Then the Gibbs measure $\rho$ associated with \eqref{main-NLS} is 
$$ d\rho(u)=e^{-V(u)}d\mu(u),\quad V(u)=\frac{1}{2}\|u\|_{L^4(\T)}^4.
$$
Formally, the measure $\rho$ can be seen as $Z^{-1}\exp(-H(u)-M(u))du$.
\\

However, if $\alpha\leq 1$, due to the fact that $\|u\|_{L^4(\T)}=\infty$, $\mu$-almost surely, a renormalization is needed, as described for instance in \cite{BTT-Toulouse} for the case $\alpha=1$. More precisely, we set
$$ 
\alpha_N=\mathbb{E}_{\mu}\left[\|\Pi_Nu\|_{L^2(\T)}^2\right]
$$
and
$$
f_N(u)=\frac{1}{2}\|\Pi_Nu\|_{L^4(\T)}^4-2\alpha_N\|\Pi_Nu\|_{L^2}^2+\alpha_N^2.
$$
Further, we define
$$ d\rho_N(u)=\beta_N e^{-f_N(u)}d\mu(u),
$$
where $\beta_N$ is chosen so that $\rho_{N}$ is a probability measure. Denote by
$$ H_N(u)=\||D_x|^{\frac{\alpha}{2}}u\|_{L^2}^2+f_N(u)
$$
the renormalized Hamiltonian functional, and the associated Hamiltonian equation
$$ i\partial_tu=\frac{\delta H_N}{\delta\ov{u}}
$$
reads
\begin{equation*}\label{weakEQ}
i\partial_tu_N+|D_x|^{\alpha}u_N+F_N(u_N)=0,
\end{equation*}
where $F_N$ stands for
$$ F_N(u_N)=\Pi_N(|u_N|^2u_N)-2\alpha_Nu_N.
$$ 
Similarly to \cite{BTT-Toulouse}, we will prove the following statement.
\begin{proposition}\label{convergence-Gibbs}
Assume that $\alpha\in\left(\frac{7}{8},1\right]$ and $1\leq p<\infty$. Then the sequence $(f_N)_{N\geq 1}$ converges in $L^p(d\mu(u))$ to some limit denoted by $f(u)$. Moreover, 
$$ e^{-f(u)}\in L^p(d\mu(u)). 
$$
Therefore, we can define a probability measure $\rho$ by
$$ 
d\rho(u)=C_{\infty}e^{-f(u)}d\mu(u).
$$
\end{proposition}
The lower bound $\alpha>\frac{7}{8}$ is by no means optimal, here we perform the simplest argument we found providing a framework for weak solutions techniques. Since $\alpha>\frac{1}{2}$ is the threshold for the renormalization of the squre of \eqref{induced-Gaussian}, we expect that the construction of the Gibbs measure can be performed for any $\alpha>\frac{1}{2}$.
\\

Observe that the measures $\mu$ and $\rho$ depend on $\alpha$ but for conciseness we omit the explicit mentioning of this dependence.  
%%%%%
\subsection{Weak solutions}
The measure construction of the previous subsection essentially implies the existence of weak solutions of \eqref{main-NLS} as we explain below.
Consider
\begin{equation}\label{truncated-NLS}
i\partial_tu+|D_x|^\alpha u+\Pi_N(|u|^2u)=0,\quad
u|_{t=0}=\displaystyle{\sum_{|n|\leq N}}\frac{g_n(\omega)}{[n]^{\frac{\alpha}{2}}}e^{inx}\,.
\end{equation}
The projection of the equation \eqref{truncated-NLS}onto $E_N$ is a Hamiltonian ODE  with a conserved energy
$$ H_N(u)=\int_{\T}|\Pi_Nu|^2dx+\frac{1}{2}\int_{\T}|\Pi_Nu|^4dx\,.
$$
Hence for any fixed $N$, \eqref{truncated-NLS} has almost surely a unique global solution $u^{\omega}_N$.
We have the following statement. 
\begin{theorem}\label{thm1}
Assume that $\alpha>1$ and $\sigma<\frac{\alpha-1}{2}$. 
There is a subsequence $(N_k)_{k\in\N}$, $N_k\rightarrow \infty$ of $(1,2,3,\cdots)$
and a sequence of $C(\R;H^{\sigma}(\T))$ valued random variables $(\widetilde{u}_{N_k})_{k\in\N}$ with the same law as $(u_{N_k}^{\omega})_{k\in\N}$ 
such that  $(\widetilde{u}_{N_k})_{k\in\N}$  converges a.s. in $C(\R;H^{\sigma}(\T))$ to some limit $u$ which solves \eqref{main-NLS} in the distributional sense. 
Moreover, $\rho$ is invariant under the map $u(0)\mapsto u(t)$, $t\in\R$.  
\end{theorem}
For $\alpha\leq 1$ we get convergence only after a renormalisation. Here is the precise statement.
\begin{theorem}\label{thm2}
Assume that $\alpha\in\left(\frac{7}{8},1 \right]$ and $\sigma<\frac{\alpha-1}{2}$.  Then there is a divergente sequence of real numbers $(c_N)_{N\in\N}$,
there is a subsequence $(N_k)_{k\in\N}$, $N_k\rightarrow \infty$ of $(1,2,3,\cdots)$ and a sequence of $C(\R;H^{\sigma}(\T))$ valued random variables  $(\widetilde{u}_{N_k})_{k\in\N}$ with the same law as $(u_{N_k}^{\omega})_{k\in\N}$, such that the sequence  $(e^{itc_{N_k}} \widetilde{u}_{N_k})_{k\in\N}$  converges a.s. in $C(\R;H^{\sigma}(\T))$ to some limit $u$. Moreover, $\rho$ defined by Proposition~\ref{convergence-Gibbs} 
is invariant under the map $u(0)\mapsto u(t)$, $t\in\R$.  
\end{theorem}
%%%%
\subsection{Uniqueness of the weak solutions}
In the case  $\alpha>1$ we can strongly improve Theorem~\ref{thm1} by showing that almost surely, the whole sequence $(u_N)_{N\in\N}$ of solutions to \eqref{truncated-NLS} converges\footnote{In an appendix we shall extend Theorem~\ref{thm3} to higher dimensions.} (without changing it). 
\begin{theorem}\label{thm3}
Assume that $\alpha>1$ and $\sigma<\frac{\alpha-1}{2}$.  The sequence  $(u^{\omega}_{N})_{N\in\N}$ 
of solutions of \eqref{truncated-NLS} converges a.s. in $C(\R;H^{\sigma}(\T))$  to some limit $u$ which solves \eqref{main-NLS} in the distributional sense. 
Moreover, $\rho$ is invariant under the map $u(0)\mapsto u(t)$, $t\in\R$.  
\end{theorem}
The proof of Theorem~\ref{thm3} uses a method introduced by  Bourgain-Bulut in \cite{BB1,BB2,BB3}.
We also mention that similar arguments were used by N.~Burq and the second author in the context of the probabilistic continuous dependence with respect to the initial data for the nonlinear wave equation with data of super-critical regularity (see \cite{BT-JEMS}).
\\

We point out that in Theorems~\ref{thm1},~\ref{thm2},~\ref{thm3} we do not show that the obtained limit satisfy the flow property which prevents us to apply the Poincar\'e recurrence theorem.
%%%%%%
\subsection{Strong  solutions}
In this article we call strong solutions these solutions which are the \emph{ unique } limits of smooth solutions of \eqref{main-NLS}, satisfying the flow property.
For that purpose we need to define the global flow of \eqref{main-NLS} for smooth data. 
The following theorem of J.~Thirouin assures the global well-posedness of \eqref{main-NLS} for smooth data.
\begin{theorem}[\cite{Thir}]\label{thm-Thir}
Assume that $\alpha>\frac{2}{3}$. Then for every $u_0\in C^{\infty}(\T)$ there is a unique solution $u \in C(\R;C^{\infty}(\T))$ of
$$
i\partial_tu+|D_x|^{\alpha}u+|u|^2u=0,\quad u|_{t=0}=u_0.
$$
Moreover, the flow map has a unique extension to the energy space $H^{\frac{\alpha}{2}}(\T)$.
\end{theorem}
In view of Theorem~\ref{thm-Thir}  and the remarkable recent work by F.~Flandoli on the Euler equation \cite{Fl} one may ask whether it is possible to construct weak solutions for $\alpha\in(\frac{7}{8},1]$ by using the smooth solutions of Theorem~\ref{thm-Thir} as an approximation sequence (compare with Theorem~\ref{thm1} and Theorem~\ref{thm2}).
\\

It tuns out that if the dispersion is slightly stronger than $\alpha>1$, we have the following convergence result.
\begin{theorem}\label{thm5}
Assume that $\alpha>\frac{6}{5}$ and $\sigma<\frac{\alpha-1}{2}$.
Then the sequence of smooth solutions $(u_N)_{N\in\N}$ of
\begin{equation*}\label{FNLS-sequence}
i\partial_tu_N+|D_x|^\alpha u_N+|u_N|^2u_N=0,\quad
u|_{t=0}=\displaystyle{\sum_{|n|\leq N}\frac{g_n(\omega)e^{inx}}{[n]^{\frac{\alpha}{2}}}},
\end{equation*}
defined by Theorem~\ref{thm-Thir} converges almost surely in $C(\R;H^{\sigma}(\T))$ to a limit which solves 
\eqref{main-NLS} in the distributional sense. 
\end{theorem}
More importantly, the unique limit satisfies the flow property. The following statement is essentially a more precise formulation of Theorem~\ref{thm5}.
\begin{theorem}\label{thm6}
Assume that $\alpha>\frac{6}{5}$. There exists a measurable set $\Sigma$ of full $\rho$ measure, so that for any $\phi\in\Sigma$, the Cauchy problem
$$ i\partial_tu+|D_x|^\alpha u+|u|^2u=0,\quad u|_{t=0}=\phi
$$
has a global solution such that
$$ 
u(t,\cdot)-e^{\frac{it}{\pi}\|\phi\|_{L^2(\T)}^2}e^{it|D_x|^\alpha}\phi\in C(\R;H^s(\T))
$$
for some $s\in\big(\frac{1}{2}-\frac{\alpha}{4},\alpha-1\big)$. The solution is unique in the sense that for every $T>0$, 
\begin{equation}\label{unicite}
 e^{-\frac{it}{\pi}\|\phi\|_{L^2(\T)}^2 } u(t,\cdot)-e^{it|D_x|^{\alpha}}\phi  \in X_T^{s,b},\quad b>1/2,
\end{equation}
where $X_T^{s,b}$ is the Bourgain space localized on $[-T,T]$ (see \eqref{localizedBourgain} below).
If we denote by $\Phi(t)$ the solution map then $\Phi(t)$ satisfies:
\begin{align*}
\Phi(t)(\Sigma)=\Sigma,\quad \forall t\in\R\,\text{ and }\,\, \Phi(t_1)\circ\Phi(t_2)=\Phi(t_1+t_2),\quad \forall t_1,t_2\in\R. 
\end{align*}
Moreover, for all $\sigma<\frac{\alpha-1}{2}$ and $t\in\R$,
$$ \|u(t,\cdot)\|_{H^{\sigma}(\T)}\leq
\Lambda(\phi)
 \log^3(1+|t|),
$$
where $\Lambda(\phi)$ is a constant depending on $\phi\in\Sigma$. Finally, for any $\rho$ measurable set $A\subset\Sigma$ and for any $t\in\R$, $\rho(A)=\rho(\Phi(t)A)$.
\end{theorem}
If $\alpha>\frac{4}{3}$, from the deterministic local well-posedness result in \cite{Cho}, the proof of Theorem~\ref{thm6} is much easier, see \cite{Dem}.
In fact, FNLS is known to be locally well-posed for initial data in $H^s(\T)$ with $s\geq\frac{1}{2}-\frac{\alpha}{4}$. If $\alpha>\frac{4}{3}$, we have $\frac{\alpha-1}{2}>\frac{1}{2}-\frac{\alpha}{4}$. Since the initial data is $\mu$-a.s. supported on $H^{\frac{\alpha-1}{2}-}(\T)$, the deterministic theory applies. 
However, if $\frac{6}{5}<\alpha\leq\frac{4}{3}$ then we need to prove a new probabilistic local well-posedness result. 
We conjecture that it is possible to extend Theorem~\ref{thm6} to the range $\alpha>1$ by adapting a more involved resolution ansatz (see Remark~\ref{fut} below). 
We will address this issue in a forthcoming work.
\\

For $\alpha>1$, a typical function with respect to $\mu$ is an $L^\infty$ function. 
As a consequence, if we were dealing with a similar problem for a parabolic PDE then thanks to the nice $L^\infty$ mapping properties of the heat flow the analysis would become essentially trivial. On the other hand, since we are dealing with a dispersive PDE, the linear problem is only well-posed in $L^2$ in the scale of the $L^p$ spaces which makes that even at positive regularities,  refined detereministic estimates and probabilistic considerations are essential in the analysis. A similar comment applies in the context of \cite{BT1/2,BT-JEMS} and all subsequent works.
\\

The proof of Theorem \ref{thm6} is divided into two parts. Firstly, we need to establish a local well-posedness theory. For this, we follow the roadmap of \cite{bourgain} (see also the subsequent works \cite{Co-Oh}, \cite{Nah-Staffilani}). An important new feature is that in sharp contrast with the case $\alpha=2$, for a general $\alpha$, the values of $$|n_1|^{\alpha}-|n_2|^{\alpha}+|n_3|^{\alpha}-|n_1-n_2+n_3|^{\alpha},\quad  n_1,n_2,n_3\in\Z
 $$ 
may be dense in an interval of size $1$. This causes losses of regularity which are delicate to control. We also emphasize that the phase factor $e^{-\frac{it}{\pi}\|\phi\|_{L^2(\T)}^2}$ in \eqref{unicite} makes the uniqueness class different from \cite{bourgain,Co-Oh,Nah-Staffilani}. Secondly, we need to extend the local solution to the global one and to prove the invariance of the good data set $\Sigma$ along the flow by using the measure invariance argument introduced by Bourgain in \cite{B}. Compared with the existing literature (see for example \cite{BTT-AIF, BT-IMRN,Sun-Tz} and references therein), the smoother part in the Bourgain space does not belong to the initial data space. This fact makes the choice of the $\Sigma$ more delicate. In particular, we make use of spaces with sum structure. 
\\

As a consequence of Theorem~\ref{thm6} and the Poincar\'e recurrence theorem, we get the following statement (we consider $\Sigma$ equipped with the topology inherited by the separable space $H^\sigma(\T)$). 
\begin{corollaire}\label{recurrence}
In the context of Theorem~\ref{thm6} for $\mu$ almost every $u_0\in\Sigma$ and all $t\in\R$, there is a subsequence $(n_k)_{k\in\N}$,  $n_k\rightarrow \infty$ of $(1,2,3,\cdots)$ , such that the solution of
$$ 
 i\partial_tu+ |D_x|^\alpha u+|u|^2u=0,\quad u|_{t=0}=u_0
 $$
 satisfies 
$$
\lim_{n_k\rightarrow \infty}\|u(n_kt)-u_0\|_{H^\sigma(\T)}=0,\quad \sigma<\frac{\alpha-1}{2}\,.
$$
\end{corollaire}
Another application of the flow property is the following stability result.
\begin{corollaire}\label{stability}
Let $f_1,f_2\in L^1(d\mu)$ and let $\Phi(t)$ be the flow of 
$$ 
 i\partial_tu+ |D_x|^\alpha u+|u|^2u=0,\quad u|_{t=0}=u_0
 $$
 defined $\mu$ a.s.   Then for every $t\in\R$, the transports of the measures
  $$
  f_1(u) d \mu(u),\quad f_2(u)d\mu(u)
  $$ 
  by $\Phi(t)$ are given by
 $$
 F_1(t,u)d\mu(u),\quad  F_2(t,u)d\mu(u)
 $$
 respectively, for suitable $F_1(t,\cdot),F_2(t,\cdot)\in L^1(d\mu)$.
 Moreover
 $$
 \|F_1(t,\cdot)-F_2(t,\cdot)\|_{L^1(d\mu)}=\|f_1-f_2\|_{L^1(d\mu)} \,.
 $$
\end{corollaire}
%%%
Corollary~\ref{stability} describes a general feature. A similar statement holds each time we deal with a PDE defining a flow under which a measure is quasi-invariant. 
For example,  thanks to a recent work by Forlano-Trenberth the result of Corollary~\ref{stability} remains true if the measure $\mu$ is replaced by the measure induced by the map 
$$
\omega\longmapsto \sum_{n\in\Z} \frac{g_n(\omega)}
{
(1+|n|^{s})^{\frac{1}{2}}
}
e^{inx},
$$
for $s>\alpha$ large enough. We refer to \cite{FT} for the  precise restriction on $s$.  There is a gap between the best $s$ and $\alpha$ leaving an interesting open problem. 
%%%%%%
\begin{remarque}
{\rm 
As already mentioned,  it is not clear to us how to get the the flow property described by Theorem~\ref{thm6} by the method of Bourgain-Bulut. At the present moment,  in the case $\alpha\in (1,\frac{6}{5}]$ we only know how to prove  almost sure convergence of the solutions of the ODE's :
$$ i\partial_tu+|D_x|^\alpha u+\Pi_N(|\Pi_Nu|^2\Pi_Nu)=0,\quad u|_{t=0}= \sum_{n\in\Z}\frac{g_n(\omega)}{[n]^{\frac{\alpha}{2}}}e^{inx}. 
$$
A similar comment applies to \cite{BB1,BB2,BB3}.
}
\end{remarque}

This article is organized as follows. In Section 2, we collect some preliminaries including the bilinear Strichartz inequality for the fractional NLS which has its own interest. From Section 3 to Section 6, we deal with the case $\alpha>\frac{6}{5}$ and prove Theorem \ref{thm5} and Theorem \ref{thm6}. More precisely, in Section 3 we prove the probabilistic local well-posedness by assuming the crucial deterministic and probabilistic tri-linear estimates which will be proved in Section 4 and Section 5. Then in Section 6, we detail the globalization procedure which allows us to obtain interesting dynamical properties, i.e. Corollary \ref{recurrence} and Corollary \ref{recurrence}.  Section 7 is devoted to the proof of Theorem \ref{thm3} by using the argument of Bourgain-Bulut. In Section 8 we deal with the case $\alpha<1$ and prove Theorem \ref{thm2} by standard Nelson type argument and probabilistic compactness method.  Finally we add an appendix to generalize the Bourgain-Bulut argument to high dimensional fractional NLS on any compact Riemannian manifold without boundary.

\subsection*{Acknowledgements}
The authors are supported by the ANR grant ODA (ANR-18-CE40- 0020-01). We would like to thank Sahbi Keraani for his comments while the first author visited the Laboratoire Paul Painlev\'e of Lille University. We are grateful to Phil Sosoe for several nice discussions while the authors visited Cornell University, in particular for pointing out the reference \cite{F}. 

\subsection*{Additional acknowledgements}
The authors would like to warmly thank Yuzhao Wang and Engin Başakoğlu who pointed out an error in Corollary 2.9 of the published version on SIAM J. Math. Analysis. The current manuscript is a  corrected version of that article. The corrections are only local which take place in Corollary 2.9, Corollary 2.11, and Subsection 4.1.2.

%%%%%%%%%%%%%%%%%%%%%%%%%%%%%%%%%%%%%%%%%%%%%%%%%%%%%
\section{Preliminaries}
\subsection{Calculus inequalities}
\begin{lemme}[\cite{FT}]\label{bound:resonance}
	If $n_1-n_2+n_3-n=0$, we define the resonant function $\Phi(\ov{n}):=|n_1|^{\alpha}-|n_2|^{\alpha}+|n_3|^{\alpha}-|n|^{\alpha}$. If $\{n_1,n_3\}\neq \{n_2,n\}$, $\Phi(\ov{n})$ never vanishes. Moreover,
	$$ |\Phi(\ov{n})|\gtrsim |n_1-n_2||n_2-n_3||n|_{\mathrm{max}}^{\alpha-2},
	$$
	where $|n|_{\mathrm{max}}=\max\{|n_1|,|n_2|,|n_3|, |n| \}$.
\end{lemme}
\begin{proof}
	See Lemma 2.1  of \cite{FT}.
\end{proof}
\begin{lemme}\label{convolution}
	Let $a>1\geq b\geq 0$ with $a+b>1$. Then there exists $C>0$, such that
	$$ \int_{\R}\frac{dy}{\langle x-y\rangle^a \langle y\rangle^{b} }\leq \frac{C}{\langle x\rangle^{b}},
	$$
	for any $x\in\R$.	
\end{lemme}	
\begin{proof}
We break the integral into $\int_{|y|\leq |x|/2}$ and $\int_{|y|>|x|/2}$. When $|y|\leq |x|/2$, we have $$ \int_{|y|\leq |x|/2}\frac{dy}{\langle x-y\rangle^a \langle y\rangle^{b} }\leq C\langle x\rangle^{-a+1-b}\log\langle x\rangle\leq C\langle x\rangle^{-b}.
$$ 
When $|y|>|x|/2$, we have
$$ \int_{|y|>|x|/2}\frac{dy}{\langle x-y\rangle^a \langle y\rangle^{b}}\leq C\langle x\rangle^{-b}.
$$ 
\end{proof}	

\begin{lemme}\label{lemme-summation}
	 Assume that $\frac{1}{2}<\beta\leq 1$, then for all $\gamma<2\beta-1$, there exists $C_{\gamma}>0$, such that for any $a\in\R$,
		$$ \sum_{n\in\Z}\frac{1}{\langle n\rangle^{\beta}\langle n-a\rangle^{\beta}}\leq\frac{C_{\gamma}}{\langle a\rangle^{\gamma}}. 
		$$
	
\end{lemme}
\begin{proof}
	We cut the sum in two parts $$\sum_{|n|\leq |a|/2}\langle n\rangle^{-\beta}\langle n-a\rangle^{-\beta}+\sum_{|n|>|a|/2}\langle n\rangle^{-\beta}\langle n-a\rangle^{-\beta}.$$
	Then the first term can be majorized by
	$$ C\langle a\rangle^{-\beta}\sum_{|n|\leq \frac{|a|}{2}}\langle n\rangle^{-\beta}\leq C\langle a\rangle^{1-2\beta}\log\langle a\rangle.
	$$
	The second term can be bounded by $C_{\gamma}\langle a\rangle^{-\gamma}$, thanks to $2\beta-1>0$.
\end{proof}
\subsection{Strichartz estimates and applications}
We proceed by the standard argument reducing the $L^4$ Strichartz estimate to a counting lemma. Denote by 
$$ S_{\alpha}(t)=e^{it|D_x|^{\alpha}}
$$
the Schr\"odinger semi-group.
Recall that the Bourgain space $X^{s,b}$ is associated with  the norm
$$ \|u\|_{X^{s,b}}^2:=\int_{\R}\sum_{n\in\Z}\langle n\rangle^{2s}\langle\tau-|n|^{\alpha}\rangle^{2b}|\widehat{u}(\tau,n)|^2d\tau.
$$
For finite time interval $I\subset\R$, the localized Bourgain space $X_I^{s,b}$ is defined via the norm 
\begin{equation}\label{localizedBourgain} \|u\|_{X_I^{s,b}}:=\inf\left\{\|v\|_{X^{s,b}}: v|_{I}=u \right\}.
\end{equation}
We will also use the notation $X_T^{s,b}$ to stand for $X_{[-T,T]}^{s,b}$. We have the following standard statements. 
\begin{lemme}[\cite{Tao}]\label{time-localization}
	Let $\eta\in \mathcal{S}(\R)$. Then for $0<T<1$, $s\in\R$ and $-\frac{1}{2}<b'\leq b<\frac{1}{2}$, we have the estimate 
	$$ \|\eta(t/T)u\|_{X^{s,b'}}\lesssim T^{b-b'}\|u\|_{X^{s,b}}.
	$$
\end{lemme}
%%%%%%
\begin{lemme}[\cite{GTV}]\label{inhomo-linear}
Let $\eta\in\mathcal{S}(\R)$. Then for $s\in\R, 1>b>\frac{1}{2}$, we have the estimate 
\begin{align*}
\Big\|\eta(t)\int_0^tS_{\alpha}(t-t')F(t')dt' \Big\|_{X^{s,b}}\lesssim \|F\|_{X^{s,b-1}}.
\end{align*}
\end{lemme}

Now we are going to derive some linear and bilinear $X^{s,b}$ estimates. Define the set of integers 
$$ 
A_{a,l,N_1,N_2}(r):=\{k\in\Z\,:\,
N_1\leq |k|\leq  2N_1,\,  N_2\leq |a-k|\leq 2N_2,\, ||k|^{\alpha}+|a-k|^{\alpha}-l|\leq r \}
$$
and
$A_{a,l,N}(r):=A_{a,l,N,N}(r).
$ For a dyadic number $N\geq 1$, we denote by $\mathbf{P}_N$ the Fourier projector on
$$ N\leq \langle n\rangle\leq 2N.
$$
For an interval $J\subset\R$, we denote by $\mathbf{P}_J$ the Fourier projector:
$$ \widehat{\mathbf{P}_Jf}(n)=\mathbf{1}_J(n)\widehat{f}(n).
$$
We have the following estimate. 
\begin{lemme}\label{Reduction-1}
	For any finite time interval $I\subset \R$, there exists $C>0$, depending only on $|I|$, such that
	$$ \|S_{\alpha}(t)\mathbf{P}_Nf\|_{L^4(I;L^4(\T) )}^2\leq C \sup_{a,l}\big(\# A_{a,l,N}(1/2)\big)^{1/2}\|\mathbf{P}_Nf\|_{L^2(\T)}^2.
	$$ 
\end{lemme}
\begin{proof}
We use an almost orthogonality argument in the time variable. 
Without loss of generality, we assume that $I=[0,1]$ and $f=\mathbf{P}_Nf$.  From a direct computation, we have
	\begin{equation}\label{desired}
	\|S_{\alpha}(t)f\|_{L^4(I;L_x^4)}^2=\Big(\sum_{a\in \Z}\|g_a(t)\|_{L_t^2(I)}^2 \Big)^{1/2},
	\end{equation} 
	where
	$$ g_a(t)=\sum_{k\in\Z}\widehat{f}(k)\widehat{f}(a-k)e^{it\varphi_a(k)},\quad \varphi_{a}(k)=|k|^{\alpha}+|a-k|^{\alpha}.
	$$
	We fix $\phi\in C_c^{\infty}(\widetilde{I})$, such that $\phi|_{I}\equiv 1$ where $\widetilde{I}$ is a slight enlargement of $I$. Thus
	\begin{equation*}
	%\label{desired1}
	\begin{split}
	\int_I|g_a(t)|^2dt\leq &\int_{\R}\phi(t)\Big|\sum_{k} \widehat{f}(k)\widehat{f}(a-k)e^{it\varphi_a(k)}\Big|^2dt\\
	=&\int_{\R}\phi(t)\Big|\sum_{l}\sum_{k:|\varphi_a(k)-l|\leq \frac{1}{2}} \widehat{f}(k)\widehat{f}(a-k)e^{it\varphi_a(k)} \Big|^2dt\\
	=&\sum_{l,l'}\sum_{|\varphi_a(k)-l|\leq \frac{1}{2}}\sum_{|\varphi_a(k')-l'|\leq \frac{1}{2}}\widehat{f}(k)\widehat{f}(a-k)\ov{\widehat{f}(k')}\ov{\widehat{f}(a-k')}\widehat{\phi}(\varphi_a(k')-\varphi_a(k) )\\
	\leq & C\sum_{l,l'}\frac{1}{1+|l-l'|^2}\sum_{k,k'}\mathbf{1}_{A_{a,l,N}(1/2)}(k)\mathbf{1}_{A_{a,l',N}(1/2)}(k')|F(a,k)F(a,k')|,
	\end{split}
	\end{equation*}
	where $F(a,k)=\widehat{f}(k)\widehat{f}(a-k)$ (here we use a slight abuse of notation : by $|\varphi_a(k)-l|\leq \frac{1}{2}$, we mean 
	$-\frac{1}{2}<\varphi_a(k)-l\leq \frac{1}{2}$).
	\\
	
	Now, by Schur's test, we arrive at 
	$$ \int_I|g_a(t)|^2dt  \leq C \sum_{l}\Big|\sum_{k}\mathbf{1}_{A_{a,l,N}(1/2)}(k)|F(a,k)|\Big|^2.
	$$
	Therefore, by Cauchy-Schwarz, we have
	\begin{equation*}
	\begin{split}
	\eqref{desired}\leq &C\Big( \sum_{a,l}\Big|\sum_{k}\mathbf{1}_{A_{a,l,N}(1/2)}(k)|\widehat{f}(k)\widehat{f}(a-k)|\Big|^2 \Big)^{1/2}\\
	\leq& C\Big(\sum_{l,a}\sum_{k}|\widehat{f}(k)\widehat{f}(a-k)|^2\mathbf{1}_{A_{a,l,N}(1/2)}(k)\#(A_{a,l,N}(1/2)) \Big)^{1/2}\\
	\leq & C\sup_{a,l}\big(\# A_{a,l,N}(1/2)\big)^{1/2}\|f\|^2_{L^2(\T)}.
	\end{split}
	\end{equation*}
	This completes the proof of Lemma \ref{Reduction-1}.
\end{proof}
We shall use the following elementary lemma.
\begin{lemme}\label{counting-lemma}
	Let $I,J$ be two intervals and $\varphi$ be a $C^1$ function, then
	\begin{equation*}
	\begin{split}
	%& (1)\quad \quad |\{\xi\in J: \varphi(\xi)\in I  \}|\leq \frac{|I|}{\inf_{\xi\in J} |\varphi'(\xi)| };\\
	&%(2)
	\quad \quad \#\{k\in J\cap \Z:\varphi(k)\in I  \}\leq 1+\frac{|I|}{\inf_{\xi\in J}|\varphi'(\xi)| }.
	\end{split}
	\end{equation*}
\end{lemme}

\begin{proposition}\label{counting-1}
	For $r\geq \frac{1}{100}$and $1<\alpha<2$, we have
	$$ \#A_{a,l,N_1,N_2}(r)\leq C \min(N_1,N_2)^{1-\frac{\alpha}{2}}r^{1/2}.
	$$	
\end{proposition}
\begin{proof}	
	First we assume that $N_1\ll N_2$ (a similar argument applies in the case  $N_2\ll N_1$). Then for $\varphi_a(\xi)=|\xi|^{\alpha}+|a-\xi|^{\alpha}$, we have $|\varphi_a'(\xi)|\gtrsim N_2^{\alpha-1}$. From Lemma \ref{counting-lemma}, we have $\#A_{a,l,N_1,N_2}(r)\lesssim rN_2^{-(\alpha-1)}+1$. On the other hand, we have the trivial bound $\#A_{a,l,N_1,N_2}(r)\lesssim N_1$. We can conclude in this case since
	$$ \min(N_1, rN_2^{-(\alpha-1)}+1 )\lesssim N_1^{1-\frac{\alpha}{2}}r^{1/2}.
	$$
	Now we assume that $N_1\sim N_2\sim N$. If $r\gtrsim N^{\alpha}$, we have the trivial estimate
	$$ \#A_{a,l,N}(r)\lesssim N\lesssim N^{1-\frac{\alpha}{2}}r^{\frac{1}{2}}.
	$$
	Now we assume that $r\ll N^{\alpha}$.
	 Let $0<\theta<1$ to be chosen later. We have 
	$$ \# A_{a,l,N}(r)=\#A_{1}(\theta)+\#A_{2}(\theta)+\#A_{3}(\theta),$$
	where
	\begin{align*}
	&  A_{1}(\theta)=A_{a,l,N}(r)\cap \{k: |k-a/2|\leq \theta^{-1}  \},\\
	& A_{2}(\theta)=A_{a,l,N}(r)\cap \{k: |k-a/2|>\theta^{-1},\,k(a-k)<0 \},\\
	 &A_3({\theta})=A_{a,l,N}(r)\cap \{k:|k-a/2|>\theta^{-1},\, k(a-k)\geq 0  \}.
	\end{align*}
We have trivially that $\#A_1(\theta)\leq 2\theta^{-1}$.  If $\xi$ and $a-\xi$ have different signs, we have
$$ |\varphi_a'(\xi)|=\alpha\big||\xi|^{\alpha-1}+|a-\xi|^{\alpha-1} \big|\gtrsim N^{\alpha-1}.
$$
Thus
$ \#A_2(\theta)\lesssim rN^{1-\alpha}.
$
If $\xi$ and $a-\xi$ have the same signs, we deduce that 
	$$ |\varphi_a'(\xi)|= \alpha| |\xi|^{\alpha-1}-|a-\xi|^{\alpha-1}|\gtrsim \frac{|2\xi-a|}{\max\{ |\xi|^{2-\alpha}, |a-\xi|^{2-\alpha}\}  }\geq \frac{\theta^{-1}}{N^{2-\alpha}},
	$$
	hence $\#A_3(\theta)\lesssim r\theta N^{2-\alpha}$.
	 Therefore,
	\begin{align}\label{!!!!}
	 \#A_{a,l,N}(r)\lesssim \theta^{-1}+r\theta N^{2-\alpha}+rN^{1-\alpha}.
		\end{align}
	If $r\ll N^{\alpha}$, we choose $\theta$ such that the first two terms have the same size. Therefore, $\theta=r^{-1/2}N^{\frac{\alpha}{2}-1}$. 
	It follows that
	$ \# A_{a,l,N}(r)\leq N^{1-\frac{\alpha}{2}}r^{1/2}, 
	$ where we used the fact that $r\ll N^{\alpha}$, in order to estimate the third term in the r.h.s. of \eqref{!!!!}.
	This completes the proof of Proposition \ref{counting-1}.
\end{proof}
%%%%%%%%%%%%%%%%%%%%%%
\begin{corollaire}\label{Strichartz-bilinear}
	Let $1<\alpha\leq 2$. We have the following linear and bilinear Strichartz estimates:
	\begin{equation*}
		\begin{split}
			& (1)\quad \|S_{\alpha}(t)\mathbf{P}_Nf\|_{L^4(I;L^4(\T))}
			\leq C N^{\frac{1}{2}\left(\frac{1}{2}-\frac{\alpha}{4}\right)}
			\|\mathbf{P}_Nf\|_{L^2(\T)}.\\
			& (2)\quad \|S_{\alpha}(t)\mathbf{P}_Mf\cdot S_{\alpha}(t)\mathbf{P}_Ng\|_{L^2(I;L^2(\T))}
			\leq C\min\{M,N\}^{\frac{1}{2}-\frac{\alpha}{4}}
			\|\mathbf{P}_Mf\|_{L^2(\T)}\|\mathbf{P}_Ng\|_{L^2(\T)}.
		\end{split}
	\end{equation*}
	Here the constant $C$ depends on $\alpha$ and on the length of $I$.
\end{corollaire}

\begin{proof}
	The estimate (1) is the direct consequence of Lemma~\ref{Reduction-1} and
	Proposition~\ref{counting-1}, applied with $r=1$.
	
	We now prove (2). By symmetry, it is enough to consider arbitrary dyadic $M,N$.
	Set
	\[
	u(t)=S_{\alpha}(t)\mathbf{P}_Mf,\qquad
	v(t)=S_{\alpha}(t)\mathbf{P}_Ng .
	\]
	For each $a\in\Z$, write the $a$-th Fourier coefficient in $x$ of $uv$ as
	\[
	G_a(t)
	=
	\sum_{k\in\Z}
	\widehat{\mathbf{P}_Mf}(k)
	\widehat{\mathbf{P}_Ng}(a-k)
	e^{it\varphi_a(k)},
	\qquad
	\varphi_a(k)=|k|^{\alpha}+|a-k|^{\alpha}.
	\]
	Then
	\[
	\|uv\|_{L^2(I;L^2(\T))}^2
	\lesssim
	\sum_{a\in\Z}\|G_a\|_{L^2(I)}^2 .
	\]
	Arguing exactly as in the proof of Lemma~\ref{Reduction-1}, using an almost
	orthogonality decomposition in the time frequency, we obtain
	\[
	\|G_a\|_{L^2(I)}^2
	\lesssim
	\sum_{l\in\Z}
	\left(
	\sum_{k}
	\mathbf{1}_{A_{a,l,M,N}(1/2) }(k)\left|
	\widehat{\mathbf{P}_Mf}(k)
	\widehat{\mathbf{P}_Ng}(a-k)
	\right|
	\right)^2 .
	\]
	Hence, by Cauchy--Schwarz,
	\[
	\begin{split}
		\|uv\|_{L^2(I;L^2(\T))}^2
		&\lesssim
		\sum_{a,l}
		\#A_{a,l,M,N}\Big(\frac{1}{2}\Big)
		\sum_{k\in A_{a,l,M,N}(1/2)}
		\left|
		\widehat{\mathbf{P}_Mf}(k)
		\widehat{\mathbf{P}_Ng}(a-k)
		\right|^2\\
		&\lesssim
		\sup_{a,l}\#A_{a,l,M,N}\Big(\frac{1}{2}\Big)
		\sum_{a,l}\sum_{k\in A_{a,l,M,N}(1/2)}
		\left|
		\widehat{\mathbf{P}_Mf}(k)
		\widehat{\mathbf{P}_Ng}(a-k)
		\right|^2 .
	\end{split}
	\]
	For fixed $a$ and $k$, there are only $O(1)$ integers $l$ such that
	$k\in A_{a,l,M,N}(1/2)$. Therefore
	\[
	\sum_{a,l}\sum_{k\in A_{a,l,M,N}(1/2)}
	\left|
	\widehat{\mathbf{P}_Mf}(k)
	\widehat{\mathbf{P}_Ng}(a-k)
	\right|^2
	\lesssim
	\|\mathbf{P}_Mf\|_{L^2(\T)}^2
	\|\mathbf{P}_Ng\|_{L^2(\T)}^2 .
	\]
	Using Proposition~\ref{counting-1}, we get
	\[
	\sup_{a,l}\#A_{a,l,M,N}(1/2)
	\lesssim
	\min\{M,N\}^{1-\frac{\alpha}{2}}.
	\]
	 Consequently,
	\[
	\|uv\|_{L^2(I;L^2(\T))}^2
	\lesssim
	\min\{M,N\}^{1-\frac{\alpha}{2}}
	\|\mathbf{P}_Mf\|_{L^2(\T)}^2
	\|\mathbf{P}_Ng\|_{L^2(\T)}^2,
	\]
	which gives (2) after taking the square root.
	This completes the proof of Corollary~\ref{Strichartz-bilinear}.
\end{proof}

%%%%%%%%%%%%%%%%%%%%%%%

%%%%
\begin{proposition}\label{bilinear-spacetime}
Let $1<\alpha\leq 2$.	For $u_1,u_2\in L^2(\R\times \T)$ such that 
	$$ \widehat{u}_j(\tau,k)=\mathbf{1}_{K_j\leq | \tau-|k|^{\alpha}| < 2K_j}\mathbf{1}_{N_j\leq |k|<2N_j}\widehat{u}_j(\tau,k),\quad j=1,2,
	$$
	we have the estimate 
	\begin{align}\label{bilinearXsb}
	\|u_1u_2\|_{L^2}\notag  \lesssim \min (N_1,N_2)^{\frac{1}{2}-\frac{\alpha}{4}}\cdot
	 \min(K_1,K_2)^{1/2} \max(K_1,K_2)^{1/4} \|u_1\|_{L^2} \cdot \|u_2\|_{L^2}.
	\end{align}
\end{proposition}
\begin{proof}
	By duality, it is sufficient to show that for any $v\in L^2(\R\times \T)$, $\|v\|_{L^2}=1$, we have
	\begin{equation}\label{Prop-2.7}
	\Big|\int_{\R\times \T}u_1u_2vdxdt\Big|\leq \min(N_1,N_2)^{\frac{1}{2}-\frac{\alpha}{4}}\cdot 
	 \min(K_1,K_2)^{1/2} \max(K_1,K_2)^{1/4}
	\|u_1\|_{L^2}\|u_2\|_{L^2}.
	\end{equation}
	The left hand-side of \eqref{Prop-2.7} can be written as 
	\begin{equation}\label{Prop2.7-0}
	\Big|\int_{\tau_1+\tau_2+\tau_3=0}\sum_{k_1+k_2+k_3=0}\widehat{u_1}(\tau_1,k_1)\widehat{u_2}(\tau_2,k_2)\widehat{v}(\tau_3,k_3)\Big|.
	\end{equation} 
	By the Cauchy-Schwarz inequality, \eqref{Prop2.7-0} can be bounded by 
	\begin{equation*}\label{Prop2.7-1}
	\begin{split}
	&\|\widehat{u_1}\|_{L_{\tau,k}^2}\|\widehat{u_2}\|_{L_{\tau,k}^2} \|\widehat{v}\|_{L_{\tau,k}^2}\cdot \sup_{(\tau_3,k_3)}(\mathrm{mes}(A(\tau_3,k_3)))^{1/2},
	\end{split}
	\end{equation*}
	where
	\begin{equation*}
	\begin{split}
	A(\tau_3,k_3)=&\{(\tau_1,k_1):K_1\leq |\tau_1-|k_1|^{\alpha}|<2K_1,  K_2\leq |\tau_3+\tau_1+|k_3+k_1|^{\alpha} |< 2K_2  \}\\
	\cap & 
	\{(\tau_1,k_1):  N_1\leq |k_1|<2N_1, N_2\leq |k_3+k_1|<2N_2 \}.
	\end{split}
	\end{equation*}
	Eliminating $\tau_1$, we can write 
	$
	A(\tau_3,k_3)\leq \min(K_1,K_2)\#B(k_3),
	$
	where
	\begin{equation*}
	\begin{split}
	B(\tau_3,k_3)=&\{k_1: N_1\leq|k_1|<2N_1, N_2\leq |k_3+k_1|<2N_2\}\\
	\cap &\{k_1: |\tau_3+|k_1|^{\alpha}+|k_3+k_1|^{\alpha}|\lesssim\max(K_1,K_2)  \}.
	\end{split}
	\end{equation*}
	Applying Proposition~\ref{counting-1}, we have $\# B(\tau_3,k_3)\lesssim \min(N_1,N_2)^{1-\frac{\alpha}{2}}\max(K_1,K_2)^{1/2}. $
	Therefore,
	$$ \mathrm{mes}(A(\tau_3,k_3))^{1/2}\leq \min(K_1,K_2)^{1/2}\max(K_1,K_2)^{1/4}\cdot \min(N_1,N_2)^{\frac{1}{2}-\frac{\alpha}{4}}
	$$
	and we obtain \eqref{Prop-2.7}. This completes the proof of Proposition~\ref{bilinear-spacetime}.
\end{proof}	
%%%%%%%%%%%%%%%%%%%%%%%%%%
\begin{corollaire}\label{bilinear-Xsb}
	Let $1<\alpha\leq 2$ and set
	\[
	s_0=\frac{1}{2}-\frac{\alpha}{4}.
	\]
	For any $s\geq s_0$ and dyadic numbers $N\gg M$, we have
	\begin{equation*}
		\begin{split}
			& (1)\quad
			\|\mathbf{P}_Nf\|_{L_{t,x}^4}
			\lesssim
			N^{\frac{s}{2}}\|\mathbf{P}_Nf\|_{X^{0,\frac{3}{8}}}\,,\\
			& (2)\quad
			\|\mathbf{P}_Nf\cdot \mathbf{P}_Mg\|_{L_{t,x}^2}
			\lesssim
			M^{s}\|\mathbf{P}_Nf\|_{X^{0,\frac{3}{8}}}
			\|\mathbf{P}_Mg\|_{X^{0,\frac{3}{8}}}\,.
		\end{split}
	\end{equation*}
	Moreover, for any $2<q<\infty$ and any $b>\frac{1}{q}$, one has the mixed estimate
	\begin{equation}\label{mixed-typeI-typeII}
		\|\mathbf{P}_Nf\cdot \mathbf{P}_Mg\|_{L_{t,x}^2}
		\lesssim_{q,b}
		\|\mathbf{P}_Nf\|_{X^{0,b}}
		\|\mathbf{P}_Mg\|_{L_t^qL_x^\infty},
	\end{equation}
	and the same estimate holds with the two factors interchanged.
\end{corollaire}

\begin{proof}
	We first prove (1). Let $u=\mathbf{P}_Nf$. We decompose $u$ dyadically in
	modulation:
	\[
	u=\sum_{K\geq 1}u_K,
	\qquad
	u_K=Q_Ku,
	\]
	where $Q_K$ is defined by
	$$ \widehat{u_K}(\tau,n)=\mathbf{1}_{\langle\tau-|n|^{\alpha}\rangle\sim K }\widehat{u}(\tau,n),
	$$
		with the usual convention that $K=1$ denotes the region
	$\langle \tau-|n|^\alpha\rangle\lesssim 1$. 
	Write
	\[
	\|u\|_{L^4_{t,x}}^2=\|u^2\|_{L^2_{t,x}},
	\]
	we use Proposition~\ref{bilinear-spacetime} and obtain
	\[
	\begin{split}
		\|u\|_{L^4_{t,x}}^2
		&\leq
		\sum_{K_1,K_2}
		\|u_{K_1}u_{K_2}\|_{L^2_{t,x}}\\
		&\lesssim
		N^{s_0}
		\sum_{K_1,K_2}
		\min(K_1,K_2)^{1/2}
		\max(K_1,K_2)^{1/4}
		\|u_{K_1}\|_{L^2_{t,x}}
		\|u_{K_2}\|_{L^2_{t,x}}.
	\end{split}
	\]
	Set
	\[
	a_K=K^{\frac38}\|u_K\|_{L^2_{t,x}}.
	\]
	If $K_1\leq K_2$, then
	\[
	K_1^{1/2}K_2^{1/4}
	\|u_{K_1}\|_{L^2}
	\|u_{K_2}\|_{L^2}
	=
	\left(\frac{K_1}{K_2}\right)^{1/8}
	a_{K_1}a_{K_2}.
	\]
	The same bound with $K_1$ and $K_2$ exchanged holds when $K_2\leq K_1$.
	Thus
	\[
	\|u\|_{L^4_{t,x}}^2
	\lesssim
	N^{s_0}
	\sum_{K_1,K_2}
	\left(
	\frac{\min(K_1,K_2)}{\max(K_1,K_2)}
	\right)^{1/8}
	a_{K_1}a_{K_2}.
	\]
	 Schur's test therefore gives
	\[
	\sum_{K_1,K_2}
	\left(
	\frac{\min(K_1,K_2)}{\max(K_1,K_2)}
	\right)^{1/8}
	a_{K_1}a_{K_2}
	\lesssim
	\sum_K a_K^2.
	\]
	Consequently,
	\[
	\|u\|_{L^4_{t,x}}^2
	\lesssim
	N^{s_0}
	\sum_KK^{\frac34}\|u_K\|_{L^2_{t,x}}^2
	\lesssim
	N^{s_0}\|u\|_{X^{0,\frac38}}^2.
	\]
	Taking the square root gives
	\[
	\|\mathbf{P}_Nf\|_{L^4_{t,x}}
	\lesssim
	N^{\frac{s_0}{2}}\|\mathbf{P}_Nf\|_{X^{0,\frac38}}.
	\]
	Since $s\geq s_0$, this proves (1).
	
	The proof of (2) follows from the same argument using (2) of Corollary \ref{Strichartz-bilinear}, and we omit the detail.

	It remains to prove \eqref{mixed-typeI-typeII}. Let
	\[
	p=\frac{2q}{q-2},
	\qquad
	\frac1p+\frac1q=\frac12.
	\]
	By H\"older's inequality,
	\[
	\|\mathbf{P}_Nf\cdot \mathbf{P}_Mg\|_{L^2_{t,x}}
	\leq
	\|\mathbf{P}_Nf\|_{L_t^pL_x^2}
	\|\mathbf{P}_Mg\|_{L_t^qL_x^\infty}.
	\]
	Since $b>\frac1q=\frac12-\frac1p$, by Minkowski (since $p\geq 2$) and the one-dimensional Sobolev embedding in
	time gives
	\begin{align*}
	\|\mathbf{P}_Nf\|_{L_t^pL_x^2}
	=
	&\|S_\alpha(-t)\mathbf{P}_Nf\|_{L_t^pL_x^2}\\
	\leq &\|S_{\alpha}(-t)\mathbf{P}_Nf\|_{L_x^2L_t^p}\\ 
	\lesssim
	&\|S_\alpha(-t)\mathbf{P}_Nf\|_{L_x^2H_t^b}
	=
	\|\mathbf{P}_Nf\|_{X^{0,b}}.
		\end{align*}
	This proves \eqref{mixed-typeI-typeII}. The proof of corollary \ref{bilinear-Xsb} is now complete.

\end{proof}
%%%%%%%%%%%%%%%%%%%%%%%%

Another consequence of Proposition \ref{bilinear-spacetime} is the following trilinear $X^{s,b}$ estimate, which yields the deterministic local well-posedness result in \cite{Cho}.
\begin{corollaire}\label{Trilinear}
Let $1<\alpha\leq 2$.	For $s\geq \frac{1}{2}-\frac{\alpha}{4}$, $0<\epsilon\ll1$, we have
	$$ \|u_1\ov{u_2}u_3\|_{X^{s,-\frac{1}{2}+\epsilon}}\lesssim \|u_1\|_{X^{s,\frac{3}{8}}}\|u_2\|_{X^{s,\frac{3}{8}}}\|u_3\|_{X^{s,\frac{3}{8}}}.
	$$	
	\end{corollaire}
%%%%
\subsection{Probabilistic estimates}
We present two probabilistic lemmas related to the Gaussian random variables. Recall that 
$(g_n)_{n\in\Z}$ denotes a family of independent standard complex-valued Gaussian random variables on a probability space $(\Omega,\mathcal{F},\mathbb{P})$. 
\begin{lemme}[Wiener chaos estimates]\label{Wiener-Chaos}
Let $c:\Z^k\rightarrow\C$. Set
$$ 
S(\omega)=\sum_{(n_1,\cdots,n_k)\in \Z^k}c(n_1,\cdots, n_k)g_{n_1}(\omega)\cdots g_{n_k}(\omega).	
$$
Suppose that $S\in L^2(\Omega)$.  Then there is a constant $C_k$ such that for every $p\geq 2$,
$$
\|S\|_{L^p(\Omega)}\leq C_k\, p^{\frac{k}{2}}\|S\|_{L^2(\Omega)}.
$$
 \end{lemme}
For a proof of Lemma~\ref{Wiener-Chaos}, we refer to \cite{Simon}.
%%%%
\begin{lemme}[Probabilistic Strichartz estimate]\label{proba-Strichartz}
	Let 
	$$
	f^{\omega}(t,x)=\sum_{n\in\Z} c_ng_{n}(\omega)e^{i(nx-[n]^{\alpha}t)}\,.
	$$ 
	Then for $2\leq q<\infty$, there exists $T_0<0$ and $c>0$ such that for all $T\leq T_0$, $R>0$
	\begin{equation*}\label{eq:proba-Strichartz}
	\mathbb{P}\{\omega: \|f^{\omega}\|_{L^q([-T,T]\times\T)}>R\|c_n\|_{l_n^2} \}\leq \exp(-cT^{-\frac{2}{q}}R^2).
	\end{equation*}	
\end{lemme}
\begin{proof}
	We can assume that $\|c_n\|_{l^2}=1$. By  Lemma~\ref{Wiener-Chaos}, there exists $C_0>0$, independent of $(c_n)_{n\in\Z}$, such that
	$$
	\big\|\sum_{n\in\Z} c_n\, g_{n}(\omega)\big\|_{L^r(\Omega)}\leq C_0\sqrt{r}\,,
	$$ 
	for every $r\geq 2$. 
	Therefore, for $r\geq q$, by the Minkowski inequality, we have
	$$
	(\mathbb{E}[\|f^{\omega}\|_{L^q([-T,T]\times \T)}^r])^{\frac{1}{r}}\leq C_1\sqrt{r}T^{\frac{1}{q}}\,.
	$$
	Then by Chebyshev's inequality, we have
	$$ 
	\mathbb{P}\{\omega:\|f^{\omega}\|_{L^q([-T,T]\times \T)}>R \}\leq C_1^r R^{-r} r^{\frac{r}{2}} T^{\frac{r}{q}}.
	$$
	By taking $r=R^2C_1^{-2}e^{-2}T^{-\frac{2}{q}}$, we obtain   
	$$
	C_1^rR^{-r}r^{\frac{r}{2}}T^{\frac{r}{q}}=e^{-R^2/((eC_1)^2\,T^{2/q})}=e^{-cT^{-\frac{2}{q}}R^2}
	$$
	with $c=(eC_1)^{-2}$. This completes the proof of Lemma \ref{proba-Strichartz}.
\end{proof}
%%%%%%%%%%%%%%%%%%%%%%%%%%%%%%%%%%%%%%%%%%%%%%%%%%%%%%%%%%%%%%%%%%%%%%%%%%%%%%%%%%%%%%%%%%%%%%%%%%%%%%%%%%%%%%%%%%%%%%
\section{Local well posedness for $\frac{6}{5}<\alpha<2$}
In this section, we prove a local well-posedness result for in the case $\frac{6}{5}<\alpha<2$. We remark that if $\alpha>\frac{4}{3}$, then $\frac{\alpha-1}{2}>\frac{1}{2}-\frac{\alpha}{4}$, and the deterministic local well-posedness of the cubic FNLS applies. Hence we will only focus on the case $\frac{6}{5}<\alpha\leq\frac{4}{3}$, where additional arguments are needed. 
\\

Introducing the gauge transform
 $$  
 v(t,x)=u(t,x)e^{\frac{it}{\pi}\int_{\T}|u|^2},
 $$
the FNLS \eqref{main-NLS} is transformed to the Wick-ordered FNLS
\begin{equation}\label{WiFNLS}
i\partial_tv+|D_x|^{\alpha}v+\big(|v|^2-\frac{1}{\pi}\int_{\T}|v|^2dx \big)v=0
\end{equation} 
with the same initial data as $u$. The flow of \eqref{WiFNLS}, if exists, will be denoted by $\Psi(t)$. We also denote by $\Psi_N(t)$ the flow map of the truncated Wick-ordered FNLS
\begin{equation}\label{WiFNLS-N}
i\partial_tv_N+|D_x|^{\alpha}v_N+\Pi_N\Big(\big(|\Pi_Nv_N|^2-\frac{1}{\pi}\int_{\T}|\Pi_Nv_N|^2dx \big)\Pi_Nv_N\Big)=0.
\end{equation}
By inverting the gauge transformation, $$u_N(t,x):=e^{-\frac{it}{\pi}\int_{\T}|\Pi_Nv_N|^2dx}\Pi_Nv_N(t,x)+\Pi_N^{\perp}v_N(t,x)
$$
satisfies the truncated FNLS
$$ i\partial_tu_N+|D_x|^{\alpha}u_N+\Pi_N\big(|\Pi_Nu_N|^2\Pi_Nu_N\big)=0,
$$
with the same initial data as $v_N$. 
Though the Wick-ordered FNLS (truncated or not) is equivalent to the original FNLS in our setting, 
it turns out that the use of the gauge transformation removes trivial resonances, which improves the regularity at multi-linear level.
\\

 The Wick-ordered nonlinearity can be written as
$$ \mathcal{N}(v):=\big(|v|^2-\frac{1}{\pi}\int_{\T}|v|^2 \big)v.
$$
More generally, $\mathcal{N}(v)$ can be written as the trilinear form $$\mathcal{N}(v,v,v):=\mathcal{N}_1(v,v,v)-\mathcal{N}_0(v,v,v),
$$ 
where the trilinear forms $\mathcal{N}_1(\cdot,\cdot,\cdot)$ and $\mathcal{N}_0(\cdot,\cdot,\cdot)$ are defined as
\begin{equation}\label{wicknonlinear}
\begin{split}
& \mathcal{N}_0(f_1,f_2,f_3):=\sum_{n\in\Z}\widehat{f_1}(n)\ov{\widehat{f_2}}(n)\widehat{f_3}(n)e^{inx},\\
&\mathcal{N}_1(f_1,f_2,f_3):=\sum_{n_2\neq n_1,n_3}\widehat{f_1}(n_1)\ov{\widehat{f_2}}(n_2)\widehat{f_3}(n_3)e^{i(n_1-n_2+n_3)x}.
\end{split}
\end{equation}
Here and in the sequel, $n_2\neq n_1,n_3$ means that $n_2\neq n_1$ and $n_2\neq n_3$.
\\

The resolution of \eqref{WiFNLS} and \eqref{WiFNLS-N} will be achieved by writing
$$ v(t)=S_{\alpha}(t)\phi+w(t),
$$
where the nonlinear part $w$ is pretended to be smoother, and it satisfies the integral equation
$$ w(t)=-i\int_0^t S_{\alpha}(t-t')\mathcal{N}\big(S_{\alpha}(t')v_0+w(t')\big)dt'.
$$
In order to formulate our local existence result, we need to introduce several quantities. 
First, we take $\chi_0\in C_c^{\infty}(-2,2)$, $\chi_0(t)=1$ for $|t|\leq 1$, such that
$$ \sum_{l\in\Z}\chi_0(t-l)=1,\quad \forall\, t\in\R.
$$
Define 
\begin{equation}\label{auxillarynorms}
\begin{split}
%& \|\phi\|_{\mathcal{Z}^{q,\epsilon}}:=\|\phi\|_{\mathcal{F}L^{\frac{\alpha}{2}-\epsilon,\frac{2}{\epsilon}}}+\|\chi_0(t)S_{\alpha}(t)\phi\|_{L_t^qW_x^{\frac{\alpha-1}{2}-\epsilon,\frac{1}{\epsilon} }} \,\, ,\\
& \mathcal{W}_{s,\epsilon}(\phi):=\sum_{l\in\Z}\langle l\rangle^{-2}\big\|\chi_0(t)\mathcal{N}(S_{\alpha}(t+l)\phi ) \big\|_{X^{s,-\frac{1}{2}+2\epsilon}}\,\, ,
\\
&\|\phi\|_{\mathcal{V}^{q,\epsilon}}:=\|\phi\|_{\mathcal{F}L^{\frac{\alpha}{2}-\frac{2\epsilon}{3},\frac{2}{\epsilon}}}+\sum_{l\in\Z}\langle l \rangle^{-2}\big\|\chi_0(t)S_{\alpha}(t+l)\phi\big\|_{L_t^qW_x^{\frac{\alpha-1}{2}-\frac{\epsilon}{2},\frac{1}{\epsilon}}} \,\, ,\\
&\|\phi\|_{\widetilde{\mathcal{V}}^{q,\epsilon}}:=\|\phi\|_{\mathcal{F}L^{\frac{\alpha}{2}-\epsilon,\frac{2}{\epsilon}}}+\sum_{l\in\Z}\langle l \rangle^{-2}\big\|\chi_0(t)S_{\alpha}(t+l)\phi\big\|_{L_t^qW_x^{\frac{\alpha-1}{2}-\epsilon,\frac{1}{\epsilon}}}\,\,.
\end{split}
\end{equation}
 The Fourier-Lebesgue norm if defined by $\|f\|_{\mathcal{F}L^{s,r}}:=\big\|\langle n\rangle^s\widehat{f}(n)\big\|_{l^r}$. We denote by $\mathcal{V}^{q,\epsilon}$ the functions with finite $\mathcal{V}^{q,\epsilon}$ norm and $\mathcal{W}^{s,\epsilon}$ the measurable subset of $H^{\frac{\alpha-1}{2}-\epsilon}$ where the functions have finite $\mathcal{W}_{s,\epsilon}$ quantity. Obviously, $\mathcal{V}^{q,\epsilon}\hookrightarrow \widetilde{\mathcal{V}}^{q,\epsilon}$, hence the auxiliary norm $\mathcal{\widetilde{V}}^{q,\epsilon}$ is weaker.  We remark that $\mathcal{W}_{s,\epsilon}(\cdot)$ is not a norm%, and the normed space $\mathcal{V}^{q,\epsilon}$ is not complete 
 . Since the partial sum $\Pi_N$ is uniformly bounded in $L^p(\T)$ for $1<p<\infty$, we have the following statement. 
\begin{lemme}\label{Dirichletkernel}
There exists a uniform constant $A_0\geq 1$, such that for all $N\in\N$,
$$ \|\Pi_N\|_{\mathcal{V}^{q,\epsilon}\rightarrow \mathcal{V}^{q,\epsilon}}\leq A_0, \quad \|\Pi_N^{\perp}\|_{\mathcal{V}^{q,\epsilon}\rightarrow \mathcal{V}^{q,\epsilon}}\leq A_0.
$$ 	
\end{lemme}
%%%%%%

%%%%%
\begin{proposition}\label{LWP-main}
Assume that $\frac{6}{5}<\alpha<2$, $2\ll q<\infty$ is large enough and $0<\epsilon\ll 1$ is small enough. Let $N\in\mathbb{N}\cup\{\infty \}$, $s\in\big[\frac{1}{2}-\frac{\alpha}{4},\alpha-1 \big)$. There exist $c>0, \kappa>0$, independent of $N$ such that the following holds true. The Cauchy problem \footnote{By convention, $\Pi_{\infty}=\mathrm{Id}$. } \eqref{WiFNLS-N} with initial data $v_N(0)=\phi_N+r_N$ is locally well-posed for data $r_N\in H^s(\T)$ and $\phi_N$ in some suitable set. More precisely, for every $R\geq 1$, if
$$ \big(\mathcal{W}_{s,\epsilon}(\phi_N)\big)^{\frac{1}{3}}+\|\phi_N\|_{\mathcal{V}^{q,\epsilon}}\leq R \textrm{ and } \|r_N\|_{H^s(\T)}\leq R,
$$
there is a unique solution of \eqref{WiFNLS-N} in the class
$$ S_{\alpha}(t)(\phi_N+r_N)+X_{\tau_R}^{s,\frac{1}{2}+2\epsilon} \text{ on } [-\tau_R,\tau_R] \text{ where } \tau_R=cR^{-\kappa}.
$$
In particular, the solution can be written as
$ v_N(t)=S_{\alpha}(t)(\phi_N+r_N)+w_N(t),
$
with
$$ \|w_N\|_{X_{\tau_R}^{s,\frac{1}{2}+2\epsilon}}\leq R^{-1}.
$$
\end{proposition}
%%%%%%%%%
By inverting the gauge transformation, we obtain the local existence for the flow $\Phi_N(t)$ as well as $\Phi(t)$. 
Note that even the global existence of $\Phi_N(t)$ is not an issue, the important point 
in Proposition~\ref{LWP-main} are the uniform in $N$ bounds.  It is standard that $\rho_N$ 
is invariant under $\Phi_N(t)$ thanks to the Liouville theorem for divergence free vector fields and the invariance of complex gaussians under rotations.
\\

Furthermore, we have a more general local convergence result, which will be useful in the construction of the global dynamics. For $R>0$, we introduce the notation
$$
 \mathcal{B}_R:=\big\{\phi\in H^{\frac{\alpha-1}{2}-\epsilon}(\T): \big(\mathcal{W}_{s,\epsilon}(\phi_N)\big)^{\frac{1}{3}}+ \|\phi\|_{\mathcal{V}^{q,\epsilon}}\leq R \big\}.
$$
%%%%%%
\begin{proposition}\label{local-convergence}
Assume that $R\geq 1$ and $\alpha,q,\epsilon$ are the numerical constants as in Proposition \ref{LWP-main}. Let $(\phi_{0,k})\subset \mathcal{B}_R$, $\phi_0\in \mathcal{B}_R$. Assume that $(r_{0,k})\subset H^s(\T)$ satisfying $\|r_{0,k}\|_{H^s(\T)}\leq 2R$. Let $N_k\rightarrow\infty$ be a subsequence of $\mathbb{N}$. Assume moreover that
$$ \lim_{k\rightarrow\infty}\mathcal{W}_{s,\epsilon}(\phi_{0,k}-\phi_0)=0,\quad \lim_{k\rightarrow\infty}\|r_{0,k}-r_0\|_{H^s(\T)}=0.
$$
Then there exist $c>0, \kappa>0$, such that on $[-T_R,T_R]$ with $T_R=cR^{-\kappa}$, we have
\begin{equation*}
\begin{split}
& \Phi_{N_k}(t)(\phi_{0,k}+r_{0,k})=e^{\frac{it}{\pi}\|\Pi_{N_k}(\phi_{0,k}+r_{0,k} )\|_{L^2(\T)}^2 } \big(\Pi_{N_k}S_{\alpha}(t)(\phi_{0,k}+r_{0,k})+w_k(t) \big)+\Pi_{N_k}^{\perp}S_{\alpha}(t)\phi_{0,k},\\
& \Phi(t)(\phi_0+r_0)=e^{\frac{it}{\pi}\|\phi_0+r_0\|_{L^2(\T)}^2 }\big(S_{\alpha}(t)(\phi_0+r_0)+w(t) \big).
\end{split}
\end{equation*}
Furthermore, 
$$ \lim_{k\rightarrow\infty}\|w_k-w \|_{X_{T_R}^{s,\frac{1}{2}+2\epsilon}}=0,\quad \text{and in particular},\quad  \lim_{k\rightarrow\infty}\sup_{|t|\leq T_R}\|w_k(t)-w(t)\|_{H^s(\T)}=0.
$$
\end{proposition}
%%%%%%%%%
The proof of Proposition \ref{LWP-main} and Proposition \ref{local-convergence} depends on the following deterministic multilinear estimate. Let $\eta\in C_c^{\infty}((-1,1))$ and $\eta_T(t)=\eta\big(\frac{t}{T}\big)$. 
\begin{proposition}\label{multi-linear}
Let $\alpha\in \big(\frac{6}{5}, 2 \big)$ and $s\in\big[\frac{1}{2}-\frac{\alpha}{4},\alpha-1 \big)$. There exist $2\ll q<\infty$, large enough, $0<\epsilon\ll 1$, small enough and $\theta=\theta(\epsilon,q)>0$, such that for all $0<T<1$, $f_1,f_2,f_3\in \mathcal{Z}^{q,\epsilon}$ and $u_1,u_2,u_3\in X^{s,\frac{1}{2}+\epsilon}$, the following estimates hold:
\begin{equation*}
\begin{split}
&(\mathrm{1})\quad  \|\eta_T(t)\mathcal{N}(S_{\alpha}(t)f_1,u_2,u_3)\|_{X^{s,-\frac{1}{2}+2\epsilon}}\lesssim T^{\theta}\|f_1\|_{\mathcal{Z}^{q,\epsilon}}\|u_2\|_{X^{s,\frac{1}{2}+\epsilon}}\|u_3\|_{X^{s,\frac{1}{2}+\epsilon}}\,\, ,
\\
&(\mathrm{2})\quad  \|\eta_T(t)\mathcal{N}(u_1,S_{\alpha}(t)f_2,u_3)\|_{X^{s,-\frac{1}{2}+2\epsilon}}\lesssim T^{\theta}\|u_1\|_{X^{s,\frac{1}{2}+\epsilon }}\|f_2\|_{\mathcal{Z}^{q,\epsilon}}\|u_3\|_{X^{s,\frac{1}{2}+2\epsilon}}\,\, ,\\
&(\mathrm{3})\quad  \|\eta_T(t)\mathcal{N}(u_1,u_2,S_{\alpha}(t)f_3)\|_{X^{s,-\frac{1}{2}+2\epsilon}}\lesssim T^{\theta}\|u_1\|_{X^{s,\frac{1}{2}+2\epsilon}}\|u_2\|_{X^{s,\frac{1}{2}+\epsilon}}\|f_3\|_{\mathcal{Z}^{q,\epsilon}}\,\, ,\\
&(\mathrm{4})\quad  \|\eta_T(t)\mathcal{N}(S_{\alpha}(t)f_1,u_2,S_{\alpha}(t)f_3)\|_{X^{s,-\frac{1}{2}+2\epsilon}}\lesssim T^{\theta}\|f_1\|_{\mathcal{Z}^{q,\epsilon}}\|u_2\|_{X^{s,\frac{1}{2}+\epsilon}}\|f_3\|_{\mathcal{Z}^{q,\epsilon}}\,\, ,\\
&(\mathrm{5})\quad  \|\eta_T(t)\mathcal{N}(S_{\alpha}(t)f_1,S_{\alpha}(t)f_2,u_3)\|_{X^{s,-\frac{1}{2}+2\epsilon}}\lesssim T^{\theta}\|f_1\|_{\mathcal{Z}^{q,\epsilon}}\|f_2\|_{\mathcal{Z}^{q,\epsilon}}\|u_3\|_{X^{s,\frac{1}{2}+2\epsilon}}\,\, ,\\
&(\mathrm{6})\quad  \|\eta_T(t)\mathcal{N}(u_1,S_{\alpha}(t)f_2,S_{\alpha}(t)f_3)\|_{X^{s,-\frac{1}{2}+2\epsilon}}\lesssim T^{\theta}\|u_1\|_{X^{s,\frac{1}{2}+\epsilon}}\|f_2\|_{\mathcal{Z}^{q,\epsilon}}\|f_3\|_{\mathcal{Z}^{q,\epsilon}}\,\, .
\end{split}
\end{equation*}
\end{proposition}
We will postpone the proof of Proposition \ref{multi-linear} to the next section and use it to prove the local existence results, Proposition \ref{LWP-main} and Proposition \ref{local-convergence}, in the rest of this section.
\\

\begin{proof}[Proof of Proposition \ref{LWP-main}]
For simplicity, we drop the subindex $N$ everywhere. Consider the mapping
\begin{equation*}
\Gamma: w(t)\rightarrow -i\int_0^tS_{\alpha}(t-t')\mathcal{N}\big(\eta_T(t')\big(S_{\alpha}(t')(\phi+r )+ w(t')\big) \big)dt',
\end{equation*}
and we want to show that $\Gamma$ is a contraction on a ball of $X_{T}^{s,\frac{1}{2}+\epsilon}$. For given $u$ on $[-T,T]\times\T$, we denote by $\widetilde{u}$ an extension of $u$ onto $\R\times\T$. Note that from Lemma \ref{inhomo-linear}, we deduce that
\begin{equation*}
\Big\|\int_0^t S_{\alpha}(t-t')\mathcal{N}(\widetilde{u}(t'))  dt'\Big\|_{X_{T}^{s,\frac{1}{2}+2\epsilon}}\lesssim \|\eta_T(t)\mathcal{N}(\widetilde{u})\|_{X^{s,-\frac{1}{2}+2\epsilon}},
\end{equation*}
where $\eta_T(t)=\eta(t/T)$ is a smooth cutoff on $[-2T,2T]$, $\eta_T(t)=1$ for $t\in[-T,T]$. Take $\widetilde{w}$ an extension of $w$ on $\R\times\T$ with the property $\widetilde{w}(t)=w(t)$ for $t\in[-T,T]$. For $\widetilde{u}(t)=\eta_T(t)\big(S_{\alpha}(t)\phi+S_{\alpha}(t)r+\widetilde{w}(t) \big)$, from Proposition \ref{multi-linear}, we have
\begin{equation*}
\|\mathcal{N}(\widetilde{u})\|_{X^{s,-\frac{1}{2}+2\epsilon}}\lesssim T^{\theta}\Big(\|\phi\|_{\mathcal{V}^{q,\epsilon}}^3+\mathcal{W}_{s,\epsilon}(\phi) +\|\eta_T(t)(S_{\alpha}(t)r+\widetilde{w})\|_{X^{s,\frac{1}{2}+\epsilon}}^3 \Big).
\end{equation*}
This implies that
\begin{equation*}
\|\Gamma(w)\|_{X_T^{s,\frac{1}{2}+2\epsilon}}\lesssim T^{\theta} \Big(\|\phi\|_{\mathcal{V}^{q,\epsilon}}^3+ \|r\|_{H_x^s}^3+\mathcal{W}_{s,\epsilon}(\phi)+\|w\|_{X_T^{s,\frac{1}{2}+\epsilon}}^3 \Big).
\end{equation*}
Moreover, if $w_1,w_2\in X_T^{s,\frac{1}{2}+\epsilon},$ the same argument, after doing simple algebraic manipulations, yields
\begin{equation*}
\begin{split}
&\|\Gamma(w_1)-\Gamma(w_2)\|_{X_T^{s,\frac{1}{2}+2\epsilon}}\\ \lesssim &T^{\theta}\Big(\|\phi\|_{\mathcal{V}^{q,\epsilon}}^2+\mathcal{W}_{s,\epsilon}(\phi)+\|r\|_{H_x^s}^3+\|w_1\|_{X_T^{s,\frac{1}{2}\epsilon}}^2+\|w_2\|_{X_T^{s,\frac{1}{2}+\epsilon}}^3 \Big)\|w_1-w_2\|_{X_T^{s,\frac{1}{2}+\epsilon}}.
\end{split}
\end{equation*}
Hence $\Gamma$ is a contraction in the ball $B_{X_T^{s,\frac{1}{2}+\epsilon}}(R^{-1})$, provided that
$$ \|\phi\|_{\mathcal{V}^{q,\epsilon}}+\big(\mathcal{W}_{s,\epsilon}(\phi_N)\big)^{\frac{1}{3}}\leq R, \quad T\leq T_R:= cR^{-\kappa},
$$
with $c>0$ small enough and $\kappa>0$ large enough. This proves the existence and uniqueness of $w_N(t)$ for all $N\in\mathbb{N}\cup\{\infty\}$.  This completes the proof of Proposition \ref{LWP-main}.
\end{proof}
%%%%
\begin{proof}[Proof of Proposition \ref{local-convergence}]
	 To simplify the notation, we denote by $ z(t)=\eta_T(t)S_{\alpha}(t)\phi_0, z_{k}(t)=\eta_T(t)S_{\alpha}(t)\phi_{0,k}$, and $y(t)=\eta_T(t)S_{\alpha}(t)r_0, y_k(t)=\eta_T(t)S_{\alpha}(t)r_{0,k}$.
	  By inverting the gauge transformation, for $t$ belonging to the  time interval of local existence theory, we have
$$
w_k(t)=-i\Pi_{N_k}\int_0^t S_{\alpha}(t)\mathcal{N}\big(z_k+y_k+w_k \big)(t')dt',  
$$
and
$$
w(t)=-i\int_0^tS_{\alpha}(t-t')\mathcal{N}\big( z+y+w \big)(t')dt'.
$$
Taking the difference, we get
	\begin{equation*}
	\begin{split}
	\|w_k-w\|_{X_T^{s,\frac{1}{2}+2\epsilon}}\leq & \Big\|\Pi_{N_k}^{\perp}\int_0^tS_{\alpha}(t-t')\mathcal{N}(z+y+w)(t')dt' \Big\|_{X_T^{s,\frac{1}{2}+2\epsilon}}\\
	+& \Big\|\int_0^tS_{\alpha}(t-t')\Pi_{N_k}\big(\mathcal{N}(z+y+w)(t')-\mathcal{N}(z_k+y_k+w_k)(t') \big)dt'
	\Big\|_{X_T^{s,\frac{1}{2}+2\epsilon}}.
	\end{split}
	\end{equation*}
	The first term on the right side is $o(1)$, as $k\rightarrow\infty$, since $\mathcal{N}(z+y+w)\in X_T^{s,-\frac{1}{2}+\epsilon}.$ Note that $\mathcal{N}(z+y+w)-\mathcal{N}(z_k+y_k+w_k)$  consists of the terms
	$$ \mathcal{N}(z-z_k,z-z_k,z-z_k),\quad \mathcal{N}(w-w_k+y-y_k,\cdot,\cdot),\quad \mathcal{N}(\cdot, w-w_k+y-y_k,\cdot),\cdots
	$$
	Therefore, the second term on the right hand-side of the last inequality can be bounded by
	$$  C\mathcal{W}_{s,\epsilon}(\phi_{0,k}-\phi_0)+CR^2T^{\theta}\big(\|w_k-w\|_{X_T^{s,\frac{1}{2}+\epsilon}}+\|r_{0,k}-r_0\|_{H_x^s} \big),
	$$
	where we used Proposition~\ref{multi-linear}. By choosing $c>0$ small enough, $\kappa>0$ large enough such that
	$ CT^{\theta}R^2<\frac{1}{2}$, we have
	$$ \|w_k-w\|_{X_T^{s,\frac{1}{2}+2\epsilon}}\leq 2C \mathcal{W}_{s,\epsilon}(\phi_{0,k}-\phi_0) +T^{\theta}R^2\|r_{0,k}-r_0\|_{H_x^s} =o(1),\quad k\rightarrow\infty.
	$$  
This completes the proof of Proposition~\ref{local-convergence}.
\end{proof}
%%%%%%%%%%%%%%%%%%%%%%%%%%%%%%%%%%%%%%
%%%%%%%%%%%%%%%%%%%%%%%%%%%%%%%%%%%%%%%
%%%%%%%%%%%%%%%%%%%%%%%%%%%%%%%%%%%%%%%%%%
%%%%%%%%%%%%%%%%%%%%%%%%%%%%%%%%%%%%%%%%%%
%%%%%%%%%%%%%%%%%%%%%%%%%%%%%%%%%%%%%%%%%%%%%%
\section{Deterministic trilinear estimate}
In this section, we prove the trilinear estimates in Proposition~\ref{multi-linear}. Note that by the symmetric role of the first place and the third place in the expression of $\mathcal{N}(\cdot,\cdot,\cdot)$, it is sufficient to prove (1), (2), (4), (5) of Proposition~\ref{multi-linear}. Note also that from the embedding 
\begin{align*}
W_x^{\frac{\alpha-1}{2}-\epsilon,\frac{1}{\epsilon}}\hookrightarrow W_x^{\frac{\alpha-1}{2}-3\epsilon,\infty} \text{ and } \mathcal{F}L^{\frac{\alpha}{2}-\epsilon,\frac{2}{\epsilon}}\hookrightarrow\mathcal{F}L^{\frac{\alpha}{2}-\epsilon,\infty},
\end{align*} 
it would be sufficient to prove stronger estimates by replacing  $\mathcal{Z}^{q,\epsilon}$ with $ L_{t,\mathrm{loc}}^qW_x^{\frac{\alpha-1}{2}-3\epsilon,\infty}\cap \mathcal{F}L^{\frac{\alpha}{2}-\epsilon,\infty}.$ 
In what follows, we may insert the smooth cutoff function $\eta_T$ on $[-2T,2T]$  without additional mention. We will carry out a case-by-case analysis on
$$ \|\eta_T(t)\mathcal{N}_0(v_1,v_2,v_3)\|_{X^{s,-\frac{1}{2}+2\epsilon}}\textrm{ and } \|\eta_T(t)\mathcal{N}_1(v_1,v_2,v_3)\|_{X^{s,-\frac{1}{2}+2\epsilon}}
$$
where $v_j$ takes one of the following forms 
%$v_j=S_{\alpha}(t)f_j$ or $u_j$, taken to be either of the type  
\begin{equation*}
\begin{split}
&(\mathrm{I})  \quad 
v_j=\eta_T(t)\sum_{n\in\Z}\hat{f_j}(n)e^{i(nx+|n|^{\alpha}t)}\in L_t^{\infty}\mathcal{F}L^{\frac{\alpha}{2}-\epsilon,\infty}\cap L_{t}^{q}W_x^{\frac{\alpha-1}{2}-3\epsilon,\infty}\,,
\\
&(\mathrm{II}) \quad v_j=\eta_T(t)v_j\in X^{s,\frac{1}{2}+\epsilon}.
\end{split}
\end{equation*}
By normalization, we may assume that 
$$ \sup_{|t|\leq 1}\|S_{\alpha}(t)f_j\|_{L_t^qW_x^{\frac{\alpha-1}{2}-3\epsilon,\infty}}+\|f_j\|_{\mathcal{F}L^{\frac{\alpha}{2}-\epsilon,\infty}}=1\text{ if $v_j$ is of type I}.
$$
and
$$ \|v_j\|_{X^{s,\frac{1}{2}+\epsilon}}=1 \text{ if } v_j \text{ is of type II}.
$$
In the sequel we will suppose that $\hat{f_j}(n)=\phi(n)$, i.e. that all $f_j$ are equal. Under this assumption the analysis is essentially the same and it will be satisfied in the applications of  Proposition~\ref{multi-linear}.
\\

Throughout this section, $\frac{6}{5}<\alpha<2$ and $\frac{1}{2}-\frac{\alpha}{4}\leq s<\alpha-1$.
First we have a simple estimate for the part $\mathcal{N}_0(\cdot,\cdot,\cdot)$.
\begin{proposition}\label{trilinear-N0}
	For any small $\epsilon>0$ and $q<\infty$ large enough, there exists $\theta>0$, such that for $0<T<1$,
	\begin{equation}\label{estimate:N0}
	\|\eta_T(t)\mathcal{N}_0(v_1,v_2,v_3)\|_{X^{s,-\frac{1}{2}+2\epsilon}}\lesssim T^{\theta}.
	\end{equation}	
\end{proposition}
One may remark that this proposition holds true for all $\alpha>1$.
\begin{proof}
 By Lemma \ref{time-localization} and the definition,
	\begin{equation}\label{N0-1}
	\begin{split}
	&\|\eta_T(t)\mathcal{N}_0(v_1,v_2,v_3)\|_{X^{s,-\frac{1}{2}+2\epsilon}}\lesssim T^{\epsilon}\|\eta_T(t)\mathcal{N}_0(v_1,v_2,v_3)\|_{X^{s,-\frac{1}{2}+3\epsilon}} \\
	=&T^{\epsilon}\Big\|\frac{\langle n\rangle^s}{\langle\tau-|n|^{\alpha}\rangle^{\frac{1}{2}-3\epsilon}}\int_{\tau=\tau_1-\tau_2+\tau_3}
	\widehat{v_1}(\tau_1,n)\ov{\widehat{v_2}}(\tau_2,n)\widehat{v_3}(\tau_3,n)d\tau_1 d\tau_2 \Big\|_{l_n^2L_{\tau}^2}.
	\end{split}
	\end{equation}
By abusing the notation, we may replace $v_j$ by $\eta_Tv_j$ if necessary.\\
	\noi
	$\bullet$ {\bf Case (1):} $v_1,v_2,v_3$ are of type (II). 
	Writting $\widehat{v_j}(\tau_,n)=\langle n\rangle^{-s}\langle\tau_j- |n|^{\alpha}\rangle^{-\frac{1}{2}-\epsilon}V_j(\tau,n),$  we estimate the $L_{\tau}^2$ norm of the second term of the right side by
	\begin{equation}\label{N0-3}
	\begin{split}
	&T^{\epsilon}\Big\| \langle n\rangle^s\int \widehat{v_1}(\tau_1,n)\ov{\widehat{v_2}}(\tau_2,n)\widehat{v_3}(\tau-(\tau_1-\tau_2), n)d\tau_1 d\tau_2\Big\|_{L_{\tau}^2}
	\\ \lesssim  &T^{\epsilon}
	\langle n\rangle^{-2s}\Big\|\int \frac{V_1(\tau_1,n)\ov{V}_2(\tau_2,n)V_3(\tau-(\tau_1-\tau_2),n)}{\langle\tau_1-|n|^{\alpha}\rangle^{\frac{1}{2}+\epsilon}
		\langle\tau_2-|n|^{\alpha}\rangle^{\frac{1}{2}+\epsilon}
		\langle\tau-(\tau_1-\tau_2)-|n|^{\alpha}\rangle^{\frac{1}{2}+\epsilon} }d\tau_1 d\tau_2\Big\|_{L_{\tau}^2}
	\\
	\lesssim &T^{\epsilon}\langle n\rangle^{-2s}\|V_1(\cdot,n)\|_{L_{\tau}^2}\|V_2(\cdot,n)\|_{L_{\tau}^2}\|V_3(\cdot,n)\|_{L^2_{\tau}},
	\end{split}
	\end{equation}
	where at the last step, we used Minkowski to pass the $L_{\tau}^2$ inside the integral and then Cauchy-Schwarz in $\tau_1,\tau_2$ variables.
	Finally, taking $l_n^2$ of the right side of \eqref{N0-3}, we obtain \eqref{estimate:N0}
	in this case.

	\noi
	$\bullet$ {\bf Case (2):} Exactly two $v_j$ of type (I), say, $v_1(\mathrm{I}), v_2(\mathrm{I})$ and $v_3(\mathrm{II})$. 
	With the same notation $V_3(\tau,n)=\langle n\rangle^s\langle\tau- |n|^{\alpha}\rangle^{\frac{1}{2}+\epsilon}\widehat{v_3}(\tau,n)$, we estimate
	\begin{equation*}
	%\label{N0-5}
	\begin{split}
	\eqref{N0-1}\lesssim &T^{\epsilon} \Big\||\phi(n)|^2\int_{\tau=\tau_1-\tau_2+\tau_3}\frac{\widehat{\eta_T}(\tau_1-|n|^{\alpha})\ov{\widehat{\eta_T}}(\tau_2-|n|^{\alpha}) V_3(\tau_3,n)}{\langle\tau_3-|n|^{\alpha}\rangle^{\frac{1}{2}+\epsilon}\langle\tau-|n|^{\alpha}\rangle^{\frac{1}{2}-3\epsilon} } d\tau_1 d\tau_2 \Big\|_{l_n^2L_{\tau}^{2}}\\
	\lesssim & T^{\frac{1}{2}}\|\langle n\rangle^{-\alpha}\|_{l_n^2}\|V_3\|_{l_n^{\infty}L_{\tau}^{2}}\lesssim T^{\frac{1}{2}},
	\end{split}
	\end{equation*}
	where we used the fact that $\eta_T(t)=\eta(T^{-1}t)$ and $\|\widehat{\eta_T}\|_{L^2(\R)}=O(T^{1/2})$.
	
	\noi
	$\bullet$ {\bf
		Case (3):} Exactly one $v_j$ of type (I), say, $v_1$(I), $v_2$(II), $v_3$(II).
	
	With the same notations, we have
	\begin{equation*}
	%\label{N0-6}
	\begin{split}
	\eqref{N0-1}\lesssim & T^{\epsilon}
	\|\langle n\rangle^{-s}\phi(n)\|_{l_n^2}\Big\|\int \frac{\widehat{\eta_T}(\tau_1-|n|^{\alpha})\ov{V}_2(\tau_2,n)V_3(\tau-(\tau_1-\tau_2),n)d\tau_1 d\tau_2 }{\langle\tau_2-|n|^{\alpha}\rangle^{\frac{1}{2}+\epsilon}\langle \tau-(\tau_1-\tau_2)-|n|^{\alpha}\rangle^{\frac{1}{2}+\epsilon}\langle\tau-|n|^{\alpha}\rangle^{\frac{1}{2}-3\epsilon} } \Big\|_{l_n^{\infty}L_{\tau}^2 }\\
	\lesssim &T^{\epsilon}\|\widehat{\eta_T}\|_{L_{\tau_1}^1}\|V_2\|_{l_n^{\infty}L_{\tau_2}^2}\|V_3\|_{l_n^{\infty}L_{\tau}^2}\lesssim T^{\epsilon}.
	\end{split}
	\end{equation*}
	\noi
	$\bullet$ {\bf
		Case (4):} All $u_j$ of type (I), then
	\begin{equation*}
	\begin{split}
	\eqref{N0-1} \lesssim & T^{\epsilon}\Big\|\frac{\langle n\rangle^s|\phi(n)|^3}{\langle \tau-|n|^{\alpha}\rangle^{\frac{1}{2}-3\epsilon} }\int \widehat{\eta_T}(\tau_1-|n|^{\alpha})\ov{\widehat{\eta_T}}(\tau_2-|n|^{\alpha})\widehat{\eta_T}(\tau-\tau_1+\tau_2-|n|^{\alpha})d\tau_1 d\tau_2 \Big\|_{l_n^2L_{\tau}^2}\\
	\lesssim & T^{\epsilon}\Big\|\langle n\rangle^s |\phi(n)|^3 
	\int |\widehat{\eta_T}(\tau_1-|n|^{\alpha})\widehat{\eta_T}(\tau_2-|n|^{\alpha})
	| \|\frac{\widehat{\eta_T}(\tau-\tau_1+\tau_2-|n|^{\alpha})}{\langle\tau-|n|^{\alpha}\rangle^{\frac{1}{2}-3\epsilon} }\|_{L_{\tau}^2}d\tau_1 d\tau_2
	\Big\|_{l_n^2}\\
	\lesssim &T^{\epsilon}.
	\end{split}
	\end{equation*}
	
	This completes the proof of Proposition \ref{trilinear-N0}.
\end{proof}
%%%%%%%%%%%%%%%%%%%%%%%%%%%%%%%%%%%%%%%%%%%%
\subsection{Estimate on $\mathcal{N}_1$ for high modulations}\

In the following two subsections, we will prove the following trilinear estimate for $\mathcal{N}_1$
\begin{proposition}\label{trilinear-N1}
	Assume that $v_1,v_2,v_3$ are not all of type $(\mathrm{I})$. Then there exists $0<\epsilon\ll 1$, small enough,  $2\ll q<\infty$, large enough, and $\theta=\theta(\epsilon)>0$ such that for $0<T<1$,
	\begin{equation}\label{estimate:N1}
	\|\eta_T(t)\mathcal{N}_1(v_1,v_2,v_3)\|_{X^{s,-\frac{1}{2}+2\epsilon}}\lesssim T^{\theta}.
	\end{equation}
\end{proposition}

%%%%%%%%%%%%%%%%%%%%%%%%%%%%%%%%%%%%%%%%%
Without loss of generality, in what follows, we assume that $v_1,v_2,v_3$ are \emph{not all of type} (II), since in this case, we can directly apply\footnote{Since we will only use $X^{s,b}$ type norms in this case, we can replace each Fourier coefficient in the expression of $\mathcal{N}_1(\cdot,\cdot,\cdot)$ by its absolute value and then apply Corollary \ref{Trilinear} for the full multiplication $v_1\ov{v}_2v_3$.} Corollary \ref{Trilinear}. 
We decompose  $v_1,v_2,v_3$  dyadically with frequencies of sizes $N_1,N_2,N_3$, respectively and denote them by $\mathbf{P}_{N_j}v_j$ respectively. 
 We denote by $N_{(1)},N_{(2)},N_{(3)}$ the decreasing ordering of $N_1,N_2,N_3$. By relabeling the index, we denote by $v_{(j)}=\mathbf{P}_{N_{(j)}}v_{*}$, the corresponding $v_j$-factors.  In the following, we use subscripts to imply that functions or variables are arranged in the decreasing order of the spatial frequencies $N_1,N_2,N_3$. By duality, we need  to estimate 
\begin{equation}\label{N1-dual'}
\int_{-2T}^{2T}\int_{\T}\mathcal{N}_1(v_1,v_2,v_3)\cdot \langle D_x \rangle^s \ov{v}dtdx,
\end{equation}
where $\|v\|_{X^{0,\frac{1}{2}-2\epsilon}}\leq 1$ and $v$ has compact support in $t$. It turns out that we can only treat
$$ \int_{-2T}^{2T}\int_{\T}\mathcal{N}_1\big(v_{(1)},v_{(2)},v_{(3)}\big)\cdot\langle D_x \rangle^s \ov{v}dtdx,
$$
and the analysis for other situations has no significant difference. In the high modulation cases, the main contribution comes from 
$$ \int_{-2T}^{2T}\int_{\T}v_1\ov{v_2}v_3\cdot\langle D_x\rangle^s\ov{v}dtdx,
$$
and we use the bilinear Strichartz inequalities and the regularization in the co-normal regularity (the $\frac{3}{8}$ exponent in the Strichartz inequality). 

The first goal of this subsection is to reduce the matter to the low modulation cases. More precisely, if there is any $v_j$ of type (II), we will reduce the estimate to the contribution where in the Fourier side,
$$ \langle\tau_j-|n_j|^{\alpha}\rangle\ll K_j, \textrm{ if  } v_j \textrm{ is of type (II)},
$$
for some suitable $K_j$, depending on different situations. 
We need to estimate the dyadic summation in $N_{(1)}, N_{(2)}, N_{(3)}$, $N$ for the following terms\footnote{The term $\mathcal{N}_0$  has been treated in the last subsection.}:
\begin{equation*}
%\label{AA}
\mathcal{A}=\Big|\int_{-2T}^{2T}\int_{\T}v_{(1)}\ov{v}_{(2)}v_{(3)}\cdot\langle D_x\rangle^s\ov{v}dtdx \Big|,\quad \mathcal{B}=\Big|\int_{-2T}^{2T}\big(v_{(1)},v_{(2)}\big)_{L_x^2}\big(v_{(3)},\langle D_x\rangle^s\mathbf{P}_Nv\big)_{L_x^2}dt \Big|,
\end{equation*}
and
\begin{equation*}
%\label{BB}
 \mathcal{C}=\Big|\int_{-2T}^{2T}\big(v_{(1)},\langle D_x\rangle^s\mathbf{P}_Nv \big)_{L_x^2} \big(v_{(2)}, v_{(3)}\big)_{L_x^2}dt \Big|.\footnote{ $ \mathcal{A},\mathcal{B},\mathcal{C} $ \text{ depend on the dyadic numbers } $N_{(1)}, N_{(2)}, N_{(3)}, N $\text{ and we omit the indices here. } }
\end{equation*}
For the proof in the rest subsections, we fix the index $\sigma=\frac{\alpha-1}{2}-3\epsilon$.
\subsubsection{ Estimates for the high modulations of $\mathcal{B},\mathcal{C}$. }\

We first estimate the quantities $\mathcal{B}$ and $\mathcal{C}$. Note that $\mathcal{B}=0$ unless $N_{(1)}\sim N_{(2)}$ and $N_{(3)}\sim N$. By Cauchy-Schwarz and then H\"older for the time integration, we have
$$ \mathcal{B}\lesssim N^s\|v_{(1)}\|_{L_t^4L_x^2}\|v_{(2)}\|_{L_t^4L_x^2}\|v_{(3)}\|_{L_t^4L_x^2}\|\mathbf{P}_Nv\|_{L_t^4L_x^2}.
$$
Since there is at least one of $v_{(j)}$ of type (II), using the interpolation between $X^{0,0}= L_{t}^2L_x^2$ and $X^{0,\frac{1}{2}+2\epsilon}\hookrightarrow L_t^{\infty}L_x^2$, we bound the $L_t^4L_x^2$ norm of $v_{(j)}(\mathrm{II})$ as follows 
$$\|v_{(j)}(\mathrm{II})\|_{X^{0,\frac{1}{4}+\epsilon}}\lesssim T^{\frac{1}{4}}\|v_{(j)}(\mathrm{II})\|_{X^{0,\frac{1}{2}+\epsilon}},$$
where we used Lemma \ref{time-localization}. 
Note that no matter type (I) or type (II), the dyadic summation over $N_{(1)}\sim N_{(2)}, N\sim N_{(3)}$ always converges, and we obtain that
\begin{align*}
 \sum_{\substack{N_{(1)}, N_{(2)}, N_{(3)}, N \text{ dyadic }\\ N_{(1)}\sim N_{(2)}, N_{(3)}\sim N }}N^s\|v_{(1)}\|_{L_t^4L_x^2}\|v_{(2)}\|_{L_t^4L_x^2}\|v_{(3)}\|_{L_t^4L_x^2}\|\mathbf{P}_Nv\|_{L_t^4L_x^2}
 \lesssim T^{\frac{1}{4}}.
\end{align*}
 Similarly, $\mathcal{C}=0$ unless $N_{(1)}\sim N$ and $N_{(2)}\sim N_{(3)}$. If $v_{(1)}$ is of type II, we obtain the same estimate as for $\mathcal{B}$, and the dyadic summation over $N_{(1)}\sim N, N_{(2)}\sim N_{(3)}$ converges. Now we assume that $v_{(1)}$ is of type (I). There are essentially two possibilities, either $v_{(2)}$ is of type (I) and $v_{(3)}$ is of type (II), or both are of type (II). For the former case, we bound $\mathcal{C}$ by
 $$ \mathcal{C}\lesssim N_{(1)}^{s-\sigma}N_{(2)}^{-\sigma}N_{(3)}^{-s}\|v_{(1)}\|_{L_t^qH_x^{\sigma}}\|v_{(2)}\|_{L_t^qH_x^{\sigma}}\|v_{(3)}\|_{L_t^{q_1}H_x^s } \|\mathbf{P}_Nv\|_{L_t^{q_1}L_x^2 }
 $$
 where for small $\epsilon>0$, large $q<\infty$,
 \begin{small}
 \begin{equation*}
 %\label{sigma,q1}
 q_1=\frac{2q}{q-2}, \text{ almost } 2.
 \end{equation*}
 \end{small}
By interpolation between $X^{0,0}=L_t^2L_x^2$ and $X^{0,\frac{1}{2}+\epsilon}\hookrightarrow L_t^{\infty}L_x^2$, we have $X^{s,\frac{1+2\epsilon}{q}}\hookrightarrow L_t^{q_1}H_x^s$, thus
\begin{equation*}
%\label{I-II}
 \mathcal{C}\lesssim N_{(1)}^{s-\sigma}N_{(2)}^{-\sigma}N_{(3)}^{-s}\|v_{(1)}\|_{L_t^qH_x^{\sigma}}
\|v_{(2)}\|_{L_t^qH_x^{\sigma}}\|v_{(3)}\|_{X^{s,\frac{1+2\epsilon}{q}}}\|\mathbf{P}_Nv\|_{X^{0,\frac{1+2\epsilon}{q}}}.
\end{equation*}
We can choose $q$ large enough such that $\frac{1+2\epsilon}{q}<\epsilon$.
For the case where both $v_{(2)}$ and $v_{(3)}$ are of type (II), we have 
$$ \mathcal{C}\lesssim N_{(1)}^{s-\sigma}N_{(2)}^{-s}N_{(3)}^{-s}\|v_{(1)}\|_{L_t^qH_x^{\sigma}}
\|v_{(2)}\|_{L_t^{\frac{3q}{q-1}}H_x^{s}}\|v_{(3)}\|_{L_t^{\frac{3q}{q-1}}H_x^{s}}\|\mathbf{P}_Nv\|_{L_t^{\frac{3q}{q-1}}L_x^{2}},
$$
and by interpolation, we obtain that
\begin{equation*}
%\label{II-II}
 \mathcal{C} \lesssim N_{(1)}^{s-\sigma}N_{(2)}^{-s}N_{(3)}^{-s}\|v_{(1)}\|_{L_t^qH_x^{\sigma}}
\|v_{(2)}\|_{X^{s,\frac{1}{6}+\delta(q,\epsilon)}}\|v_{(3)}\|_{X^{s,\frac{1}{6}+\delta(q,\epsilon)}}\|\mathbf{P}_Nv\|_{X^{0,\frac{1}{6}+\delta(q,\epsilon)}},
\end{equation*}
where
\begin{small}$$ \delta(q,\epsilon)=\frac{(1+2\epsilon)(q+2)}{6q}-\frac{1}{6}<\epsilon,
$$
\end{small}
provided that $q$ is chosen large enough. For each $v_j$ of type (II) and $v$, we divide them as
$$ v_j(\tau,n)=v^{\mathrm{high}}_{j}+v^{\mathrm{low}}_{j}, \quad v(\tau,n)=v^{\mathrm{high}}+v^{\mathrm{low}},
$$
where
\begin{align*}
 \widehat{v^{\text{high}}_j}(\tau,n)=\mathbf{1}_{\langle\tau-|n|^{\alpha} \rangle^{\frac{1}{3}}\geq  N_{(1)}^{s-\sigma} } \widehat{v_j}(\tau,n),\quad \widehat{v^{\text{high}}}(\tau,n)=\mathbf{1}_{\langle\tau-|n|^{\alpha} \rangle^{\frac{1}{3}}\geq  N_{(1)}^{s-\sigma} } \widehat{v}(\tau,n).
\end{align*}
Then for the case $v_{(2)}=v_{(2)}(\mathrm{I}), v_{(3)}=v_{(3)}(\mathrm{II})$, if one of $v_{(3)}^{\text{low}}, \mathbf{P}_Nv^{\text{low} }=0$, we have
\begin{align*}
%\label{CC1}
\sum_{\substack{ N_{(1)}, N_{(2)}, N_{(3)}, N\\
		N_{(1)}\sim N, N_{(2)}\sim N_{(3)}
	} 
}& N_{(1)}^{s-\sigma}N_{(2)}^{-s}N_{(3)}^{-s} \|v_{(1)}\|_{L_t^qH_x^{\sigma}} \|v_{(2)}\|_{L_t^qH_x^{\sigma}} \|v_{(3)}\|_{X^{s,\frac{1+2\epsilon}{q}}}\|\mathbf{P}_Nv\|_{X^{0,\frac{1+2\epsilon}{q}}}\notag \\
\lesssim &T^{1/2}.
\end{align*}
For the case $v_{(2)}=v_{(2)}(\mathrm{II}), v_{(3)}=v_{(3)}(\mathrm{II})$, if one of $v_{(2)}^{ \text{low}},v_{(3)}^{\text{low}}, \mathbf{P}_Nv^{\text{low} }=0$, we have
\begin{align*}
%\label{CC2}
\sum_{\substack{ N_{(1)}, N_{(2)}, N_{(3)}, N\\
		N_{(1)}\sim N, N_{(2)}\sim N_{(3)}
	} 
} &N_{(1)}^{s-\sigma}N_{(2)}^{-s}N_{(3)}^{-s}\|v_{(1)}\|_{L_t^qH_x^{\sigma}}
\|v_{(2)}\|_{X^{s,\frac{1}{6}+\delta(q,\epsilon)}}\|v_{(3)}\|_{X^{s,\frac{1}{6}+\delta(q,\epsilon)}}\|\mathbf{P}_Nv\|_{X^{0,\frac{1}{6}+\delta(q,\epsilon)}}\notag \\
\lesssim T^{1/3}.
\end{align*}

\subsubsection{ Estimates for the high modulations of $\mathcal{A}$. }\

Since there is no significant issue, we will drop the conjugate sign. It remains to estimate the dyadic summation over $N_{(1)}, N_{(2)}, N_{(3)},$ $ N$ for
$$ \mathcal{A}=\Big|\int_{-2T}^{2T}\int_{\T} v_{(1)}v_{(2)}v_{(3)}\cdot \langle D_x\rangle^sv dtdx
 \Big|.
$$
%%%%%%%%%%%%%%%%%%%%%%%%%%%%%%%%%%%%%%%%%%%%%%%%%%
\noi
$\bullet$ {\bf Case A:} $v_{(1)}$ and $v_{(2)}$ are of type (II).

In this case $v_{(3)}$ must be of type (I). Put
\[
\beta_q:=\frac{1+2\epsilon}{q}.
\]
We choose $q$ large enough and $\epsilon>0$ small enough so that
$\frac1q<\beta_q<\beta_q+\epsilon<\frac12$.
Regrouping the terms as
$\|v_{(1)}v_{(3)}\|_{L_{t,x}^2}\cdot
\|v_{(2)}\langle D_x\rangle^s\mathbf{P}_Nv\|_{L_{t,x}^2}$, we estimate the first factor by H\"older and the one-dimensional Sobolev embedding in time, exactly as in the proof of \eqref{mixed-typeI-typeII}:
\[
\begin{split}
	\|v_{(1)}v_{(3)}\|_{L_{t,x}^2}
	&\leq
	\|v_{(1)}\|_{L_t^{\frac{2q}{q-2}}L_x^2}
	\|v_{(3)}\|_{L_t^qL_x^\infty}  \\
	&\lesssim
	N_{(3)}^{-\sigma}\|v_{(1)}\|_{X^{0,\beta_q}}
	\lesssim
	T^{\epsilon}N_{(3)}^{-\sigma}
	\|v_{(1)}\|_{X^{0,\beta_q+\epsilon}},
\end{split}
\]
where we used Lemma~\ref{time-localization} in the last step. On the other hand, by (2) of Corollary~\ref{bilinear-Xsb},
\[
\|v_{(2)}\langle D_x\rangle^s\mathbf{P}_Nv\|_{L_{t,x}^2}
\lesssim
N^sN_{(2)}^s
\|v_{(2)}\|_{X^{0,\frac38}}
\|\mathbf{P}_Nv\|_{X^{0,\frac38}}.
\]
Consequently,
\begin{equation}\label{caseA-new-high-bound}
	\mathcal{A}
	\lesssim
	T^{\epsilon}N^sN_{(2)}^sN_{(3)}^{-\sigma}
	\|v_{(1)}\|_{X^{0,\beta_q+\epsilon}}
	\|v_{(2)}\|_{X^{0,\frac38}}
	\|\mathbf{P}_Nv\|_{X^{0,\frac38}}.
\end{equation}
In the Fourier side, either $\langle\tau_j-|n_j|^{\alpha}\rangle^{\frac{1}{8}-\epsilon}\gtrsim (N\wedge N_{(2)})^{\epsilon}$, $j=1,2$ or $\langle\tau-|n|^{\alpha}\rangle^{\frac{1}{8}-\epsilon}\gtrsim (N\wedge N_{(2)})^{\epsilon} $ hold true for some $\epsilon>0$, the dyadic summation over $N_{(1)}\geq N_{(2)} \geq N_{(3)}, N\leq N_{(1)}$ converges. 
Hence it remains to estimate the contributions to \eqref{N1-dual'} with a cutoff on the Fourier side on the region satisfying
\begin{equation}\label{low-modulation1}
	\begin{split}
		\langle &\tau-|n|^{\alpha}\rangle^{\frac{1}{8}}\ll (N_{(2)}\wedge N)^{2\epsilon} \textrm{ and }
		\\
		&\langle \tau_j-|n_j|^{\alpha}\rangle^{\frac{1}{8}}\ll (N_{(2)}\wedge N)^{2\epsilon},\textrm{ if } v_j \textrm{ of type (II) and $N_{(3)}\ll N_{(1)}$}.
	\end{split}
\end{equation}

\noi
$\bullet$ {\bf Case B:} $v_{(1)}$ is of type (II) and $v_{(2)}$ is of type (I).

Suppose first that $v_{(3)}$ is of type (II). We regroup the factors as
$\|v_{(1)}v_{(2)}\|_{L_{t,x}^2}\cdot
\|v_{(3)}\langle D_x\rangle^s\mathbf{P}_Nv\|_{L_{t,x}^2}$. Arguing as in Case A, we get
\[
\|v_{(1)}v_{(2)}\|_{L_{t,x}^2}
\lesssim
T^{\epsilon}N_{(2)}^{-\sigma}
\|v_{(1)}\|_{X^{0,\beta_q+\epsilon}},
\]
and by (2) of Corollary~\ref{bilinear-Xsb},
\[
\|v_{(3)}\langle D_x\rangle^s\mathbf{P}_Nv\|_{L_{t,x}^2}
\lesssim
N^sN_{(3)}^s
\|v_{(3)}\|_{X^{0,\frac38}}
\|\mathbf{P}_Nv\|_{X^{0,\frac38}}.
\]
Therefore
\begin{equation}\label{caseB-new-high-bound}
	\mathcal{A}
	\lesssim
	T^{\epsilon}N^sN_{(3)}^sN_{(2)}^{-\sigma}
	\|v_{(1)}\|_{X^{0,\beta_q+\epsilon}}
	\|v_{(3)}\|_{X^{0,\frac38}}
	\|\mathbf{P}_Nv\|_{X^{0,\frac38}}.
\end{equation}
As before, if in the Fourier side either
\[
\langle\tau_j-|n_j|^{\alpha}\rangle^{\frac18-\epsilon}
\gtrsim (N\wedge N_{(3)})^{\epsilon},
\qquad j=1,3,
\]
or
\[
\langle\tau-|n|^{\alpha}\rangle^{\frac18-\epsilon}
\gtrsim (N\wedge N_{(3)})^{\epsilon},
\]
then we gain a factor $(N\wedge N_{(3)})^{-\delta}$ for some $\delta>0$, and the dyadic summation converges as in Case A. Hence it remains only to consider the low-modulation region
\begin{equation}\label{low-modulation2}
	\begin{split}
		\langle &\tau-|n|^{\alpha}\rangle^{\frac{1}{8}}\ll (N_{(3)}\wedge N)^{2\epsilon} \textrm{ and }
		\\
		&\langle \tau_j-|n_j|^{\alpha}\rangle^{\frac{1}{8}}\ll (N_{(3)}\wedge N)^{2\epsilon},\textrm{ if } v_j \textrm{ of type (II) and $N_{(2)}\ll N_{(1)}$}.
	\end{split}
\end{equation}

%%%%%%%%%%%%%%%%%%%%%%%%%%%%%%%%%%%%%%%%%%%%%%%%%%

%%%%%%%%%%%%%%%%%%%%%%%%%%%%%%%%%%%%%%%%%%%%%%

\noi
$\bullet$ {\bf Case C:} $v_{(1)}$ is of type (I), and $v_{(2)}, v_{(3)}$ are of type (II).\\
Using the bilinear Strichartz estimate and Lemma~\ref{time-localization}, we have
\begin{equation*}
\begin{split}
\mathcal{A}\lesssim &
N^s\|v_{(1)} v_{(2)}\|_{L_{t,x}^2}\|v_{(3)}\mathbf{P}_Nv\|_{L_{t,x}^2}\\ \lesssim & T^{2\epsilon}(N_{(1)})^{s-\sigma}(N_{(2)})^{-s}\|v_{(2)}\|_{X^{s,\frac{3}{8}+2\epsilon}}\|v_{(3)}\|_{X^{s,\frac{3}{8}}}\|\mathbf{P}_Nv\|_{X^{0,\frac{3}{8}}}.
\end{split}
\end{equation*}
If $N_{(2)}\sim N_{(1)}$, the dyadic summation converges directly, without reducing to the low modulation.
Hence, it remains to estimate the contribution to \eqref{N1-dual'} from the region satisfying
\begin{equation}\label{low-modulation3}
\begin{split}
&\langle \tau-|n|^{\alpha}\rangle^{\frac{1}{8}}\ll N_{(1)}^{s-\sigma}N_{(2)}^{-s} \textrm{ and }\\
&\langle\tau_j-|n_j|^{\alpha}\rangle^{\frac{1}{8}}\ll N_{(1)}^{s-\sigma}N_{(2)}^{-s}, \textrm{ if } v_j \textrm{ is of type (II) and $N_{(2)}\ll N_{(1)}$}.
\end{split}
\end{equation}

\noi
$\bullet$ {\bf Case D:} $v_{(1)}$ of type (I), and either $v_{(2)}(\mathrm{II}), v_{(3)}(\mathrm{I})$ or $v_{(2)}$(I), $v_{(3)}$(II).

Suppose that $v_{(2)}=v_{(2)}$(I) and $v_{(3)}=v_{(3)}$(II). 
We have
\begin{equation*}
\begin{split}
\mathcal{A}\lesssim &N^s\|v_{(1)}\|_{L_t^qL_{x}^{\infty}}\|v_{(2)}\|_{L_t^qL_{x}^{\infty}}\|v_{(3)}\|_{L_t^{\frac{2q}{q-2}}L_{x}^2}\|\mathbf{P}_Nv\|_{L_t^{\frac{2q}{q-2}}L_{x}^2}\\ \lesssim &T^{\epsilon}N^s(N_{(1)}N_{(2)})^{-\sigma}N_{(3)}^{-s}\|v_{(3)}\|_{X^{s,\frac{1+2\epsilon}{q}+\epsilon }}\|\mathbf{P}_Nv\|_{X^{0,\frac{1+2\epsilon}{q} }},
\end{split}
\end{equation*}
where we use the interpolation $X^{0,\frac{1+2\epsilon}{q}}\subset L_t^{\frac{2q}{q-2}}L_x^2$ and Lemma \ref{time-localization} as before. Since $s<\alpha-1$, we may choose $\epsilon\ll 1$, $q\gg 1$, such that $s<2\sigma$ and $\frac{1+2\epsilon}{q}<\epsilon$, then if $N_{(2)}\sim N_{(1)}$, the dyadic summation converges. Otherwise, it reduces to estimate the contribution to \eqref{estimate:N1} from the Fourier region satisfying
\begin{equation}\label{low-modulation4}
\begin{split}
&\langle \tau-|n|^{\alpha}\rangle^{\frac{1}{2}}\ll N_{(1)}^{s-\sigma} \textrm{ and }\\
&\langle\tau_j-|n_j|^{\alpha}\rangle^{\frac{1}{2}}\ll N_{(1)}^{s-\sigma}\textrm{ if } v_j \textrm{ is of type (II) and $N_{(2)}\ll N_{(1)}$}.
\end{split}
\end{equation}

Suppose that $v_{(2)}=v_{(2)}$(II) and $v_{(3)}=v_{(3)}(\mathrm{I})$, then we obtain the similar bound  (switching the role of $v_{(2)}$ and $v_{(3)}$ and using bilinear Strichartz)
$$ \mathcal{A}\lesssim T^{\epsilon}N^sN_{(1)}^{-\sigma}N_{(3)}^{-\sigma}N_{(2)}^{-s}\|v_{(2)}\|_{X^{s,\frac{3}{8}+2\epsilon}}\|\mathbf{P}_Nv\|_{X^{0,\frac{3}{8}+\epsilon}}.
$$
Hence it reduces the matter to the same low modulation case \eqref{low-modulation4}.
%%This is the only place in the high modulation analysis where we need the constraint $s<\alpha-1$. Otherwise, we should reduce one more low-modulation case with the constraint $N_{(2)}\sim N_{(1)}$.
In summary,  when $\mathbf{P}_{N_j}v_j$ is of type (II), we may write it as
$$ \mathbf{P}_{N_j}v_j=\mathbf{P}_{N_j}v_j^{\text{low} }+ \mathbf{P}_{N_j}v_j^{\text{high} },
\quad  \mathbf{P}_{N}v=\mathbf{P}_{N}v^{\text{low} }+ \mathbf{P}_{N}v^{\text{high} }
$$
where
$$ \widehat{\mathbf{P}_{N_j}v_j^{\text{low}} }=\mathbf{1}_{\langle\tau-|n|^{\alpha}\rangle\leq K}\widehat{\mathbf{P}_{N_j}v_j}(\tau,n),\quad \widehat{\mathbf{P}_{N}v^{\text{low}} }=\mathbf{1}_{\langle\tau-|n|^{\alpha}\rangle\leq K}\widehat{\mathbf{P}_{N}v}(\tau,n),
$$ 
and the modulation $K$ is given specifically, according to the case (A), (B), (C), (D). The $\mathbf{P}_{N_j}v_j^{\text{low}}$ is called the low-modulation portion.
From the discussions above, if at least one of the type (II) $\mathbf{P}_{N_1}v_1, \mathbf{P}_{N_2}v_2, \mathbf{P}_{N_3}v_3 $ or $\mathbf{P}_Nv$ has zero low modulation portion,
we have
$$  \int\mathcal{N}_1\big(\mathbf{P}_{N_1}v_1\mathbf{P}_{N_2}\ov{v_2}\mathbf{P}_{N_3}v_3 \big)\cdot\mathbf{P}_N\ov{v}dtdx\lesssim T^{\theta}c_{N_1,N_2,N_3,N},$$
where $$
 \sum_{N_1,N_2,N_3,N\text{ dyadic } } c_{N_1,N_2,N_3,N}\lesssim 1. 
$$ 
Therefore, the remaining main contributions come from the low modulation portions $\mathbf{P}_{N_j}v_j^{\text{low } } $ and $\mathbf{P}_{N}v^{\text{low } } $ \footnote{Note that we have inserted implicitly time cutoff functions to perform the integration in $t$ over finite intervals.}.
In what follows, we assume that $\langle\tau-|n|^{\alpha}\rangle\ll K$ and $\langle\tau_j-|n_j|^{\alpha}\rangle\ll K$ if $v_j=v_j$(II) without stating explicitly. Moreover, we assume that each $v_j$ is decomposed dyadically in spatial frequency $|n_j|\sim N_j$, satisfying $N_{(2)}\ll N_{(1)}$ for Cases (B)(C)(D), and $N_{(3)}\ll N_{(1)}$ for Case (A). 
\subsection{Low modulation reduction}
The goal of this subsection is to setup suitable low-modulation estimates that we need. 
Set
$$ \Gamma(\ov{n}):=\{(n_1,n_2,n_3)\in \Z^3: n=n_1-n_2+n_3, n_2\neq n_1,n_3 \},
$$
and
$$ \Gamma_{2}(\lambda,n):=\{(\tau_1,\tau_2,\tau_3)\in\R^3: 
%\tau=
\lambda+|n|^{\alpha}=\tau_1-\tau_2+\tau_3 \}.
$$
Let us recall a standard representation for functions in $X^{s,b}$. 
Given a function $f(t,x)$, we can write $f$ as
$$
f(t,x)=\int  \langle\lambda\rangle^{-b}\Big(\sum_{m}\langle m\rangle^{2s}\langle\lambda\rangle^{2b}|\widehat{f}(\lambda+|m|^{\alpha},m)|^2\Big)^{\frac{1}{2}}
 \Big(e^{i\lambda t}\sum_{n}a_{\lambda}(n)e^{inx+i|n|^{\alpha}\lambda}\Big)d\lambda,
$$
where 
$$ a_{\lambda}(n)=\frac{\widehat{f}(\lambda+|n|^{\alpha},n)}{\Big(\sum_{m}\langle m\rangle^{2s}|\widehat{f}(\lambda+|m|^{\alpha},m)|^2\Big)^{\frac{1}{2}}}.
$$
Note that $\sum_{n}\langle n\rangle^{2s}|a_{\lambda}(n)|^2=1$. For $\|f\|_{X^{s,b}}\leq 1$, if its modulation is bounded from above by some $K\geq 1$, then by Cauchy-Schwarz, we have
\begin{equation*}
%\label{representation-bound}
\int\langle \lambda\rangle^{-b}\Big(\sum_{n}\langle n\rangle^{2s}\langle\lambda\rangle^{2b}|\widehat{f}(\lambda+|n|^{\alpha},n)|^2\Big)^{\frac{1}{2}}d\lambda\lesssim 1+K^{1-2b}\,.
\end{equation*}
As explained in the last subsection, we need to estimate the low-modulation component of $\|\eta_T(t)\mathcal{N}_1(v_1,v_2,v_3)\|_{X^{s,-\frac{1}{2}+2\epsilon}}$. Since at least one of $v_1,v_2,v_3$ is of type (II), we can replace $v_j(\mathrm{II})$ by $\eta_T(t)v_j(\mathrm{II})$, and estimate only $\|\kappa(t)\mathcal{N}_1(v_1,v_2,v_3)\|_{X^{s,-\frac{1}{2}+2\epsilon}}$, with some fixed time cutoff $\kappa\in C_c^{\infty}(\R)$, $\kappa(t)\equiv 1$ if $|t|\leq 1$ and $\kappa(t)\eta_T(t)=\eta_T(t)$, for $T<1$. We denote by $(\kappa(t)\mathcal{N}_1(v_1,v_2,v_3))_K^{\text{low}}$  the modulation smaller than $K$.
By the H\"older inequality, we have
\begin{equation*}
\begin{split}
&\|(\kappa(t)\mathcal{N}_1(v_1,v_2,v_3))_K^{\text{low}}\|_{X^{s,-\frac{1}{2}+2\epsilon}}\\=&\Big(\sum_{n}\langle n\rangle^{2s}\int_{|\lambda|<K}\frac{|(\mathcal{F}_{t,x}\kappa(t)\mathcal{N}_1(v_1,v_2,v_3))(\lambda+|n|^{\alpha},n)|^2 }{\langle\lambda\rangle^{1-4\epsilon} }d\lambda \Big)^{1/2}\\
\lesssim &K^{2\epsilon}\|\mathbf{1}_{|\lambda|<K}\langle n\rangle^s\mathcal{F}_{t,x}(\kappa(t)\mathcal{N}_1(v_1,v_2,v_3)) (\lambda+|n|^{\alpha},n)\|_{L_{\lambda}^{\infty}l_n^2}\,.
\end{split}
\end{equation*}
Note that
\begin{equation*}
\begin{split}
(\mathcal{F}_{t,x}\mathcal{N}_1(v_1,v_2,v_3) )(\tau,n)=\sum_{(n_1,n_2,n_3)\in\Gamma(\ov{n})}\int_{(\tau_1,\tau_2,\tau_3)\in\Gamma_2(\tau-|n|^\alpha,n)}\widehat{v_1}(\tau_1,n_1)\ov{\widehat{v_2}}(\tau_2,n_2)\widehat{v_3}(\tau_3,n_3)d\tau_1 d\tau_2,
\end{split}
\end{equation*}
where
$$ \widehat{v_j}(\tau_j,n_j)=\phi(n_j)\delta(\tau_j-|n_j|^{\alpha})\quad \textrm{if $v_j$ is of type (I) }\footnote{\textrm{ We send the time-cutoff $\eta_T(t)$ to the $v_j$ of type (II).}}
$$
or
\begin{equation*}
\begin{split}
\widehat{v_j}(\tau_j,n_j)=&\int_{|\lambda_j|<K}\langle \lambda_j\rangle^{-\frac{1}{2}+\epsilon}c_j(\lambda_j)a_{\lambda_j}(n_j)\delta(\tau_j-\lambda_j-|n_j|^{\alpha})d\lambda_j \quad \textrm{if $v_j$ is of type (II)},
\end{split}
\end{equation*}
with $\sum_{n_j} \langle n_j\rangle^{2s}|a_{\lambda_j}(n_j)|^2=1$ and
$$ c_{j}(\lambda_j)=\Big(\sum_{m_j}\langle m_j\rangle^{2s}\langle \lambda_j\rangle^{1-2\epsilon}|\widehat{v_j}(\lambda_j+|m_j|^{\alpha},m_j)|^2 \Big)^{\frac{1}{2}}.
$$
Therefore, if there is exactly one $v_j$ of type (II), say $v_1$(I), $v_2$(I), $v_3$(II), a direct calculation yields
\begin{equation*}
\begin{split}
(\mathcal{F}_{t,x}\kappa(t)\mathcal{N}_1(v_1,v_2,v_3) )(\tau,n)\\:=\sum_{(n_1,n_2,n_3)\in\Gamma(\ov{n})}\int_{ |\lambda_3|<K}&\langle\lambda_3\rangle^{-\frac{1}{2}+\epsilon}c_3(\lambda_3)\phi(n_1)\ov{\phi}(n_2)\\
\times&\widehat{\kappa}\left(\tau-\lambda_3-|n_1|^{\alpha}+|n_2|^{\alpha}-|n_3|^{\alpha} \right)a_{\lambda_3}(n_3) d\lambda_3.
\end{split}
\end{equation*}
If $v_2$, $v_3$ are of type (II), and $v_1$ of type (I), we have
\begin{equation*}
\begin{split}
(\mathcal{F}_{t,x}\kappa(t)\mathcal{N}_1(v_1,v_2,v_3) )(\tau,n)\\:=\sum_{(n_1,n_2,n_3)\in\Gamma(\ov{n})}\iint_{|\lambda_2|<K, |\lambda_3|<K}&\langle\lambda_2\rangle^{-\frac{1}{2}+\epsilon}\langle\lambda_3\rangle^{-\frac{1}{2}+\epsilon}\ov{c_2}(\lambda_2)c_3(\lambda_3)\phi(n_1)\\
\times&\widehat{\kappa}\left(\tau+\lambda_2-\lambda_3-|n_1|^{\alpha}+|n_2|^{\alpha}-|n_3|^{\alpha} \right)\ov{a_{\lambda_2}}(n_2)a_{\lambda_3}(n_3) d\lambda_2 d\lambda_3.
\end{split}
\end{equation*}
Since we only care about the low modulation part of $\mathcal{N}_1(v_1,v_2,v_3)$, below $|\lambda|\lesssim K$, applying the H\"older inequality, 
we obtain that
\begin{equation*}
%\label{N0-bound6}
\begin{split}
\|(\kappa(t)\mathcal{N}_1(v_1,v_2,v_3))_K^{\text{low}}\|_{X^{s,-\frac{1}{2}+2\epsilon}} \lesssim K^{2\epsilon}\sup_{\langle\lambda\rangle< K}\left\| \langle n\rangle^s (\mathcal{F}_{t,x}(\kappa(t)\mathcal{N}_1(v_1,v_2,v_3) )(\lambda+|n|^{\alpha},n) \right\|_{l_n^2}.
\end{split}
\end{equation*}
Since $v_j=\eta_T(t)v_j$, if it is of type (II), from Lemma~\ref{time-localization}, we have
$$ \int_{\R}|c_j(\lambda_j)|^2d\lambda_j=\|v_j\|_{X^{s,\frac{1}{2}-\epsilon}}\lesssim T^{2\epsilon}\|v_j\|_{X^{s,\frac{1}{2}+\epsilon}}.
$$
Therefore, we obtain that
\begin{equation}\label{goal-0}
\begin{split}
&\|(\kappa(t)\mathcal{N}_1(v_1,v_2,v_3))_K^{\text{low}}
\|_{X^{s,-\frac{1}{2}+2\epsilon}}\\ \lesssim &T^{2\epsilon}K^{3\epsilon}\sup_{\substack{|\lambda|<K\\
		|\lambda_j|<K,j=2,3 }}\Big\|\sum_{(n_1,n_2,n_3)\in\Gamma(\ov{n})} \phi(n_1)\ov{a_{\lambda_2}}(n_2)a_{\lambda_3}(n_3) \widehat{\kappa}(\lambda+\lambda_2-\lambda_3-\Phi(\ov{n})) \Big\|_{l_n^2}
\end{split}
\end{equation}
or
\begin{equation}\label{goal-0'}
\begin{split}
&\|(\kappa(t)\mathcal{N}_1(v_1,v_2,v_3))_K^{\text{low}}\|_{X^{s,-\frac{1}{2}+2\epsilon}}\\ \lesssim &T^{2\epsilon}K^{2\epsilon}\sup_{\substack{|\lambda|<K\\
		|\lambda_3|<K }}\Big\|\sum_{(n_1,n_2,n_3)\in\Gamma(\ov{n})} \phi(n_1)\ov{\phi(n_2)}a_{\lambda_3}(n_3) \widehat{\kappa}(\lambda-\lambda_3-\Phi(\ov{n})) \Big\|_{l_n^2},
\end{split}
\end{equation}
depending on how many $v_j$ are of type (II).

From the discussion of the last subsection, to finish the proof, we need to estimate the R.H.S. of \eqref{goal-0} and \eqref{goal-0'},  according to the constraint $K$, defined as \eqref{low-modulation1},\eqref{low-modulation2},\eqref{low-modulation3} and \eqref{low-modulation4}, according to Case (A),(B),(C),(D), respectively. We will do this by dyadically decomposing $|n_j|\sim N_j$. In what follows, we only estimate each dyadic pieces of R.H.S of \eqref{goal-0} or \eqref{goal-0'}, satisfying that $N_{(2)}\ll N_{(1)}$, for Cases (B)(C)(D), and $N_{(3)}\ll N_{(1)}$ for Case (A), and deduce the correct numerology so that the final dyadic summation over $N_1, N_2, N_3$ will converge. In summary, we have to deal with the following cases:

\noi
$\bullet$ {\bf Case 1:} $v_{(1)}=v_{(1)}$(I),$v_{(2)}=v_{(2)}$(II), $v_{(3)}=v_{(3)}$(II) and $N_{(2)}\ll N_{(1)}$. The modulation bound in this case is
$$  K_{1}=N_{(1)}^{8(s-\sigma)}N_{(2)}^{-8s}.
$$ 
Therefore, the dyadic pieces of \eqref{goal-0'} is bounded by
\begin{equation*}
%\label{eq:Case1}
\begin{split}
T^{2\epsilon}K_{1}^{3\epsilon}\sup_{|\mu|\lesssim K_{1} }\Big(\sum_{|n|\lesssim N_{(1)}}\langle n\rangle^{2s}\Big|\sum_{(n_1,n_2,n_3)\in\Gamma(\ov{n}) }\widehat{\kappa}(\mu-\Phi(\ov{n}) )a_1(n_1)a_2(n_2)a_3(n_3)
\Big|^2 \Big)^{1/2},
\end{split}
\end{equation*}
where $a_{(1)}(n)=\phi(n)$ and $\sum_{|n|\sim N^{(j)}}|a_{(j)}(n)|^2\lesssim N_{(j)}^{-2s}, j=2,3$.

\noi
$\bullet$ {\bf Case 2:} $v_{(1)}=v_{(1)}$(I), and exactly one of $v_{(2)}, v_{(3)}$ is of type (II) and $N_{(2)}\ll N_{(1)}$. In this case, the modulation bound is
$$ K_{2}=N_{(1)}^{2(s-\sigma)},
$$
and the dyadic pieces of \eqref{goal-0} is bounded by
\begin{equation*}
%\label{eq:Case2}
\begin{split}
T^{\epsilon}K_{2}^{2\epsilon}\sup_{|\mu|\lesssim K_{2} }\Big(\sum_{|n|\lesssim N_{(1)}}\langle n\rangle^{2s}\Big|\sum_{(n_1,n_2,n_3)\in\Gamma(\ov{n}) }\widehat{\kappa}(\mu-\Phi(\ov{n}) )a_1(n_1)a_2(n_2)a_3(n_3)
\Big|^2 \Big)^{1/2},
\end{split}
\end{equation*}
where
$a_{(1)}(n)=\phi(n)$, and one  of $a_{(2)}(n)$, $a_{(3)}(n)$ is $\phi(n)$, while the rest one satisfies $\sum_{|n|\sim N_{(j)}}|a_{(j)}(n)|^2\lesssim N_{(j)}^{-2s}$.

\noi
$\bullet$ {\bf Case 3:} $v_{(1)}=v_{(1)}$(II), and one of $v_{(2)},v_{(3)}$ is of type (I) and $N_{(3)}\ll N_{(1)}$. In this case, the modulation bound is
$$ K_{3}=N_{(2)}^{\epsilon}.
$$
and the dyadic pieces of \eqref{goal-0} (or \eqref{goal-0'}) are bounded by
\begin{equation*}
%\label{eq:Case3}
\begin{split}
T^{\epsilon}K_{3}^{3\epsilon}\sup_{|\mu|\lesssim K_{3} }\Big(\sum_{|n|\lesssim N_{(1)}}\langle n\rangle^{2s}\Big|\sum_{(n_1,n_2,n_3)\in\Gamma(\ov{n}) }\widehat{\kappa}(\mu-\Phi(\ov{n}) )a_1(n_1)a_2(n_2)a_3(n_3)
\Big|^2 \Big)^{1/2},
\end{split}
\end{equation*}
where $\sum_{|n|\sim N_{(1)}}|a_{(1)}(n)|^2\sim N_{(1)}^{-2s}$, $a_{(j)}(n)=\phi(n)$ or $\sum_{|n|\sim N_{(j)}}|a_{(j)}(n)|^2\lesssim N_{(j)}^{-2s}$. Moreover, at least one of $a_{(2)}(n), a_{(3)}(n)$ is of the form $\phi(n)$.
%%%%%%%%%%%%%%
\subsection{Estimate of low modulation cases:}

Using the fact that $\kappa\in \mathcal{S}(\R)$, we observe that modulo an error of $C_L(N_{(1)})^{-L}$, for any $L\in\mathbb{N}$, we may reduce the estimate to the following expression \footnote{In the situation where $\Phi(\ov{n})\in\Z$, namely $\alpha=2$, we can simply reduce the constraint by $\Phi(\ov{n})=\mu$. However, for $\alpha<2$, the values of $\Phi(\ov{n})$ maybe dense in an interval, and this will be responsible for the loss of derivatives when we perform the counting argument.}. 
\begin{equation}\label{eq:lowmodulationreduction}
T^{\epsilon}N_{(1)}^{s+\epsilon}\sup_{|\mu|\lesssim K}\Big(\sum_{|n|\lesssim N_{(1)}}\Big|\sum_{\substack{(n_1,n_2,n_3)\in\Gamma(\ov{n})\\
		|\Phi(\ov{n})-\mu|\leq 1 }} a_1(n_1)a_2(n_2)a_3(n_3)\Big|^2 \Big)^{1/2}.
\end{equation}
Now we perform the case-by-case analysis. Denote by
$$ \widetilde{\Phi}(n,n_2,n_3)=|n+n_2-n_3|^{\alpha}-|n_2|^{\alpha}+|n_3|^{\alpha}-|n|^{\alpha}.
$$
\noi
$\bullet$  {\bf Case 1:}
Denote  $b_j(n)=a_j(n)\langle n\rangle^{s}$, if $v_j$ is of type (II).
We first assume that $n_1=n_{(1)}, n_2=n_{(2)}$ and $n_3=n_{(3)}$.
\begin{equation*}
\begin{split}
A:=\{(n,n_2,n_3):& n_3\neq n_2, n_3\neq n, |n_j|\sim N_j, j=2,3; |n+n_2-n_3|\sim N_1;\\ &|\widetilde{\Phi}(n,n_2,n_3)-\mu|\leq 1  \},
\end{split}
\end{equation*}
where $\mu$ can be viewed as a fixed parameter.
Note that $|\phi(n+n_2-n_3)|\lesssim (N_{(1)})^{-\frac{\alpha}{2}+2\epsilon}$ on $A$. Applying Cauchy-Schwarz to the summation over $n_2,n_3$, we obtain that
\begin{equation*}\label{eq:lowmodulation1}
\begin{split}
\eqref{eq:lowmodulationreduction}\lesssim & T^{\epsilon}N_{(1)}^{\left(s-\frac{\alpha}{2}\right)+2\epsilon}N_{(2)}^{-s}N_{(3)}^{-s}\\ \times &\Big[\sum_{|n|\lesssim N_{(1)}}\Big(\sum_{n_2,n_3}|b_2(n_2)|^2\mathbf{1}_{A}(n,n_2,n_3) \Big)
\Big(\sum_{n_2,n_3}|b_3(n_3)|^2\mathbf{1}_{A}(n,n_2,n_3) \Big) \Big]^{1/2}.
\end{split}
\end{equation*}
The second line of the right hand side can be majorized by
$$ \Big[\sum_{|n|\lesssim N_{(1)}}\sum_{n_2,n_3}|b_2(n_2)|^2\mathbf{1}_{A}(n,n_2,n_3) \Big]^{1/2}\cdot \sup_{|n|\lesssim N_{(1)}}\Big(\sum_{n_2,n_3}|b_3(n_3)|^2\mathbf{1}_{A}(n,n_2,n_3) \Big)^{1/2}.
$$ 
Thanks to $N_{(1)}\gg N_{(2)}$, viewing $n_2$ as parameter, for fixed $n,n_3$, 
$$ \Big|\frac{\partial\widetilde{\Phi}}{\partial n_2} \Big|\sim |n+n_2-n_3|^{\alpha-1}\sim N_{(1)}^{\alpha-1}, \textrm{ thus }\sum_{n_2}\mathbf{1}_A(n,n_2,n_3)\lesssim 1.
$$
Thus $\Big(\sum_{n_2,n_3}|b_3(n_3)|^2\mathbf{1}_A(n,n_2,n_3)\Big)^{1/2}\lesssim 1$.
Viewing $n$ as parameter, for fixed $n_2,n_3$,
$$ \Big|\frac{\partial\widetilde{\Phi}}{\partial n}\Big|\sim \Big||n+n_2-n_3|^{\alpha-1}-|n|^{\alpha-1} \Big|\sim |n_2-n_3|N_{(1)}^{\alpha-2},
$$
then
$ \sum_{n}\mathbf{1}_{A}(n,n_2,n_3)\lesssim 1+\frac{(N_{(1)})^{2-\alpha} }{|n_2-n_3| }.
$
Therefore, if $N_{(1)}^{2-\alpha}\gg N_{(2)}$, we obtain that
$$ \Big(\sum_{|n|\lesssim N_{(1)}}\sum_{n_2,n_3}|b_2(n_2)|^2\mathbf{1}_{A}(n,n_2,n_3) \Big)^{1/2}\lesssim N_{(1)} ^{\left(1-\frac{\alpha}{2}\right)+2\epsilon}.
$$
This yields
$$ \eqref{eq:lowmodulationreduction}\lesssim T^{\epsilon}N_{(1)}^{\left(s-\frac{\alpha}{2}+1-\frac{\alpha}{2}\right)+3\epsilon} N_{(2)}^{-s}N_{(3)}^{-s},
$$
which is conclusive, if $s<\alpha-1$. If $N_{(1)}^{2-\alpha}\lesssim N_{(2)}$, we estimate
\begin{equation*}
\begin{split}
&\sum_{|n|\lesssim N_{(1)}}\sum_{n_2,n_3}|b_2(n_2)|^2\mathbf{1}_A(n,n_2,n_3)\\
\leq &\sum_{n_2}|b_2(n_2)|^2\Big[\sum_{n_3:|n_3-n_2|\gtrsim (N_{(1)})^{2-\alpha}}\sum_{n}\mathbf{1}_A(n,n_2,n_3)+\sum_{n_3:|n_3-n_2|\ll  (N_{(1)})^{2-\alpha}}\sum_{n}\mathbf{1}_A(n,n_2,n_3)\Big]\\
\lesssim &N_{(3)}+N_{(1)}^{2-\alpha+\epsilon}.
\end{split}
\end{equation*}
Therefore,
$$ \eqref{eq:lowmodulationreduction}\lesssim T^{\epsilon}\Big[ N_{(1)}^{s-\frac{\alpha}{2}+1-\frac{\alpha}{2}+3\epsilon} N_{(2)}^{-s}N_{(3)}^{-s}+N_{(1)}^{s-\frac{\alpha}{2}+3\epsilon}N_{(2)}^{-s}N_{(3)}^{-s} N_{(3)}^{\frac{1}{2}} \Big],\footnote{\textrm{This bound can not be improved if we perform the Wiener chaos estimate as in \cite{bourgain}, due to the loss in the counting. }}
$$
which can be majorized by $T^{\epsilon}(N_{(1)})^{-\delta(\epsilon)}$, for some $\delta(\epsilon)>0$,  provided that
$ s<\alpha-1$. For the remaining case $n_2=n_{(1)}$, there is no significant difference in the argument.

\noi
$\bullet$ {\bf Case 2:}
Denote $b_j(n)=a_j(n)\langle n\rangle^{s}$, if $v_j$ is of type (II),
where $\mu$ can be viewed as a fixed parameter. The modulation bound is $K_{2}=N_{(1)}^{2(s-\sigma)}$. 
Without loss of generality, we may assume that $n_1=n_{(1)}$ and $a_1(n_1)=\phi(n_1)$. Since $N_{(1)}\gg N_{(2)}$, we must have
$$ |\Phi(\ov{n})|\gtrsim |n_2-n_3||n_2-n_1|N_{(1)}^{2-\alpha}\gtrsim N_{(1)}^{\alpha-1},
$$
where $\Phi(\ov{n})$ is defined in Lemma \ref{bound:resonance}. 
For non-zero contributions, $|\Phi(\ov{n})-\mu|\leq 1$ ,where $|\mu|\lesssim K_{2}$, it holds 
$$ N_{(1)}^{\alpha-1}\lesssim |\Phi(\ov{n})|\leq |\mu|+|\Phi(\ov{n})-\mu|\lesssim N_{(1)}^{2(s-\sigma)}.
$$ 
This constraint is violated since $2(s-\sigma)<\alpha-1$ if $\epsilon>0$ is chosen small enough. This means that all the contributions are zero. The same argument applies to the case where $n_2=n_{(1)}$.

\noi
$\bullet$ {\bf Case 3: } Note that the case where $v_{(2)}, v_{(3)}$ are both of type (I) is already considered in the Case(B). It turns out there that the high-modulation analysis is conclusive. 
Now we assume that $v_{(2)}=v_{(2)}$(II) and $v_{(3)}=v_{(3)}$(I), this is the situation in Case (A), and we have $N_{(3)}\ll N_{(1)}$. In this case, we still have $|\Phi(\ov{n})|\gtrsim N_{(1)}^{\alpha-1}$, and the constraint for the non-zero contributions is
$$ N_{(1)}^{\alpha-1}\lesssim |\Phi(\ov{n})|\leq |\mu|+|\Phi(\ov{n})-\mu|\lesssim N_{(1)}^{16\epsilon},
$$
which is empty for small $\epsilon$. Thus the contributions in this case are all zero. This completes the proof of Proposition \ref{trilinear-N1}. Hence the proof of Proposition \ref{multi-linear} is also completed.

\begin{remarque}{\rm
	There is a room in the reduction to low modulations, but the case when the highest frequency is of type (I) is independent of this reduction, and it leads to the restriction $s<\alpha-1$.  More precisely, the use of the Fourier-Lebesgue space gives $\alpha/2$ regularization, while the degeneration of the curvature of the resonant surface causes a derivative loss of order $1-\frac{\alpha}{2}$. Therefore, we need to impose $s-\frac{\alpha}{2}+\big(1-\frac{\alpha}{2}\big)<0$ ($s$ comes from the fact that we evaluate the nonlinearity in $X^{s,b}$). We emphasize that here, the reason for the restriction $s<\alpha-1$ is different from the same restriction appearing in the next section. } 
\end{remarque}
%%%%%%%%%%%%%%%%%%%%%%%%%%%%%%%%%%%%%%%%%%%%%%%%%%%%%%%%%%%%%%%%%%%%%%%%%%%%%%%%%%%%%%%%
\section{Probabilistic linear and trilinear estimates}
In order to use measure invariance arguments to construct global solutions, we need to prove large deviation estimates for the linear norm $\|\cdot\|_{\mathcal{V}^{q,\epsilon}}$ and the trilinear quantity $\mathcal{W}_{s,\epsilon}(\cdot)$ defined in \eqref{auxillarynorms}. Let us introduce some notations. For 
$M<K\leq \infty$, we set $$z^M_{1,K}(t)=S_{\alpha}(t)\Pi_M^{\perp}\Pi_K
\big( \sum_{n\in\Z} \frac{g_n(\omega)}{[n]^{\frac{\alpha}{2}}}e^{inx}\big)=\sum_{M< |n|\leq K}\frac{g_n(\omega)}{[n]^{\frac{\alpha}{2}}}e^{inx-i|n|^{\alpha}t}.
$$
%%%%%%%%%
\begin{lemme}\label{eq:z3}
	Fix $\eta\in C_c^{\infty}(\R)$ and assume that $1<\alpha<2$, $M_j<K_j\leq \infty$, $j=1,2,3.$ Then for any $s<\alpha-1$, $0<\epsilon\ll 1$, there exist $0<\epsilon_0\ll 1, c>0$, such that for any $\lambda\geq 1$,
	\begin{equation*}\label{z3:largedeviation}
	\begin{split}
	&\mathbb{P}\Big\{\omega: \Big\|\int_0^tS_{\alpha}(t-t')\eta(t')\mathcal{N}\big(z_{1,K_1}^{M_1}, z_{1,K_2}^{M_2}, z_{1,K_3}^{M_3} \big)(t')dt' \Big\|_{X^{s,\frac{1}{2}+\epsilon}}>\max\{M_1,M_2,M_3 \}^{-\epsilon_0}\lambda \Big\}\\ \leq &\exp\big(-c\lambda^{2/3} \big).
	\end{split}
	\end{equation*}
\end{lemme}
\begin{proof}
From Lemma~\ref{inhomo-linear}, we have
	\begin{equation*}
	\begin{split}
	&\Big\|\int_0^tS_{\alpha}(t-t')\eta(t')\mathcal{N}\big(z_{1,K_1}^{M_1},z_{1,K_2}^{M_2},z_{1,K_3}^{M_3} \big) \Big\|_{X^{s,\frac{1}{2}+\epsilon}}\\
	\lesssim & \Big\|\eta(t)\mathcal{N}_0\big(z_{1,K_1}^{M_1},z_{1,K_2}^{M_2},z_{1,K_3}^{M_3} \big)\Big\|_{X^{s,-\frac{1}{2}+\epsilon}}+\Big\|\eta(t)\mathcal{N}_1\big(z_{1,K_1}^{M_1},z_{1,K_2}^{M_2},z_{1,K_3}^{M_3} \big)\Big\|_{X^{s,-\frac{1}{2}+\epsilon}}. 
	\end{split}
	\end{equation*}
	Set
	$$I_{M_j,K_j}:=\{n\in\Z: M_j< |n|\leq K_j  \}.
	$$
	Note that
	\begin{equation*}
	\begin{split}
	&\Big\|\eta(t)\mathcal{N}_0\big(z_{1,K_1}^{M_1},z_{1,K_2}^{M_2},z_{1,K_3}^{M_3} \big)\Big\|_{X^{s,-\frac{1}{2}+\epsilon}}\\=&\Big\|\mathbf{1}_{n\in \cap_{j=1}^3I_{M_j,K_j} }\langle n\rangle^{s}\langle\tau-|n|^{\alpha}\rangle^{-\frac{1}{2}+\epsilon}\widehat{\eta}(\tau-|n|^{\alpha}) \frac{|g_n(\omega)|^2g_n(\omega)}{[n]^{\frac{3\alpha}{2}}}\Big\|_{L_{\tau}^2l_n^2}.
	\end{split}
	\end{equation*}
	By Minkowski's inequality, for $p\geq 2$, we have
	\begin{equation}\label{N0Xsb}
	\begin{split}
	&\Big\|\eta(t)\mathcal{N}_0\left(z_{1,K_1}^{M_1},z_{1,K_2}^{M_2}, z_{1,K_3}^{M_3} \right) \Big\|_{L^p(\Omega;X^{s,-\frac{1}{2}+\epsilon})}\\
	\leq &\Big\|\mathbf{1}_{n\in \cap_{j=1}^3I_{M_j,K_j} }\langle n\rangle^{s}\langle\tau-|n|^{\alpha}\rangle^{-\frac{1}{2}+\epsilon}\widehat{\eta}(\tau-|n|^{\alpha}) \frac{|g_n(\omega)|^2g_n(\omega)}{[n]^{\frac{3\alpha}{2}}}\Big\|_{L_{\tau}^2l_n^2L^p(\Omega)}.
	\end{split}
	\end{equation}
	It follows from the property of Gaussian random variables that 
	\begin{equation*}
	%\label{N0Xsb1}
	\begin{split}
	\textrm{(RHS) of \eqref{N0Xsb}}\lesssim &p^{3/2}\Big\|\mathbf{1}_{n\in\cap_{j=1}^3I_{M_j,N_j} }\langle n\rangle^{s-\frac{3\alpha}{2}}\langle\tau-|n|^{\alpha}\rangle^{-\frac{1}{2}+\epsilon}\widehat{\eta}(\tau-|n|^{\alpha}) \Big\|_{L_{\tau}^2l_n^2} \\
	\lesssim &p^{3/2}\max\{M_1,M_2,M_3 \}^{s+\frac{1}{2}-\frac{3\alpha}{2} }\lesssim p^{3/2}\max\{M_1,M_2,M_3\}^{-\big(\frac{3\alpha}{2}-s-\frac{1}{2} \big)},
	\end{split}
	\end{equation*}
	in which the index is negative. Recall the notation
	$$ \Gamma(\ov{n}):=\{(n_1,n_2,n_3):n=n_1-n_2+n_3, n_2\neq n_1, n_2\neq n_3 \}.
	$$
	Similarly, applying Minkowski's inequality and the Wiener chaos estimate of Lemma~\ref{Wiener-Chaos}, we have
	\begin{equation}\label{N1Xsb}
	\begin{split}
	&\Big\|\eta(t)\mathcal{N}_1\big(z_{1,N_1}^{K_1}, z_{1,N_2}^{K_2}, z_{1,N_3}^{K_3} \big) \Big\|_{L^p(\Omega;X^{s,-\frac{1}{2}+\epsilon} )}^2\\
	\lesssim & p^{3}\Big\| \langle\tau-|n|^{\alpha}\rangle^{-\frac{1}{2}+\epsilon}\langle n\rangle^s\sum_{\substack{(n_1,n_2,n_3)\in \Gamma(\ov{n})\\
			n_j\in I_{M_j,K_j},j=1,2,3 } }\frac{g_{n_1}(\omega)\ov{g}_{n_2}(\omega)g_{n_3}(\omega) }{[n_1]^{\frac{\alpha}{2}}
		[n_2]^{\frac{\alpha}{2}}
		[n_3]^{\frac{\alpha}{2}} }\widehat{\eta}(\tau-|n|^{\alpha}-\Phi(\ov{n})) \Big\|_{L_{\tau}^2l_n^2L^2(\Omega)}^2.
	\end{split}
	\end{equation}
	For fixed $n$, using independence, we have
	\begin{equation*}
	\begin{split}
	&\Big\|\sum_{\substack{(n_1,n_2,n_3)\in \Gamma(\ov{n})\\
			n_j\in I_{M_j,K_j},j=1,2,3 } }\frac{g_{n_1}(\omega)\ov{g}_{n_2}(\omega)g_{n_3}(\omega) }{[n_1]^{\frac{\alpha}{2}}
		[n_2]^{\frac{\alpha}{2}}
		[n_3]^{\frac{\alpha}{2}} }\widehat{\eta}(\tau-|n|^{\alpha}-\Phi(\ov{n})) \Big\|_{L^2(\Omega)}^2\\
	\lesssim &\sum_{\substack{(n_1,n_2,n_3)\in \Gamma(\ov{n})\\
			n_j\in I_{M_j,K_j},j=1,2,3 } }
	\frac{|\widehat{\eta}(\tau-|n|^{\alpha}-\Phi(\ov{n})) |^2 }{\langle n_1\rangle^{\alpha}
		\langle n_2\rangle^{\alpha}
		\langle n_3\rangle^{\alpha} }.
	\end{split}
	\end{equation*}
	Therefore, 
	\begin{equation*}
	\begin{split}
	\textrm{(RHS) of \eqref{N1Xsb}}\lesssim & p^3\int_{\R}\sum_{n}\sum_{\substack{(n_1,n_2,n_3)\in \Gamma(\ov{n})\\
			n_j\in I_{M_j,K_j},j=1,2,3 } }\frac{\langle n\rangle^{2s}\langle\tau-|n|^{\alpha}\rangle^{-1+2\epsilon}|\widehat{\eta}(\tau-|n|^{\alpha}-\Phi(\ov{n})) |^2 }{\langle n_1\rangle^{\alpha}
		\langle n_2\rangle^{\alpha}
		\langle n_3\rangle^{\alpha} }d\tau.
	\end{split}
	\end{equation*}
	Since $|\widehat{\eta}(\tau)|\leq C_L\langle\tau\rangle^{-L}$ for any $L\in\mathbb{N}$, applying Lemma \ref{convolution}, we have
	\begin{equation*}
	\begin{split}
	\textrm{(RHS) of \eqref{N1Xsb}}\lesssim p^3J,\quad J:=\sum_{n}\sum_{\substack{(n_1,n_2,n_3)\in \Gamma(\ov{n}) \\
			n_j\in I_{M_j,K_j},j=1,2,3 } }
	\frac{\langle n\rangle^{2s} }{\langle n_1\rangle^{\alpha}
		\langle n_2\rangle^{\alpha} 
		\langle n_3\rangle^{\alpha}\langle\Phi(\ov{n})\rangle^{1-2\epsilon}  }.
	\end{split}
	\end{equation*}
	We decompose the summation into dyadic pieces $|n_j|\sim N_j$ where $M_j/2\leq N_j\leq 2K_j$ for $j=1,2,3$. We write
	$$ J=\sum_{N_1,N_2,N_3}J_{N_1,N_2,N_3}.
	$$
	Denote by $N_{(1)}\geq N_{(2)}\geq N_{(3)}$ the non-increasing order of $N_1,N_2,N_3$. Recall that from Lemma \ref{bound:resonance}, $|\Phi(\ov{n})|\gtrsim |n_1-n_2||n_2-n_3|N_{(1)}^{\alpha-2}$.
	
	If $N_1\sim N_2\sim N_3$, we have
	\begin{equation}\label{highhighhigh}
	\begin{split}
	J_{N_1,N_2,N_3}\lesssim &N_{(1)}^{2s-3\alpha+(2-\alpha)(1-2\epsilon)}\sum_{\substack{n_2\neq n_1,n_3\\
			|n_j|\sim N_j,j=1,2,3 }}\frac{1}{\langle n_1-n_2\rangle^{1-2\epsilon}\langle n_2-n_3\rangle^{1-2\epsilon} }\\
	\lesssim & N_{(1)}^{2s+3-4\alpha+2\alpha\epsilon}.
	\end{split}
	\end{equation}
	If $N_{(1)}\sim N_{(2)}\gg N_{(3)}$, we have
	\begin{equation}\label{hhl}
	\begin{split}
	J_{N_1,N_2,N_3}\lesssim &N_{(1)}^{2s-2\alpha+(2-\alpha)(1-2\epsilon)}N_{(3)}^{-\alpha}\sum_{\substack{n_2\neq n_1,n_3\\
			|n_j|\sim N_j,j=1,2,3 } }\frac{1}{\langle n_1-n_2\rangle^{1-2\epsilon}\langle n_2-n_3\rangle^{1-2\epsilon} }\\
	\lesssim &N_{(1)}^{2s-2\alpha+(2-\alpha)(1-2\epsilon)+4\epsilon}.
	\end{split}
	\end{equation} 
	The worst case is $N_{(1)}\gg N_{(2)}\geq N_{(3)}$, saying\footnote{Other cases are similar or better. } $N_1\sim N_{(1)}$, $|\Phi(\ov{n})|\gtrsim N_{(1)}^{\alpha-1}|n_2-n_3|$, thus
	\begin{equation}\label{hll}
	\begin{split}
	J_{N_1,N_2,N_3}\lesssim & N_{(1)}^{2s-\alpha-(\alpha-1)(1-\epsilon)}\cdot N_{(2)}^{-\alpha}N_{(3)}^{-\alpha}\sum_{\substack{n_2\neq n_1,n_3\\
			|n_j|\sim N_j, j=1,2,3 } } \frac{1}{\langle n_2-n_3\rangle^{1-2\epsilon}}\\
	\lesssim &N_{(1)}^{2s-2(\alpha-1)+(\alpha-1)\epsilon }\cdot N_{(2)}^{-\alpha+2\epsilon}N_{(3)}^{1-\alpha}.
	\end{split}
	\end{equation} 
	If $s<\alpha-1$, we may choose $\epsilon>0$ such that $s<\alpha-1-\epsilon$. 
	
	To estimate $J$, we write
	$$ J=\sum_{N_1\sim N_2\sim N_3} J_{N_1,N_2,N_3}+\sum_{N_{(1)}\gg N_{(2)}\geq N_{(3)}}J_{N_1,N_2,N_3}+
	\sum_{N_{(1)}\sim N_{(2)}\gg N_{(3)}}J_{N_1,N_2,N_3}.
	$$
	For the summation over dyadic integers satisfying $N_1\sim N_2\sim N_3\sim N_{(1)}$, the non-zero contributions satisfy $N_{(1)}\gtrsim \max\{M_1,M_2,M_3 \}$, thus the dyadic summation over $N_1\sim N_2\sim N_3$ is bounded by $\displaystyle{ \max\{ M_1,M_2,M_3 \}^{2s+3-4\alpha+2\alpha\epsilon} }$. For the summation over dyadic integers satisfying $N_{(1)}\gg N_{(2)}\geq N_{(3)}$, the non-zero contributions satisfy
	$N_{(1)}\geq \max\{M_1,M_2,M_3 \} $, hence the summation can be bounded by $\displaystyle{\max\{M_1,M_2,M_3 \}^{(\alpha-3)\epsilon} }$. From the constraint of $s$, we have
	$$  J\lesssim \max\{M_1,M_2,M_3 \}^{-(3-\alpha)\epsilon}.
	$$
	The rest argument follows from an application of Chebyshev's inequality, as in the proof of Lemma~\ref{proba-Strichartz}.
\end{proof}
%%%
\begin{remarque}\label{fut}
{\rm	
From \eqref{highhighhigh},\eqref{hhl} and \eqref{hll}, we see that the constraint $s<\alpha-1$ comes only from the high-low-low frequency interactions. 
The other cases give $s<4\alpha-3$ and $s<3\alpha-2$ respectively. 
In these other cases the condition $s\geq \frac{1}{2}-\frac{\alpha}{4}$ gives the full range $\alpha>1$. 
The situation therefore reminds the impressive recent work \cite{Deng2} and as a consequence  we conjecture that Theorem~\ref{thm6} and Theorem~\ref{thm5} can be extended to $\alpha>1$, and even to some values of $\alpha\leq1$ after suitable renormalizations.  
}
\end{remarque}

\begin{corollaire}\label{convergence:W}
Assume that $1<\alpha<2$, then for any $s<\alpha-1$, there exist $\epsilon_0>0, 0<\epsilon\ll 1,c>0$, such that for any $\lambda\geq 1$, $i_1,i_2,i_3\in\{0,1\}$ and $M\in\N$, $K\in\N\cup\{+\infty\}$, $M\leq K$ 
\begin{align*}
\mathbb{P}\Big\{\omega: M^{\epsilon_0(i_1+i_2+i_3)}\mathcal{W}_{s,\epsilon}\big(\Pi_K(\Pi_M^{\perp})^{i_1}\phi^{\omega}, \big(\Pi_K(\Pi_M^{\perp})^{i_2}\phi^{\omega},\big(\Pi_K(\Pi_M^{\perp})^{i_3}\phi^{\omega}\big) >\lambda \Big\}\leq e^{-c\lambda^{2/3}}.
\end{align*}
\end{corollaire}

\begin{proof}
 Denote by $\phi_{i_j}:=\Pi_K(\Pi_M^{\perp})^{i_j}\phi^{\omega}$. From the Wiener chaos estimates, 
 it is sufficient to obtain the following estimate for large $p<\infty$:
\begin{align*}
\Big\| \sum_{l\in\Z}\langle l\rangle^{-2} \Big\| \int_0^t\chi_0(t)S_{\alpha}(t-t')\mathcal{N}\big(z_{1,K}^{M_1},z_{1,K}^{M_2},
z_{1,K}^{M_3} \big)(t'+l)dt'
\Big\|_{X^{s,\frac{1}{2}+2\epsilon}}
\Big\|_{L^p(\Omega)}\leq CM^{-\epsilon_0}p^{3/2},
\end{align*}
where $M_j=M$ if $i_j=1$ and $M_j=0$ if $i_j=0$.
Since $\sum_{l\in \Z}\langle l\rangle^{-2}<\infty$, it is sufficient to show that for any $l\in\Z$
\begin{align}\label{!!}
 \Big\| \int_0^t\chi_0(t)S_{\alpha}(t-t')\mathcal{N}\big(z_{1,K}^{M_1},z_{1,K}^{M_2},
 z_{1,K}^{M_3} \big)(t'+l)dt'
 \Big\|_{L^p\big(\Omega;X^{s,\frac{1}{2}+2\epsilon}\big)} \leq CM^{-\epsilon_0}p^{3/2}.
\end{align}
Since $S_{\alpha}(l)\phi^{\omega}$ has the same law as $\phi^{\omega}$, we obtain \eqref{!!} from the same proof of Lemma \ref{z3:largedeviation}. This completes the proof of Corollary \ref{convergence:W}.

\end{proof}
\begin{lemme}\label{linear-convergence}
	Assume that $1<\alpha<2$ and $M<K\leq \infty$. Then for any $t_0\in\R$, any $\epsilon>0$,  there exist $2\leq q<\infty$, $0<\epsilon_0\ll 1$, such that for all $\lambda\geq 1$,
	\begin{equation*}
	\begin{split}
	\mathbb{P}\Big\{\omega:\big\|z_{1,K}^M\big\|_{\mathcal{V}^{q,\epsilon}}> M^{-\epsilon_0}\lambda  \Big\}<e^{-c\lambda^2},
	\end{split}
	\end{equation*}
	where $c>0$ is some uniform constant.
\end{lemme}
\begin{proof}
Denote by $\sigma_0=\frac{\alpha-1}{2}-\frac{\epsilon}{2}, \sigma_1=\frac{\alpha}{2}-\frac{2\epsilon}{3}, r=\frac{1}{\epsilon}$. From Wiener chaos estimates and by the same argument as in the proof of Corollary~\ref{convergence:W}, it would be sufficient to show that for all large $p<\infty$,	
\begin{align*}
\Big\|\|\chi_0(t)z_{1,K}^M\|_{L_t^q\mathcal{F}L^{\sigma_1,2r}\cap L_t^qW_x^{\sigma_0,r} } \Big\|_{L^p(\Omega)}\leq CM^{-\epsilon_0}\sqrt{p}. 
\end{align*}
 We first deal with the Fourier-Lebesgue norm $\mathcal{F}L^{\sigma_1,2r}$. Note that $S_{\alpha}(t)$ keeps the Fourier-Lebesgue norm invariant, it suffices to show that for large $p$,
 $$ \big\|\mathbf{1}_{M\leq |n|\leq K}\langle n\rangle^{\sigma_1-\frac{\alpha}{2}}g_n(\omega)\big\|_{L^p(\Omega;l_n^{2r} )}\leq CM^{-\epsilon_0}\sqrt{p}.
 $$
 Note that $\big(\frac{\alpha}{2}-\sigma_1\big)2r=\frac{4}{3}>1$, take $p\geq 2r$, from Minkowski, we have
 $$ \big\|\mathbf{1}_{M\leq |n|\leq K}\langle n\rangle^{\sigma_1-\frac{\alpha}{2}}g_n(\omega)\big\|_{L^p(\Omega;l_n^{2r} )}\leq \big\|\mathbf{1}_{M\leq |n|\leq K}\langle n\rangle^{\sigma_1-\frac{\alpha}{2}}g_n(\omega)\big\|_{l_n^{2r}L^p(\Omega)}.
 $$
 From a property of the Gaussian random variables, we have
 \begin{equation*}
 \begin{split}
 \big\|\mathbf{1}_{M\leq |n|\leq K}\langle n\rangle^{\sigma_1-\frac{\alpha}{2}}g_n(\omega)\big\|_{l_n^{2r}L^p(\Omega)}\leq & C\sqrt{p}\|\mathbf{1}_{|n|\geq M}\langle n\rangle^{\sigma_1-2\alpha}\|_{l_n^{2r}}\leq CM^{-\frac{1}{6r}}\sqrt{p}.
 \end{split}
 \end{equation*}
Next we deal with the Sobolev norm $L_t^qW_x^{\sigma_0,r}$. Again, for $p\geq 2r, p\geq q$, we have
\begin{align*}
\Big\|\|\chi_0(t)z_{1,K}^M\|_{ L_t^qW_x^{\sigma_0,r} } \Big\|_{L^p(\Omega)}=& \Big\|
\Big\|\chi_0(t)\sum_{M\leq |n|\leq K} \frac{\langle n\rangle^{\sigma_0} g_n(\omega) e^{inx+i|n|^{\alpha}t}}{[n]^{\frac{\alpha}{2}} }
\Big\|_{L_t^qL_x^{r}}
\Big\|_{L^p(\Omega)}\\
\leq &\Big\|
\Big\|\chi_0(t)\sum_{M\leq |n|\leq K} \frac{\langle n\rangle^{\sigma_0} g_n(\omega) e^{inx+i|n|^{\alpha}t}}{[n]^{\frac{\alpha}{2}} }
\Big\|_{L^p(\Omega)}
\Big\|_{L_t^qL_x^{r}}.
\end{align*}
By Wiener chaos estimate, there exists $C>0$, such that for any $(t,x)$,
\begin{equation*}
\begin{split}
\Big\|\chi_0(t)\sum_{M\leq |n|\leq K} \frac{\langle n\rangle^{\sigma_0} g_n(\omega) e^{inx+i|n|^{\alpha}t}}{[n]^{\frac{\alpha}{2}} }
\Big\|_{L^p(\Omega)}\leq &C\sqrt{q} \|\mathbf{1}_{M\leq |n|\leq K}\langle n\rangle^{\sigma_0-\frac{\alpha}{2}} \|_{l_n^2}\\
\leq & CM^{-\frac{\epsilon}{2}}\sqrt{p}.
\end{split}
\end{equation*}
The proof of Lemma \ref{linear-convergence} is now complete.
\end{proof}
%%%%%%%%%%%%%%%%%%%%%%%%%%%%%%%%%%%%%%%%%%%%%%%%%%%%%%%%%%%%%%%%%%%%%%%%%%%%%%%%%%%%%%%%%%%%%%%%%%%%%%%%%%%%%%%%%%%%%%%%%%%%%%%%%%%%%%%%%%%%%%%%%

\section{Global well-posedness and flow property when $\frac{6}{5}<\alpha<2$}

\subsection{Enhanced local convergence}
Throughout this section, we fix the small parameter $\epsilon>0,$ and the large parameter $q<\infty$ 
as required in the previous sections. We also fix the constants
$$
\frac{6}{5}<\alpha<2,\quad  \sigma=\frac{\alpha-1}{2}-\epsilon,\quad \frac{1}{2}-\frac{\alpha}{4}<s<\alpha-1.
$$
	We remark that in contrast with  previous situations (as for instance in  \cite{BT-IMRN},\cite{Sun-Tz}), here the nonlinear evolution part though more regular lives in different function spaces which may not be embedded into the function space of the linear evolution part. This causes difficulties to construct the invariant data set. To overcome this difficulty, we define the summed space $\mathcal{Y}^{s,\epsilon}:=\mathcal{V}^{q,\epsilon}+H^s(\T)$ via the norm
		$$ \|u\|_{\mathcal{V}^{q,\epsilon}+H^s}:=\inf\{\|u_1\|_{\mathcal{V}^{q,\epsilon}}+\|u_2\|_{H^s(\T)}: \textrm{ if } u=u_1+u_2 \text{ for some } u_1\in \mathcal{V}^{q,\epsilon},u_2\in H^s(\T) \}.
		$$	
		Since $\mathcal{V}^{q,\epsilon}$ and $H^s(\T)$ are continuously embedded into $L^2(\T)$, from Lemma 2.3.1 of \cite{Bergh}, $(\mathcal{V}^{q,\epsilon}+H^s,\|\cdot\|_{\mathcal{V}^{q,\epsilon}+H^s})$ is a normed space. We introduce the summed space structure, since the gauged linear evolution part should be measured by $\mathcal{V}^{q,\epsilon}$ norm and the quantity $\mathcal{W}_{s,\epsilon}$, while the nonlinear evolution should be measured by $H^s$ norm. The analysis in this section is somewhat soft and topological.
\\

We need to introduce some notations. For functions $f_1,f_2,f_3$, we extend the nonlinear quantity $\mathcal{W}_{s,\epsilon}(\cdot)$ to the following canonical trilinear form:
		\begin{align*}
		\mathcal{W}_{s,\epsilon}(f_1,f_2,f_3):=\sum_{l\in\Z}\langle l\rangle^{-2}\big\| \chi_0(t)\mathcal{N}\big(\big(S_{\alpha}(t+l)f_{1},S_{\alpha}(t+l)f_{2},S_{\alpha}(t+l)f_{3} \big)
		\big\|_{X^{s,-\frac{1}{2}+2\epsilon}}.
		\end{align*}
	Note that for any two fixed entries, $\mathcal{W}_{s,\epsilon}(\cdot,f_2,f_3), \mathcal{W}_{s,\epsilon}(f_1,\cdot,f_3)$ satisfy the triangle inequality.
			Given a finite set $\mathcal{J}$ of functions, the notation
		$$ \sum_{f_{j}\in\mathcal{J} } \mathcal{W}_{s,\epsilon}\big(f_1,f_2,f_3 \big)
		$$
		means to sum over all possible $f_1,f_2,f_3\in\mathcal{J}$ of $\mathcal{W}_{s,\epsilon}(f_1,f_2,f_3)$. For the projector $\Pi_N^{\perp}$, we denote by 
		$$ \big(\Pi_{N}^{\perp}\big)^{j}=\Pi_N^{\perp},\text{ if }j=1;\quad \big(\Pi_{N}^{\perp}\big)^{j}=\mathrm{Id},\text{ if }j=0.
		$$
We will make use of the following simple quasi-invariance property.
\begin{lemme}[Quasi-invariance]\label{quasi-invarianceW}
	There exists a constant $A_1>0$, such that for all $|t_0|\leq \frac{1}{2}$ and all $\phi,\phi_1,\phi_2,\phi_3$ 
	$$  \mathcal{W}_{s,\epsilon}\big(S_{\alpha}(t_0)\phi_1,S_{\alpha}(t_0)\phi_2,S_{\alpha}(t_0)\phi_3 \big)\leq A_1 \mathcal{W}_{s,\epsilon}(\phi_1,\phi_2,\phi_3),\quad \|S_{\alpha}(t_0)\phi\|_{\mathcal{V}^{q,\epsilon}}\leq A_1\|\phi\|_{\mathcal{V}^{q,\epsilon}}.
	$$ 
\end{lemme}
\begin{proof}
	From the support property of $\chi_0$, we have for any $t\in\R$,  $|t_0|\leq \frac{1}{2}$,
	$$\chi_0(t-t_0)=\chi_0(t-t_0)\sum_{|m|\leq 3}\chi_0(t-m).$$
	Note that the $X^{s,b}$ norm is invariant under the time-shifting, from Lemma \ref{time-localization}, we have
	\begin{align*}
	& \big\|\chi_0(t)\mathcal{N}\big(S_{\alpha}(t_0+t+l)\phi_1,S_{\alpha}(t_0+t+l)\phi_2,S_{\alpha}(t_0+t+l)\phi_3 \big) \big\|_{X^{s,b}}\\ \leq &C\sum_{|m|\leq 3}\big\|\chi_0(t-m)\mathcal{N}\big(S_{\alpha}(t+l)\phi_1,S_{\alpha}(t+l)\phi_2,S_{\alpha}(t+l)\phi_3 \big) \big\|_{X^{s,b}}\\
	\leq &C\sum_{|m|\leq 3}\|\chi_0(t)\mathcal{N}\big(S_{\alpha}(t+l+m)\phi_1,S_{\alpha}(t+l+m)\phi_2,S_{\alpha}(t+l+m)\phi_3 \big) \big\|_{X^{s,b}}.
	\end{align*}
	Multiplying by $\langle l\rangle^{-2}$ and sum over $l\in\Z$, we obtain the first inequality. The second one follows from a similar argument, and we omit the details. This completes the proof of Lemma~\ref{quasi-invarianceW}.
\end{proof}
%%%%
\begin{lemme}\label{multi-linearW}
For all $f_1,f_2,f_3\in \widetilde{\mathcal{V}}^{q,\epsilon}$ and $g_1,g_2,g_3\in H^{s,}$, the following estimates hold
	\begin{equation*}
	\begin{split}
	&(\mathrm{1})\quad  \mathcal{W}_{s,\epsilon}(f_1,g_2,g_3)\lesssim \|f_1\|_{\widetilde{\mathcal{V}}^{q,\epsilon}}\|g_2\|_{H^{s}}\|h_3\|_{H^{s}}\,\, ,
	\\
	&(\mathrm{2})\quad  \mathcal{W}_{s,\epsilon}(g_1,f_2,g_3)\lesssim \|g_1\|_{H^s}\|f_2\|_{\widetilde{\mathcal{V}}^{q,\epsilon}}\|g_3\|_{H^s}\,\, ,\\
	&(\mathrm{3})\quad  \mathcal{W}_{s,\epsilon}(g_1,g_2,f_3)\lesssim \|g_1\|_{H^s}\|g_2\|_{H^s}\|f_3\|_{\widetilde{\mathcal{V}}^{q,\epsilon}}\,\, ,\\
	&(\mathrm{4})\quad  \mathcal{W}_{s,\epsilon}(f_1,g_2,f_3)\lesssim \|f_1\|_{\widetilde{\mathcal{V}}^{q,\epsilon}}\|g_2\|_{H^s}\|f_3\|_{\widetilde{\mathcal{V}}^{q,\epsilon}}\,\, ,\\
	&(\mathrm{5})\quad  \mathcal{W}_{s,\epsilon}(f_1,f_2,g_3)\lesssim \|f_1\|_{\widetilde{\mathcal{V}}^{q,\epsilon}}\|f_2\|_{\widetilde{\mathcal{V}}^{q,\epsilon}}\|g_3\|_{H^s}\,\, ,\\
	&(\mathrm{6})\quad  \mathcal{W}_{s,\epsilon}(g_1,f_2,f_3)\lesssim \|g_1\|_{H^s}\|f_2\|_{\widetilde{\mathcal{V}}^{q,\epsilon}}\|f_3\|_{\widetilde{\mathcal{V}}^{q,\epsilon}}\,\, ,\\
	&(\mathrm{7})\quad  \mathcal{W}_{s,\epsilon}(g_1,g_2,g_3)\lesssim \|g_1\|_{H^s}\|g_2\|_{H^s}\|g_3\|_{H^s}\,\, .
	\end{split}
	\end{equation*}
\end{lemme}
\begin{proof}
Since the proof of each inequality is an application of the corresponding inequality in Proposition \ref{multi-linear} and Corollary \ref{Trilinear}, we only prove (1). Take another cutoff $\widetilde{\chi}_0(t)$ such that $\widetilde{\chi}_0(t)=1$ on the support of $\chi_0$. Thus for every $l\in\Z$, from Lemma \ref{inhomo-linear} and (1) of Proposition \ref{multi-linear}, we estimate
\begin{align*}
&\Big\|\chi_0(t)\int_0^tS_{\alpha}(t-t')\mathcal{N}\big(S_{\alpha}(t'+l)f_1,S_{\alpha}(t'+l)g_2,S_{\alpha}(t'+l)g_3 \big)dt'\Big\|_{X^{s,\frac{1}{2}+2\epsilon}} \\=&\Big\|\chi_0(t)\int_0^tS_{\alpha}(t-t')\mathcal{N}\big(\widetilde{\chi}_0^3(t')S_{\alpha}(t'+l)f_1,S_{\alpha}(t'+l)g_2,S_{\alpha}(t'+l)g_3 \big)dt'\Big\|_{X^{s,\frac{1}{2}+2\epsilon}}\\
\lesssim & \|S_{\alpha}(l)f_1\|_{\mathcal{Z}^{q,\epsilon}} \|\widetilde{\chi}_0(t)S_{\alpha}(t+l)g_2\|_{X^{s,\frac{1}{2}+2\epsilon}} 
\|\widetilde{\chi}_0(t)S_{\alpha}(t+l)g_3\|_{X^{s,\frac{1}{2}+2\epsilon}}\\
\lesssim &\|\chi_l(t)S_{\alpha}(t)f_1\|_{L_t^{\infty}\mathcal{F}L^{\frac{\alpha}{2}-\epsilon,\frac{2}{\epsilon}}\cap L_t^qW_x^{\frac{\alpha-1}{2}-\epsilon,\frac{1}{\epsilon}} }\|g_2\|_{H^s}\|g_3\|_{H^s}.
\end{align*}
To complete the proof of Lemma \ref{multi-linearW}, we multiply by $\langle l\rangle^{-2}$ and sum over $l\in\Z$.  
\end{proof}
%	\begin{lemme}\label{nonlinear-convergence}
	%	Let $(f_k)_{k\in\N}\subset \mathcal{V}^{q,\epsilon}$ be a bounded sequence and $f\in \mathcal{V}^{q,\epsilon}$. Let $g_k\rightarrow g $ in $\subset H^s(\T)$. Assume that
%		$$ \lim_{k\rightarrow\infty}\mathcal{W}_{s,\epsilon}(f_k-f)=0.
	%	$$
	%	Then
	%	$$ \lim_{k\rightarrow\infty}\mathcal{W}_{s,\epsilon}\big((f_k-f)+(g_k-g) \big)=0.
%		$$
%	\end{lemme}
%	\begin{proof}
%		TO DO.
%	\end{proof}

For $\phi,\psi$, we define the pseudo-distance
\begin{align}\label{pseuso-distance}
\mathbf{d}(\phi,\psi):=\sum_{f_2,f_3\in \{\phi,\psi \} } 2\mathcal{W}_{s,\epsilon}\big(\phi-\psi,f_2,f_3\big)+\sum_{f_1,f_3\in\{\phi,\psi\} } \mathcal{W}_{s,\epsilon}\big(f_1,\phi-\psi,f_3\big).
\end{align}
Note that $\mathbf{d}(\phi,\psi)=\mathbf{d}(\psi,\phi)$.
For $i_1,i_2,i_3\in\{0,1 \}$, we define
\begin{align}\label{pseudo-tail}
\Gamma_{N,s,\epsilon}^{i_1,i_2,i_3}(f_1,f_2,f_3):=\mathcal{W}_{s,\epsilon}\big(\big(\Pi_N^{\perp}\big)^{i_1}f_1,\big(\big(\Pi_N^{\perp}\big)^{i_2}f_2,
\big(\big(\Pi_N^{\perp}\big)^{i_3}f_3  \big).
\end{align} 
We denote by $\Gamma_{N,s,\epsilon}^{i_1,i_2,i_3}(f):=\Gamma_{N,s,\epsilon}^{i_1,i_2,i_3}(f,f,f)$. For any two fixed entries, $\Gamma_{N,s,\epsilon}^{i_1,i_2,i_3}$ satisfies the triangle inequality for the third entry. We will also need the following lemma.
\begin{lemme}\label{abstract}
	Let $V,W$ be two normed spaces. Let $(\phi_k)_{k\in\N}\subset V+W$ be a bounded sequence and $\phi\in V+W$. Assume that $\phi_k\rightarrow \phi$ in $V+W$. Then there exist subsequences $(\varphi_k)_{k\in\N}\subset V$ and $(\psi_k)_{k\in\N}\subset W$, $\varphi\in V, \psi\in W$, satisfying
	\begin{align*}
	&\limsup_{k\rightarrow\infty}\big(\|\varphi_k\|_V +\|\psi_k\|_W \big)\leq \|\phi\|_{V+W}+1,\\
	& \|\varphi\|_V +\|\psi\|_W\leq \|\phi\|_{V+W}+1,
	\end{align*}
	such that $\varphi_k\rightarrow \varphi$ in $V$ and $\psi_k\rightarrow \psi$ in $W$.
\end{lemme}		
\begin{proof}
	By definition, for any $k$, there exist $f_k\in V, g_k\in W$, such that $\phi_k-\phi=f_k+g_k$, $f_k\rightarrow 0$ in $V$ and $g_k\rightarrow 0$ in $W$. There exist $\varphi\in V,\psi\in W$, such that
	$$ \|\varphi\|_V+\|\psi\|_W\leq \|\phi\|_{V+W}+1.
	$$ 
	Let $\varphi_k=\varphi+f_k$ and $\psi_k=\psi+g_k$, then
	$$ \|\varphi_k\|_V+\|\psi_k\|_{W}\leq \|\varphi\|_V+\|\psi\|_W+\|f_k\|_V+\|g_k\|_W\leq \|\phi\|_{V+W}+1+o(1)
	$$
	as $k\rightarrow\infty$. This completes the proof of Lemma \ref{abstract}.
\end{proof}		
The key step to construct the invariant set and the global dynamics is the following enhanced local convergence result. 
\begin{proposition}[Enhanced local convergence]\label{enhanced-localconvergence}
	Assume that $\alpha,q,\epsilon$ be the numerical constants as in Proposition \ref{LWP-main}. Let $(\phi_{k})\subset \mathcal{V}^{q,\epsilon}+H^s$, $\phi\in \mathcal{V}^{q,\epsilon}+H^s$ satisfying
	$$ \|\phi_{k}\|_{\mathcal{V}^{q,\epsilon}+H^s}+\|\phi\|_{\mathcal{V}^{q,\epsilon}+H^s}\leq R,\quad \lim_{k\rightarrow\infty} \|\phi_k-\phi\|_{\mathcal{V}^{q,\epsilon}+H^s}=0.
	$$
	Let $N_k\rightarrow\infty$ be a subsequence of $\N$. For $\mathcal{J}_k=\{\phi_k,\phi \}$, assume that
	\begin{align}\label{ugly1}
	\sum_{\substack{f_{j}\in\mathcal{J}_k\\
			i_1,i_2,i_3\in\{0,1\} } } \Gamma_{N_k,s,\epsilon}^{i_1,i_2,i_3}(f_1,f_2,f_3)\leq R^3.
	\end{align}
	Assume moreover that
	\begin{align}\label{ugly2}
	\lim_{k\rightarrow\infty}\sum_{\substack{f_{j}\in\mathcal{J}_k\\
			i_1,i_2,i_3\in\{0,1\}\\
			i_1+i_2+i_3>0 } } \Gamma_{N_k,s,\epsilon}^{i_1,i_2,i_3}(f_1,f_2,f_3)=0,
	\end{align}	
	and
	\begin{align}\label{ugly3}
	\lim_{k\rightarrow\infty}\mathbf{d}(\phi_k,\phi)=0.
	\end{align}
	Then there exist $c>0, \kappa>0$, such that for all $t\in [-\tau_R,\tau_R]$ with $\tau_R=c(R+2)^{-\kappa}$, we have
	\begin{align}\label{convergence:iterated1}
	\lim_{k\rightarrow\infty}\|\Phi_{N_k}(t)\phi_{k}-\Phi(t)\phi \|_{\mathcal{V}^{q,\epsilon}+H^s}=0.
	\end{align}
	Furthermore, with $\mathcal{J}_{k,t}=\{\Phi_{N_k}(t)\phi_k, \Phi(t)\phi  \}$, we have
%	\begin{align}\label{ugly4}
%	\sum_{ \substack{ f_j\in\mathcal{J}_{k,t}\\
%	i_1,i_2,i_3\in\{0,1\}  } } \Gamma_{N_k,s,\epsilon}^{i_1,i_2,i_3}(f_1,f_2,f_3)\leq 2R,
%	\end{align}
		\begin{align}\label{ugly5}
	\lim_{k\rightarrow\infty}\sum_{ \substack{ f_j\in\mathcal{J}_{k,t}\\
			i_1,i_2,i_3\in\{0,1\}\\ i_1+i_2+i_3>0  } } \Gamma_{N_k,s,\epsilon}^{i_1,i_2,i_3}(f_1,f_2,f_3)=0,
	\end{align}
	and
		\begin{align}\label{ugly6'}
	\lim_{k\rightarrow\infty}\mathbf{d}\big(\Phi_{N_k}(t)\phi_k,\Phi(t)\phi\big)=0.
	\end{align}
\end{proposition}
\begin{remarque}
	As a consequence of \eqref{ugly1} and \eqref{ugly3}, we have $\mathcal{W}_{s,\epsilon}(\phi_k-\phi)\rightarrow 0$. This convergence relation is enough to prove that $\Phi_{N_k}(t)\phi_k-\Phi(t)\phi\rightarrow 0$ in $\mathcal{V}^{q,\epsilon}+H^s$. The closeness of the conditions \eqref{ugly2},\eqref{ugly3} are important for the iteration.
\end{remarque}
\begin{proof}
	Thanks to Lemma \ref{abstract}, there exist sequences $(\phi_{0,k})_{k\in\N}\subset \mathcal{V}^{q,\epsilon}, (r_{0,k})_{k\in\N}\subset H^s(\T)$, and $\phi_0\in \mathcal{V}^{q,\epsilon}, r_0\in H^s$, such that
\begin{align}\label{decomposition}
 \phi_k=\phi_{0,k}+r_{0,k}, \quad \phi=\phi_0+r_0,
\end{align}
	satisfying
	$$ \|\phi_{0,k}\|_{\mathcal{V}^{q,\epsilon}}+\|r_{0,k}\|_{H^s}\leq R+2, \quad \|\phi_0\|_{\mathcal{V}^{q,\epsilon}}+\|r_0\|_{H^s}\leq R+2
	$$
	and
	$$ \lim_{k\rightarrow\infty}\big(\|\phi_{0,k}-\phi_0\|_{\mathcal{V}^{q,\epsilon}}+\|r_{0,k}-r_0\|_{H^s} \big)=0.
	$$
	Moreover, we have
\begin{align}\label{tail-r} \lim_{k\rightarrow\infty}\|\Pi_{N_k}^{\perp}r_0\|_{H^s}=0,\text{ and } \lim_{k\rightarrow\infty}\|\Pi_{N_k}^{\perp}r_{0,k}\|_{H^s}=0
\end{align}
	by writing $\|\Pi_{N_k}r_{0,k}\|_{H^s}\leq \|\Pi_{N_k}r_0\|_{H^s}+\|\Pi_{N_k}^{\perp}(r_0-r_{0,k})\|_{H^s}$.
	Developing the trilinear expression of $\mathcal{W}_{s,\epsilon}(\phi_{0,k}-\phi_0 )=\mathcal{W}_{s,\epsilon}\big( (\phi_k-\phi_0)-(r_{0,k}-r_0) \big)$, from the hypothesis and Lemma \ref{multi-linearW}, we deduce that
	$$ \lim_{k\rightarrow\infty}\mathcal{W}_{s,\epsilon}(\phi_{0,k}-\phi_0)=0.
	$$
	All the hypothesis of Proposition \ref{local-convergence} are satisfied. We thus deduce that for all $t\in [-\tau_R,\tau_R]$, with $\tau_R=c(R+2)^{-\kappa}$,
	\begin{align*}
 \Phi_{N_k}(t)\phi_k=&\underbrace{e^{\frac{it}{\pi}\|\Pi_{N_k}\phi_k\|_{L^2}^2 }\Pi_{N_k}S_{\alpha}(t)\phi_{0,k}+\Pi_{N_k}^{\perp}S_{\alpha}(t)\phi_{0,k}}_{\mathcal{V}^{q,\epsilon} \text{ part }}\\ + &\underbrace{ e^{\frac{it}{\pi}\|\Pi_{N_k}\phi_k\|_{L^2}^2 }\big(\Pi_{N_k}S_{\alpha}(t)r_{0,k}+w_k(t)) \big)+\Pi_{N_k}^{\perp}S_{\alpha}(t)r_{0,k}
	}_{H^s\text{ part } },\\
	\Phi(t)\phi=&\underbrace{e^{\frac{it}{\pi}\|\phi\|_{L^2}^2 }S_{\alpha}(t)\phi_{0}}_{\mathcal{V}^{q,\epsilon} \text{ part }} + \underbrace{ e^{\frac{it}{\pi}\|\phi\|_{L^2}^2 }\big(S_{\alpha}(t)r_{0}+w(t)) \big)
	}_{H^s\text{ part } },
	\end{align*}
	where $w_k(t)\in E_{N_k}$.
	Moreover,
	\begin{align}\label{converge:wk} \lim_{k\rightarrow\infty}\sup_{|t|\leq \tau_R}\|w_k(t)-w(t)\|_{H^s}=0.
	\end{align}
	Denote by $b_k(t)=e^{\frac{it}{\pi}\|\Pi_{N_k}\phi_k\|_{L^2}^2 } $, $b(t)=e^{\frac{it}{\pi}\|\phi\|_{L^2}^2 }$. Clearly, since $\phi_k\rightarrow \phi$ in $L^2(\T)$, $b_k(t)\rightarrow b(t)$ for all $t\in \R$. Taking the difference of $\Phi_{N_k}(t)\phi_k$ and $\Phi(t)\phi$, we have 
	$$ \Phi_{N_k}(t)\phi_k-\Phi(t)\phi=\varphi_k(t)+\psi_k(t),
	$$
	where
	\begin{align*}
	\varphi_k(t)=&\big(b_k(t)-b(t) \big)\Pi_{N_k}S_{\alpha}(t)\phi_{0,k}+\big(1-b(t)\big)\Pi_{N_k}^{\perp}S_{\alpha}(t)\phi_0+\Pi_{N_k}^{\perp}S_{\alpha}(t)\big(\phi_{0,k}-\phi_0\big)\\
	+&b(t)\Pi_{N_k}S_{\alpha}(t)(\phi_{0,k}-\phi_0),\\
	\psi_k(t)=&b_k(t)\big(\Pi_{N_k}S_{\alpha}(t)(r_{0,k}-r_0)+w_k(t)-w(t) \big)+(b_k(t)-b(t))\big(\Pi_{N_k}S_{\alpha}(t)r_0+w(t)\big)\\
	+&(1-b(t))\Pi_{N_k}^{\perp}S_{\alpha}(t)r_0+\Pi_{N_k}^{\perp}S_{\alpha}(t)(r_{0,k}-r_0).	
	\end{align*}
	From \eqref{converge:wk}, we have for all $|t|\leq \tau_R$, 
	$\psi_k(t)\rightarrow 0$ in $H^s$. To show that $\Phi_{N_k}(t)\phi_k$ converges to $\Phi(t)\phi$ in $\mathcal{V}^{q,\epsilon}+H^s$, it will be sufficient to prove that $\varphi_k(t)\rightarrow 0$ in $\mathcal{V}^{q,\epsilon}$ for all $|t|\leq \tau_R$. We note that $\Pi_{N_k}, \Pi_{N_k}^{\perp}$ are uniformly bounded on $\mathcal{V}^{q,\epsilon}$, since they can be represented by Hilbert transformation, up to modulation. Thus from Lemma \ref{quasi-invarianceW} we have 
	$$ \lim_{k\rightarrow\infty}\|\big(b_k(t)-b(t)\big)\Pi_{N_k}S_{\alpha}(t)\phi_{0,k} \|_{\mathcal{V}^{q,\epsilon}}=0.
	$$ 
	Next we prove that $\Pi_{N_k}^{\perp}S_{\alpha}(t)\phi_0$ converges to $0$ in $\mathcal{V}^{q,\epsilon}$. The Fourier-Lebesgue norm $\mathcal{F}L^{\frac{\alpha}{2}-\epsilon,\frac{2}{\epsilon} }$ of $\Pi_{N_k}^{\perp}S_{\alpha}(t)\phi_0$ converges to $0$ can be deduced easily from the fact that $S(t)\phi_0\in\mathcal{F}L^{\frac{\alpha}{2}-\epsilon,\frac{2}{\epsilon}}$. For the Sobolev norm $L_t^qW_x^{\frac{\alpha-1}{2}-\epsilon,\frac{1}{\epsilon}}$, we first observe that for almost every $t'\in\R$, $\Pi_{N_k}^{\perp}S_{\alpha}(t')S_{\alpha}(t)\phi_0\rightarrow 0$ in $W_x^{\frac{\alpha-1}{2}-\epsilon,\frac{1}{\epsilon}}$. Indeed, the uniform boundeness of $\Pi_{N_k}^{\perp}$ on $W_x^{\frac{\alpha-1}{2}-\epsilon,\frac{1}{\epsilon}}$ allows us to first prove the convergence for smooth functions and then a density argument. By Lebesgue's dominating convergence theorem, we have $\Pi_{N_k}^{\perp}\chi_l(t')S_{\alpha}(t')S_{\alpha}(t)\phi_0\rightarrow 0$ in $L_t^qW_{x}^{\frac{\alpha-1}{2}-\epsilon,\frac{1}{\epsilon}}$, for all $l\in\Z$. Consequently, $\Pi_{N_k}^{\perp}S_{\alpha}(t)\phi_0\rightarrow 0$ in $\mathcal{V}^{q,\epsilon}$. The convergence of the term $\Pi_{N_k}^{\perp}S_{\alpha}(t)(\phi_{0,k}-\phi_0)$ follows from the convergence of $\phi_{0,k}$ to $\phi_0$ in $\mathcal{V}^{q,\epsilon}$. Since the definition of $\mathcal{V}^{q,\epsilon}$ norm allows us to obtain a comparable norm after shifting $|t|\leq 1$. This proves \eqref{convergence:iterated1}.

Next we verify \eqref{ugly5} and \eqref{ugly6'}. We first claim that after changing the constant $R$ to $R+(2R)^3$ and $\mathcal{J}_k$ to $\{\phi_{0,k},\phi_0 \}$, \eqref{ugly1},\eqref{ugly2},\eqref{ugly3} still hold. Indeed, for each $f_j\in\{\phi_{0,k},\phi_0 \}$, by decomposition \eqref{decomposition}, there is a $\widetilde{f}_j\in\{\phi_{k},\phi \}$, such that $g_j=\widetilde{f}_j-f_j\in \{r_{0,k},r_0 \}$. Therefore, we can write
$$ \Gamma_{N_k,s,\epsilon}^{i_1,i_2,i_3}(f_1,f_2,f_3)\leq \Gamma_{N_k,s,\epsilon}^{i_1,i_2,i_3}(\widetilde{f}_1,f_2,f_3)+\mathcal{W}_{s,\epsilon}\big((\Pi_{N_k}^{\perp})^{i_1} g_1,(\Pi_{N_k}^{\perp})^{i_2} f_2, (\Pi_{N_k}^{\perp})^{i_3}f_3 \big),
$$ 
and the second term can be bounded by $R^3$ from Proposition \ref{multi-linearW}. We successively replace $f_2$ by $\widetilde{f}_2$ and a term bounded by $R^3$. Thus we obtain the analogue for \eqref{ugly1} for $\phi_{0,k},\phi_0$ with the upper bound $R+2^3R^3$. Now if one of $i_1,i_2,i_3$ is non zero, say $i_1=1$, we have
 $$\Gamma_{N_k,s,\epsilon}^{1,i_2,i_3}(f_1,f_2,f_3)\leq \Gamma_{N_k,s,\epsilon}^{1,i_2,i_3}(\widetilde{f}_1,f_2,f_3)+\mathcal{W}_{s,\epsilon}\big(\Pi_{N_k}^{\perp}  g_1,(\Pi_{N_k}^{\perp})^{i_2} f_2, (\Pi_{N_k}^{\perp})^{i_3}f_3 \big).
 $$
From Proposition \ref{multi-linearW} and \eqref{tail-r}, the second term of r.h.s can be bounded by $CR^2\|\Pi_{N_k}^{\perp}g_1\|_{H^s}\cdot $, and it converges to $0$. Next, we write
\begin{align*}
\Gamma_{N_k,s,\epsilon}^{1,i_2,i_3}(\widetilde{f}_1,f_2,f_3)\leq&\Gamma_{N_k,s,\epsilon}^{1,i_2,i_3}(\widetilde{f}_1,\widetilde{f}_2, f_3 )+ \mathcal{W}_{s,\epsilon}\big(\Pi_{N_k}^{\perp}\widetilde{f}_1, (\Pi_{N_k}^{\perp})^{i_2}g_2, (\Pi_{N_k}^{\perp})^{i_3} f_3 \big)\\
\leq & \Gamma_{N_k,s,\epsilon}^{1,i_2,i_3}(\widetilde{f}_1,\widetilde{f}_2,\widetilde{f}_3)+\mathcal{W}_{s,\epsilon}\big(\Pi_{N_k}^{\perp}\widetilde{f}_1, (\Pi_{N_k}^{\perp})^{i_2}g_2, (\Pi_{N_k}^{\perp})^{i_3} f_3 \big)\\
+& \mathcal{W}_{s,\epsilon}\big(\Pi_{N_k}^{\perp}\widetilde{f}_1, (\Pi_{N_k}^{\perp})^{i_2}\widetilde{f}_2, (\Pi_{N_k}^{\perp})^{i_3} g_3 \big).
\end{align*} 
Thus from Proposition \ref{multi-linearW} and the assumption \eqref{ugly2}, 
$$ \Gamma_{N_k,s,\epsilon}^{1,i_2,i_3}(f_1,f_2,f_3)\leq o(1)+CR^2\|\Pi_{N_k}^{\perp}\widetilde{f}_1\|_{\widetilde{\mathcal{V}}^{q,\epsilon}}.
$$
Since by definition, $\|\Pi_{N}^{\perp}f\|_{\widetilde{\mathcal{V}}^{q,\epsilon} }\leq CN^{-\epsilon/2}\|\Pi_N^{\perp}f\|_{\mathcal{V}^{q,\epsilon}}$, we have $\Gamma_{N_k,s,\epsilon}^{1,i_2,i_3}(f_1,f_2,f_3)=o(1)$, as $k\rightarrow\infty$. For the convergence of $\mathbf{d}(\phi_{0,k},\phi_0)$, by decomposition \eqref{decomposition} and using the triangle inequality, we have
\begin{align*}
 \mathbf{d}(\phi_{0,k},\phi_0)\leq &\sum_{f_2,f_3\in \{\phi_{0,k},\phi_0 \} }\mathcal{W}_{s,\epsilon}(\phi_k-\phi, f_2,f_3)+\mathcal{W}_{s,\epsilon}(r_{0,k}-r_0, f_2,f_3)\\
 +&\sum_{f_1,f_3\in \{\phi_{0,k},\phi_0 \} }\mathcal{W}_{s,\epsilon}(f_1,\phi_k-\phi,f_3)+\mathcal{W}_{s,\epsilon}(f_1,r_{0,k}-r_0,f_3).
\end{align*}
From Proposition \ref{multi-linearW}, the terms containing the entries $r_{0,k}-r_0$ converge to $0$, and the rests containing only $\phi_k-\phi,\phi_k,\phi$ as entries, which can be bounded by $\mathbf{d}(\phi_k,\phi)$.
Thus $\mathbf{d}(\phi_{0,k},\phi_0)\rightarrow 0$ as $k\rightarrow\infty$.
\\

Next we verify \eqref{ugly5}. Note that
$$ \Pi_{N_k}^{\perp}\Phi_{N_k}(t)\phi_k=\Pi_{N_k}^{\perp}S_{\alpha}(t)\phi_{0,k}+\Pi_{N_k}^{\perp}S_{\alpha}(t)r_{0,k}
$$ 	
and
$$ \Pi_{N_k}^{\perp}\Phi(t)\phi=b(t)\Pi_{N_k}^{\perp}S_{\alpha}(t)\phi+b(t)\Pi_{N_k}^{\perp}w(t).
$$
 For any $f_1,f_2,f_3\in \{\Phi_{N_k}(t)\phi_k, \Phi(t)\phi   \}$, by the triangle inequality, $\Gamma_{N_k,s,\epsilon}^{1,i_2,i_3}\big(f_1,f_2,f_3 \big)$ can be bounded by linear combinations of 
 $$ \mathcal{W}_{s,\epsilon}\big(\Pi_{N_k}^{\perp}S_{\alpha}(t)\widetilde{f}_1, f_2,f_3 \big),\quad  \widetilde{f}_1\in\{\phi_{0,k},b(t)\phi_0 \} $$ and 
 $$ \mathcal{W}_{s,\epsilon}\big(\Pi_{N_k}^{\perp}S_{\alpha}(t)g_1,f_2,f_3 \big), \quad g_1\in\{ r_{0,k}, b(t)w(t) \}.
 $$
Since $\|\cdot\|_{\mathcal{V}^{q,\epsilon}}$ and $\mathcal{W}_{s,\epsilon}(\cdot,\cdot,\cdot )$ is quasi-invariant under an $S_{\alpha}(t)$ action for $|t|\leq 1$, we obtain \eqref{ugly5} after using the triangle inequalities,  Proposition \ref{multi-linearW} and the previous claim. Finally, to verify \eqref{ugly6'}, we observe that $\mathbf{d}\big(\Phi_{N_k}(t)\phi_k,\Phi(t)\phi \big)$ can be expressed as linear combinations of the forms
$$ \mathcal{W}_{s,\epsilon}\big(\varphi_k(t)+\psi_k(t),f_2,f_3 \big),\quad \mathcal{W}_{s,\epsilon}\big(f_1,\varphi_k(t)+\psi_k(t),f_3),\quad f_1,f_2,f_3\in \{\Phi_{N_k}(t)\phi_k, \Phi(t)\phi  \}
$$
which contains the terms of the following forms:
$$ \mathcal{W}_{s,\epsilon}\big(\psi_k(t),\cdot,\cdot \big), \mathcal{W}_{s,\epsilon}(\cdot, \psi_k(t),\cdot);\mathcal{W}_{s,\epsilon}(\varphi_k(t),\cdot,\cdot),\mathcal{W}_{s,\epsilon}(\cdot,\varphi_k(t),\cdot).
$$
From Proposition \ref{multi-linearW}, the first two type of terms containing $\psi_k(t)$ converge to $0$. For the other two terms, if there is one place $\cdot$  filled by some functions in $H^s$, it converges to $0$, by Proposition \ref{multi-linearW} and the fact that $\varphi_k(t)\rightarrow 0$ in $\mathcal{V}^{q,\epsilon}$. The last possibility to treat is the term $\mathcal{W}_{s,\epsilon}(\varphi_k(t),\varphi_k(t),\varphi_k(t) )$. By the triangle inequality, it can be bounded by the terms of the form
$$ \mathcal{W}_{s,\epsilon}(\varphi_{1,k}(t),\varphi_{2,k}(t),\varphi_{3,k}(t) ),
$$ 
where $\varphi_{j,k}(t)$ is one of the functions:
\begin{align*}
 \big(b_k(t)-b(t) \big)\Pi_{N_k}S_{\alpha}(t)\phi_{0,k},\quad \big(1-b(t)\big)\Pi_{N_k}^{\perp}S_{\alpha}(t)\phi_0\\
\Pi_{N_k}^{\perp}S_{\alpha}(t)\big(\phi_{0,k}-\phi_0\big),\quad b(t)\Pi_{N_k}S_{\alpha}(t)(\phi_{0,k}-\phi_0).
\end{align*}
From the quasi-invariance of the quantity $\mathcal{W}_{s,\epsilon}(\cdot,\cdot,\cdot)$ under the $S_{\alpha}(t)$ action and hypothesis \eqref{ugly2}, \eqref{ugly3}, we deduce that $\mathcal{W}_{s,\epsilon}(\varphi_{1,k}(t),\varphi_{2,k}(t),\varphi_{3,k}(t) )$ converges to $0$, hence \eqref{ugly6'} is verified. The proof of Proposition~\ref{enhanced-localconvergence} is now complete.
\end{proof}
\subsection{Construction of the global flow}
		\begin{proposition}\label{longtime-1}
			Assume that $s\in\big[\frac{1}{2}-\frac{\alpha}{4}-\alpha-1 \big)$ and $\sigma\leq\frac{\alpha-1}{2}-\epsilon$. There exist constants $C>0, D>0$, $\delta>0$ such that for all $m\in\N, N\geq 1$, there exists a $\rho_N$ measurable set $\widetilde{\Sigma}_N^m\subset H^{\sigma}(\T)$, such that
			$$ \rho_N(H^{\sigma}\setminus \widetilde{\Sigma}_N^m)\leq 2^{-m+1}.
			$$
			For all $\phi\in\widetilde{\Sigma}_N^m$, $t\in\R$,
			$$ \|\Phi_N(t)\phi\|_{\mathcal{V}^{q,\epsilon}+H^s}+N^{\delta}\|\Pi_N^{\perp}\Phi_N(t)\phi\|_{\mathcal{V}^{q,\epsilon}+H^s}\leq Cm^{3/2}\left(1+\log(1+|t|)\right)^{3/2},
			$$
			and for all $i_1,i_2,i_3\in\{0,1 \}$,
			\begin{align*}
		\Gamma_{N,s,\epsilon}^{i_1,i_2,i_3}\big(\Phi_N(t)\phi \big) \leq CN^{-\delta(i_1+i_2+i_3)}m^{3/2}\left(1+\log(1+|t|)\right)^{3/2}.
			\end{align*}
			In particular,
			$$ \|\Phi_N(t)\phi\|_{H^{\sigma}(\T)}\leq Cm^{3/2}\left(1+\log(1+|t|)\right)^{3/2}.
			$$
			Moreover, for all $t_0\in\R$, $m\in\N, N\geq 1$,
			\begin{equation}\label{including}
			\Phi_N(t_0)(\widetilde{\Sigma}_N^m)\subset\widetilde{\Sigma}_N^{Dm(1+\log_2(1+|t_0|))}.
			\end{equation}
		\end{proposition}
	We need the following lemma.
	\begin{lemme}\label{decomposition1}
	Assume that $\phi\in \mathcal{V}^{q,\epsilon}+H^s$ such that for some $R>0, \delta>0,$
	$$ \|\phi\|_{\mathcal{V}^{q,\epsilon}+H^s}\leq R, \quad \|\Pi_{N}^{\perp}\phi\|_{\mathcal{V}^{q,\epsilon}+H^s }\leq N^{-\delta}R.
	$$
	Then there exists $\phi_0\in\mathcal{V}^{q,\epsilon}, r_0\in H^s$, such that
	$$ \|\phi_0\|_{\mathcal{V}^{q,\epsilon}}+\|r_0\|_{H^s}\leq 2A_0(R+1), \quad \|\Pi_{N}^{\perp}\phi_0\|_{\mathcal{V}^{q,\epsilon}}+\|\Pi_{N}^{\perp}r_0\|_{H^s}\leq N^{-\delta}A_0(R+1),
	$$
	where $A_0>0$ is a uniform constant.
	\end{lemme}
\begin{proof}
	By definition, there exists $\varphi_N,\varphi\in \mathcal{V}^{q,\epsilon}$ and $\psi_N,\psi\in H^s$, such that
	$$ \phi=\varphi+\psi,\quad \Pi_{N}^{\perp}\phi=\varphi_N+\psi_N
	$$
	and
	$$ \|\varphi\|_{\mathcal{V}^{q,\epsilon}}+\|\psi\|_{H^s}\leq R+1, \quad \|\varphi_N\|_{\mathcal{V}^{q,\epsilon}}+\|\psi_N\|_{H^s}\leq N^{-\delta}(R+1).
	$$
	Note that in a priori, we do not know if $\varphi_N\in E_N^{\perp}$ and $\psi_N\in E_N^{\perp}$. Since $(\Pi_N^{\perp})^2\phi=\Pi_N^{\perp}\phi$, we can replace $\varphi_N,\psi_N$ by $\Pi_N^{\perp}\varphi_N, \Pi_N^{\perp}\psi_N$, from Lemma \ref{Dirichletkernel},
	 we have
	 $$ \|\Pi_{N}^{\perp}\varphi_N\|_{\mathcal{V}^{q,\epsilon}}+\|\Pi_N^{\perp}\psi_N\|_{H^s}\leq A_0N^{-\delta}(R+1), \quad \|\Pi_{N}\varphi\|_{\mathcal{V}^{q,\epsilon}}+\|\Pi_N\psi\|_{H^s}\leq A_0(R+1).
	 $$
	 Let $\phi_0=\Pi_N\varphi+\Pi_N^{\perp}\varphi_N, r_0=\Pi_N\psi+\Pi_N^{\perp}\psi_N$ and using the triangle inequality, the proof of Lemma \ref{decomposition1} is complete.
\end{proof}
		\begin{proof}[Proof of Proposition \ref{longtime-1}]
			The construction is slightly different compared to \cite{BTT-AIF}, due to the multi-linear and sum space structures. 
			For $m,k\in\N$ and $D>0$ to be chosen later, we define the set
			\begin{equation}\label{BmkN}
			\begin{split}
			B_N^{m,k}(D):=&\big\{\phi\in H^{\sigma}(\T): \|\phi\|_{\mathcal{\mathcal{V}}^{q,\epsilon}+H^s}\leq D(mk)^{3/2}, \|\Pi_N^{\perp}\phi\|_{\mathcal{V}^{q,\epsilon}+H^s}\leq N^{-\delta}D(mk)^{3/2}  \}\\
			\cap &\big\{ \phi\in H^{\sigma}(\T): \forall i_1,i_2,i_3\in\{0,1 \}, \Gamma_{N,s,\epsilon}^{i_1,i_2,i_3}(\phi) \leq  N^{-\delta(i_1+i_2+i_3)}D^3(mk)^{9/2} \big\}.
			\end{split}
			\end{equation} 
			By Lemma \ref{decomposition1}, for $\phi\in B_N^{m,k}(D)$, there exists a decomposition
			$ \phi=\phi_0+r_0 $, such that
			$$
			 \|\phi_0\|_{\mathcal{V}^{q,\epsilon}}+\|r_0\|_{H^s}\leq 2A_0D(mk)^{3/2},
			\quad  \|\Pi_N^{\perp}\phi_0\|_{\mathcal{V}^{q,\epsilon}}+\|\Pi_N^{\perp}r_0\|_{H^s}\leq A_0N^{-\delta}D(mk)^{3/2}\,.
			$$
			Arguing as in the proof of Proposition \ref{enhanced-localconvergence}, we deduce that there exists $C_0>0$, and $\delta<\frac{\epsilon}{6}$, such that
		\begin{equation*}\label{multilinearboundGamma}
	 \Gamma_{N,s,\epsilon}^{i_1,i_2,i_3}(\phi_0)\leq C_0 N^{-\delta(i_1+i_2+i_3)}D^3(mk)^{9/2},\forall i_1,i_2,i_3\in\{0,1 \}.
				\end{equation*}
			From Proposition \ref{LWP-main}, the time for local existence is 
			$ \tau_{m,k}=c(2A_0D)^{-\kappa}(mk)^{-\frac{3\kappa}{2}}.
			$
			Then for any $|t|\leq\tau_{m,k}$, we can write the solution as
		\begin{align*}
		&\Phi_N(t)\phi=\varphi_N(t)+\psi_N(t),\\
		&\varphi_N(t)=e^{\frac{it}{\pi}\|\Pi_{N}\phi\|_{L^2}^2 }\Pi_{N}S_{\alpha}(t)\phi_{0}+\Pi_{N}^{\perp}S_{\alpha}(t)\phi_{0}\in \mathcal{V}^{q,\epsilon}\,, \\  &\psi_N(t)= e^{\frac{it}{\pi}\|\Pi_{N}\phi\|_{L^2}^2 }\big(\Pi_{N}S_{\alpha}(t)r_{0}+w_N(t) \big)+\Pi_{N}^{\perp}S_{\alpha}(t)r_{0}\in H^s
		\end{align*}
			with the property that $w_N(t)\in E_N$, and
			\begin{equation*}\label{boundeness}
			\begin{split}
			\sup_{|t|\leq \tau_{m,k}}\big(\|\varphi_N(t)\|_{\mathcal{V}^{q,\epsilon}}+\|\psi_N(t)\|_{H^s} \big)\leq 4A_0D(mk)^{3/2},
			\end{split}
			\end{equation*}
			since $S_{\alpha}(t)$ keeps invariant of the $H^s$ norm, quasi-invariant the $\mathcal{V}^{q,\epsilon}$ and $\Pi_N^{\perp}w_N=0$. Therefore, for $|t|\leq \tau_{m,k}$
			\begin{equation*}\label{boundeness1}
			\begin{split}
			\|\Phi_N(t)\phi\|_{\mathcal{V}^{q,\epsilon}+H^s }\leq 4A_0D(mk)^{3/2}.
			\end{split}
			\end{equation*}
Since
$$ \Pi_{N}^{\perp}\Phi_N(t)\phi=\Pi_N^{\perp}S_{\alpha}(t)(\phi_0+r_0),
$$ 	
from the quasi-invariance of the norm $\mathcal{V}^{q,\epsilon}$ under $S_{\alpha}(t), |t|\leq \frac{1}{2}$, we obtain that
\begin{align*}
\|\Pi_N^{\perp}\Phi_N(t)\phi\|_{\mathcal{V}^{q,\epsilon}+H^s}\leq A_1\|\Pi_N^{\perp}\phi\|_{\mathcal{V}^{q,\epsilon}+H^s}\leq A_1N^{-\delta}D(mk)^{3/2}.
\end{align*}	
Next, we estimate the quantities $\Gamma_{N,s,\epsilon}^{i_1,i_2,i_3}(\Phi_N(t)\phi)$ for all $i_1,i_2,i_3\in\{0,1 \}$. By expanding $\Phi_N(t)\phi=\varphi_N(t)+\psi_N(t)$ and using the triangle inequality, we note that except for the term $ \Gamma_{N,s,\epsilon}^{i_1,i_2,i_3}(\varphi_N(t))$, the other terms are of the form
$$\Gamma_{N,s,\epsilon}^{i_1,i_2,i_3}(\psi_N(t),\cdot,\cdot), \Gamma_{N,s,\epsilon}^{i_1,i_2,i_3}(\cdot,\psi_N(t),\cdot),\Gamma_{N,s,\epsilon}^{i_1,i_2,i_3}(\cdot,\cdot,\psi_N(t)).
$$
Therefore, from Proposition \ref{multi-linearW}, the terms containing $\psi_N$ in one of their entries can be estimated by
$$ C\big\|\psi_N(t) \big\|_{H^s}^3+C\|\varphi_N(t)\|_{\widetilde{\mathcal{V}}^{q,\epsilon}}^3\leq CD^3(mk)^{3/2}.
$$
Furthermore, if one of $i_1+i_2+i_3>0$, we gain $N^{-\delta}$ with $\delta<\frac{\epsilon}{6}$ from either $\|\Pi_N^{\perp}r_0\|_{H^s}\leq A_0N^{-\delta}D(mk)^{3/2} $ or $\|\Pi_N^{\perp}\phi_0\|_{\mathcal{V}^{q,\epsilon}}\leq A_0N^{-\delta}D(mk)^{3/2}$. 
For the term $\Gamma_{N,s,\epsilon}^{i_1,i_2,i_3}(\varphi_N(t))$, we use the triangle inequality to estimate it by the sum of the terms $\Gamma_{N,s,\epsilon}^{i_1,i_2,i_3}(f_1,f_2,f_3 )$, where $f_1,f_2,f_3\in \{e^{\frac{it}{\pi}\|\Pi_N\phi\|_{L^2}^2}\Pi_NS_{\alpha}(t)\phi_0, \Pi_N^{\perp}S_{\alpha}(t)\phi_0  \} $. From Lemma \ref{quasi-invarianceW}, we obtain that
$$ \Gamma_{N,s,\epsilon}^{i_1,i_2,i_3}(\varphi_N(t))\leq 2^3C_0A_1N^{-\delta(i_1+i_2+i_3)}D^3(mk)^{9/2}.
$$
Consequently, for all $|t|\leq \tau_{m,k}$ and $i_1,i_2,i_3\in\{0,1\}$,
$$ \Gamma_{N,s,\epsilon}^{i_1,i_2,i_3}\big(\Phi_N(t)\phi \big)\leq C_1N^{-\delta(i_1+i_2+i_3)}D^3(mk)^{9/2}.
$$
%Since multiplying by a bounded function in $t$ and applying $S_{\alpha})(t)$ leaves the quantity $\mathcal{W}_{s,\epsilon}(\cdot,\cdot,\cdot)$, we obtain 	
%%%%%%%%%%%%%%%%%%%%%%%%%%%%%%%%%%%%%%%%%%%%
Since $\|\phi\|_{\mathcal{V}^{q,\epsilon}+H^s}\leq \|\phi\|_{\mathcal{V}^{q,\epsilon}}$, we deduce from Corollary \ref{convergence:W} and Lemma \ref{linear-convergence} that
			$$ \rho_N(H^{\sigma}\setminus B_N^{m,k}(D))\leq e^{-cD^{2/3}mk}.
			$$
			Next, we set
			\begin{equation}\label{dataset0}
			  \widetilde{\Sigma}_N^{m,k}(D):=\bigcap_{|j|\leq \frac{2^k}{\tau_{m,k}}}\Phi_N(-j\tau_{m,k})\big(B_N^{m,k}(D)\big),
			\end{equation}
			from the invariance of $\rho_N$ under the flow $\Phi_N(t)$, we have
			\begin{equation*}
			\begin{split}
			&\rho_N\big(H^{\sigma}\setminus \widetilde{\Sigma}_N^{m,k}(D) \big)\leq \sum_{|j|\leq \frac{2^k}{\tau_{m,k}}} \rho_N\big(H^{\sigma}\setminus \Phi_N(-j\tau_{m,k})\big(B_N^{m,k}(D) \big) \big)\\
			\leq & \frac{2^{k+2}}{\tau_{m,k}}\rho_N\big(H^{\sigma}\setminus B_N^{m,k}(D) \big)
			\leq  \frac{2^{k+2}}{c}D^{\kappa}(mk)^{\frac{3\kappa}{2}} e^{-cD^{2/3}(mk)}\leq 2^{-mk},
			\end{split}
			\end{equation*}
			provided that $D$ is chosen large enough. Now we define the desired data set by
			\begin{equation}\label{dataset1}
			\widetilde{\Sigma}_N^m:=\bigcap_{k\geq 1}\widetilde{\Sigma}_N^{m,k}(D).
			\end{equation}
			It is clear that $\rho_N\big(H^{\sigma}\setminus\widetilde{\Sigma}_N^m\big)\leq 2^{-m+1}$. 
			\\
			%For $t_0\in\R$, there exists $k_0\in\Z$, such that $2^{k_0}\leq |t|<2^{k_0+1}$. Then for all $k\geq k_0$, there exists $|j|\leq 2^k\tau_{m,k}^{-1}$, such that $|t_0-j\tau_{m,k}|\leq \tau_{m,k}.$ In particular, for $\phi\in\widetilde{\Sigma}_N^{m}$, by  definition, $\phi\in \widetilde{\Sigma}_N^{m,k_0}$ and thus $\Phi_N(j\tau_{m,k_0})\phi\in B_N^{m,k_0+1}(D)$. By writing $\Phi_N(t_0)\phi=\Phi_N(t_0-j\tau_{m,k_0})\Phi_N(j\tau_{m,k_0})\phi$, we deduce that $\Phi_N(t_0)\phi\in B_N^{\sqrt{D}m,\sqrt{D}k_0+1}(D)\subset B_N^{\sqrt{D}m+\sqrt{D}(k_0+2),l}(D)$ for all $l\in\N$, provided that $D$ is chosen large enough. Consequently, we have
		%	$$ \|\Phi_N(t_0)\phi\|_{\mathcal{V}^{q,\epsilon}+H^s}+ N^{\delta}\|\Pi_N^{\perp}\Phi_N(t_0)\phi\|_{\mathcal{V}^{q,\epsilon}+H^s}\leq 2D(\sqrt{D}m+1+\sqrt{D}\log_2(1+|t_0|) )^{3/2}, 
		%	$$
		%	and
		%	$$ \Gamma_{N,s,\epsilon}^{i_1,i_2,i_3}\big(\Phi_N(t_0)\phi\big)\leq 2D(\sqrt{D}m+1+\sqrt{D}\log_2(1+|t_0|)^{3/2}  ).
		%	$$
		%	In particular, $\Phi_N(t_0)\phi\in \widetilde{\Sigma}_N^{\sqrt{D}m+\sqrt{D}\log_2(1+|t_0|)+8}$.
		
		Finally, we prove \eqref{including}. 
		 Let $m_0=Dm\log_2(1+|t_0|)$.
		Take any $\phi\in\widetilde{\Sigma}_N^m$, by definitions \eqref{dataset0} and \eqref{dataset1}, we need to show that for any $l\geq 1$ and $|j|\leq 2^l/\tau_{m_0,l}$, $\Phi_N(j\tau_{m_0,l})\Phi_N(t_0)\phi\in B_N^{m_0,l}(D)$.
		By definition of the set $B_N^{m,k}(D)$ in \eqref{BmkN}, we observe that for any $C_0\geq 1$ and $l_0\geq ,l_0\in\N$,
		\begin{equation}\label{inclu}
		 B_N^{m,k}(C_0D)\subset B_N^{(\lfloor C_1^{\frac{2}{3}}\rfloor+1)m,k}(D),\quad B_N^{m,l_0k}(D)=B_N^{l_0m,k}(D).
		\end{equation}
		Moreover, a previous argument (the local theory) yields
		\begin{equation}\label{shorttime}
	 \Phi_N(t)(B_N^{m,k}(D))\subset B_N^{m,k}(C_2D),\forall |t|\leq \tau_{m,k}
		\end{equation} 
		where $C_2>4A_0+A_1+C_1$ is some uniform constant.
		For $t_0\neq 0$, without loss of generality, we assume that $|t_0|\geq 1$. Then there exists $k_0\in\N$, such that $2^{k_0}\leq |t_0|<2^{k_0+1}$. Denote by $k_1=k_0+2$. Take $\phi\in \widetilde{\Sigma}_N^m$. If $l< k_1$, then $|t_0+j\tau_{m_0,l}|\leq 2^{2k_1-1}$, and there exists $|j_2|\leq 2^{2k_1}/\tau_{m,2k_1}$, such that $|t_0+j\tau_{m_0,l}-j_2\tau_{m,2k_1}|\leq\tau_{m,2k_1}$. Thus by definition and \eqref{shorttime} $$\Phi_N(t_0+j\tau_{m_0,l})\phi=\Phi_N(t_0+j\tau_{m_0,l}-j_2\tau_{m,2k_1})\Phi_N(j_2\tau_{m,2k_1})\phi\in B_N^{m,2k_1}(C_2D).
		$$  
		Using \eqref{inclu}, we have $$\Phi_N(t_0+j\tau_{m_0,l})\phi\in B_N^{(\lfloor C_2^{\frac{2}{3}}\rfloor+1)2mk_1,1}(D)\subset B_N^{m_0,l}(D),
		$$
		provided that $D$ is chosen large enough. If $l\geq k_1$, then $|t_0+j\tau_{m_0,l}|\leq 2^{2l-1}$. Then without loss of generality, there exists $j_2\leq 2^{2l}/\tau_{m,2l}$,  such that
		$$ j_2\tau_{m,2l}\leq |t_0+j\tau_{m_0,l}|\leq (j_2+1)\tau_{m,2l},
		$$ 
		and we can write
		$$ \Phi_N(t_0+j\tau_{m_0,l})\phi=\Phi_N(t_0+j\tau_{m_0,l}-j_2\tau_{m,2l} )\Phi_N(j_2\tau_{m,2l})\phi\in B_N^{m,2l}(C_2D).
		$$
        Again from \eqref{inclu}, we have $\Phi_N(t_0+j\tau_{m_0,l})\phi\in B_N^{(\lfloor C_2^{\frac{2}{3}} \rfloor+1)2m,l}(D)\subset B_N^{m_0,l}(D)$.
		 This completes the proof of Proposition \ref{longtime-1}.
		\end{proof}
		Define
		\begin{align*}
	\Sigma^m:=\Big\{ \phi\in \mathcal{V}^{q,\epsilon}+H^s\,:\, \,&\exists\, N_k\rightarrow\infty, \phi_k\in \widetilde{\Sigma}_{N_k}^m, \|\phi_k-\phi\|_{\mathcal{V}^{q,\epsilon}+H^s}\rightarrow 0, \mathbf{d}(\phi_k,\phi)\rightarrow 0, \\
	  & \text{ and }  \sum_{ \substack{f_j\in\{\phi_k,\phi \}\\ i_1+i_2+i_3>0} }\Gamma_{N_k,s,\epsilon}^{i_1,i_2,i_3}(f_1,f_2,f_3)\rightarrow 0           \Big\}.
	\end{align*}

		\begin{lemme}
			Assume that $\sigma\leq \frac{\alpha-1}{2}-\epsilon$ as in Proposition \ref{longtime-1}. Then
				\begin{equation}\label{inclusion}
			\begin{split}
			\limsup_{N\rightarrow\infty}\widetilde{\Sigma}_N^m=\bigcap_{N=1}^{\infty}\bigcup_{N'=N}\widetilde{\Sigma}_N^m\subset \Sigma^m.
			\end{split}
			\end{equation}
			and 
			$$ \rho(\Sigma^m)\geq \rho(H^{\sigma})-2^{-m}.
			$$
		\end{lemme}
		\begin{proof}
			We first prove the inclusion \eqref{inclusion}. Take $\phi\in \limsup_{N}\widetilde{\Sigma}_N^m$, there exists a sequence $N_k\rightarrow\infty$, such that $\phi\in \widetilde{\Sigma}_{N_k}^m$ for all $k$. We set $\phi_k=\phi$, then trivially we verify that $\phi\in\Sigma^m$. By Fatou's lemma,
			$$ \rho(\Sigma^m)=\rho\big(\limsup_{N\rightarrow\infty}\widetilde{\Sigma}_N^m \big)\geq \limsup_{N\rightarrow\infty}\rho\big(\widetilde{\Sigma}_N^m\big).
			$$
			From the construction of the Gibbs measure, we know that
			$$ \lim_{N\rightarrow\infty}\big(\rho\big(\widetilde{\Sigma}_N^m\big)-\rho_N\big(\widetilde{\Sigma}_N^m\big) \big)=0.
			$$
			Therefore, from Proposition \ref{longtime-1}, we have 
			$$ \limsup_{N\rightarrow\infty}\rho\big(\widetilde{\Sigma}_N^m\big)\geq \limsup_{N\rightarrow\infty}\rho_N\big(\widetilde{\Sigma}_N^m\big)\geq \rho(H^{\sigma})-2^{-m}.
			$$
		\end{proof}
		Consequently, the set
		$$ \Sigma:=\bigcup_{m=1}^{\infty}\Sigma^m
		$$
		is of full $\rho$ measure. The last step to prove Theorem~\ref{thm6} is the following proposition, ensuring the global existence and the flow property of $\Phi(t)$.
		\begin{proposition}\label{longtime-2}
			For every integer $m\in\N$ and every $\phi\in\Sigma^m$, the solution $\Phi(t)\phi$ with initial data $\phi$ is globally defined. Moreover, there exists $C>0$, such that for every $\phi\in\Sigma^m$ and $t\in\R$, we have
			$$ \|\Phi(t)\phi\|_{\mathcal{V}^{q,\epsilon}+H^s}+\big(\mathcal{W}_{s,\epsilon}(\Phi(t)\phi )\big)^{\frac{1}{3}}\leq Cm^{\frac{3}{2}}\big(1+\log(1+|t|)\big)^{3/2}.
			$$ 
			Furthermore, $\Phi(t)\Sigma=\Sigma$. In other words, the flow map $\Phi(t)$ is globally defined on $\Sigma$.
		\end{proposition}
%%%%%%%%%%%%%%%%%%%%%%%%%%%%%%%%%%%%%%%%%%%%%

%%%%%%%%%%%%%%%%%%%%%%%%%%%%%%%%%%%%%%%%%%%%%%%%%%%%%	

		\begin{proof}[Proof of Proposition \ref{longtime-2}]
			Take $\phi\in \Sigma^m$, by definition, there is a sequence $N_k\rightarrow\infty$, and a sequence  $\phi_k\in\widetilde{\Sigma}_{N_k}^m$, such that
			$$ \|\phi_k-\phi\|_{\mathcal{V}^{q,\epsilon}+H^s}+\mathbf{d}(\phi_k,\phi)\rightarrow 0,\quad k\rightarrow\infty,
			$$
			and
			\begin{equation}\label{limit} \lim_{k\rightarrow\infty}\sum_{ \substack{f_j\in\{\phi_k,\phi,\}\\ i_1+i_2+i_3>0 }}\Gamma_{N_k,s,\epsilon}^{i_1,i_2,i_3}(f_1,f_2,f_3)=0.
			\end{equation}
			 By definition, for all $k\in \N$ and $t\in\R$ we have
			\begin{align}
			& \|\Phi_{N_k}(t)\phi_k\|_{\mathcal{V}^{q,\epsilon}+H^s}+N_k^{\delta}\|\Pi_{N_k}^{\perp}\Phi_{N_k}(t)\phi_k\|_{\mathcal{V}^{q,\epsilon}+H^s} \leq Cm^{\frac{3}{2}}\big(1+\log(1+|t|)\big)^{3/2},\label{bound1}\\
			&\Gamma_{N_k,s,\epsilon}^{i_1,i_2,i_3}(\Phi_{N_k}(t)\phi_k)\leq CN_{k}^{-\delta(i_1+i_2+i_3)}m^{\frac{9}{2}}\big(1+\log(1+|t|)\big)^{9/2}.\label{bound2}
			\end{align}
	At $t=0$, passing $k$ to the limit, we obtain that $\|\phi\|_{\mathcal{V}^{q,\epsilon}+H^s}\leq C(m+1)^{3/2}$.
	Using triangle inequality, we deduce that for any $f_1,f_2,f_3\in\{\phi_k,\phi \}$, $$\Gamma_{N_k,s,\epsilon}^{0,0,0}(f_1,f_2,f_3)\leq \Gamma_{N_k,s,\epsilon}^{0,0,0}(\phi_k)+3\mathbf{d}(\phi_k,\phi).
	$$
	Thus from \eqref{bound2} and \eqref{limit} we have
	$$ \sum_{ \substack{f_j\in\{\phi_k,\phi,\}\\ i_1,i_2,i_3\in\{0,1\} }}\Gamma_{N_k,s,\epsilon}^{i_1,i_2,i_3}(f_1,f_2,f_3)\leq C(m+1)^{9/2}.
	$$
	Denote by $\Lambda_T=2Cm^{\frac{3}{2}}\big(1+\log(1+|t|)\big)^{3/2}$ for any given $T>0$. We need show that there exists a uniform constant $C'>0$, such that $\Phi(t)\phi$ exists on $[0,T]$ and satisfies
	\begin{align*}
	&\|\Phi(t)\phi\|_{\mathcal{V}^{q,\epsilon}+H^s }+\big(\mathcal{W}_{s,\epsilon}(\Phi(t)\phi)\big)^{\frac{1}{3}}\leq C'\Lambda_T.
	\end{align*}	
Let $\tau_T=c2^{-\kappa}(\Lambda_T+1)^{-\kappa}$, the time in Proposition~\ref{enhanced-localconvergence} for $R=2(\Lambda_T+1)$ and divide $[0,T]$ by $N_T\sim T/\tau_T$ intervals of size $\tau_T$. With $R=2(\Lambda_T+1)$, the hypotheses of Proposition~\ref{enhanced-localconvergence}
are satisfied. Thus we have for $t\in[0,\tau_T]$, 
	\begin{align}\label{convergencelinearnorm}
\lim_{k\rightarrow\infty}\|\Phi_{N_k}(t)\phi_{k}-\Phi(t)\phi \|_{\mathcal{V}^{q,\epsilon}+H^s}=0.
\end{align}
Furthermore, with $\mathcal{J}_{k,t}=\{\Phi_{N_k}(t)\phi_k, \Phi(t)\phi  \}$, we have
\begin{align*}
\lim_{k\rightarrow\infty}\sum_{ \substack{ f_j\in\mathcal{J}_{k,t}\\
		i_1,i_2,i_3\in\{0,1\}\\ i_1+i_2+i_3>0  } } \Gamma_{N_k,s,\epsilon}^{i_1,i_2,i_3}(f_1,f_2,f_3)=0,
\end{align*}
and
\begin{align}\label{convergencepseudodisance}
\lim_{k\rightarrow\infty}\mathbf{d}\big(\Phi_{N_k}(t)\phi_k,\Phi(t)\phi\big)=0.
\end{align}
Note that $\Phi_{N_k}(t)\phi_k\in\widetilde{\Sigma}_{N_k}^{Dm(1+\log(1+|t|))}$, then by definition
$$ \Phi(t)\phi\in \Sigma^{Dm(1+\log_2(1+|t|))},\quad \forall |t|\leq \tau_T.
$$
By passing $k$ to infinity, \eqref{convergencelinearnorm} implies that
$$ \|\Phi(t)\phi\|_{\mathcal{V}^{q,\epsilon}+H^s}\leq Cm^{\frac{3}{2}}(1+\log(1+|\tau_T|))^{3/2}\leq R/2.
$$
Similarly, using \eqref{convergencepseudodisance} and passing $k\rightarrow\infty$ of \eqref{bound2} at $t=\tau_T$, we obtain that
$$ \sum_{ \substack{ f_j\in\mathcal{J}_{k,t}\\
		i_1,i_2,i_3\in\{0,1\}\\ i_1+i_2+i_3>0  } } \Gamma_{N_k,s,\epsilon}^{i_1,i_2,i_3}(f_1,f_2,f_3)\leq R^3/2.
$$
In particular, the hypotheses of Proposition \ref{enhanced-localconvergence} are satisfied for the initial data $\Phi(\tau_T)\phi$ and the approximating sequence $(\Phi_{N_k}(\tau_T)\phi_k)_{k\in\N}$, with the same $R=2(\Lambda_T+1)$. This allows us to repeatedly use Proposition \ref{enhanced-localconvergence} to the interval $[2\tau_T,3\tau_T],\cdots, [(N_T-1)\tau_T,N_T\tau_T] $. This procedure shows that the solution $\Phi(t)\phi$ exists globally in $t\in\R$. Moreover, 
$$ \Phi(t)\phi\in \Sigma^{Dm(1+\log_2(1+|t|)) }
$$
holds for all $t\in\R$. This implies that $\Phi(t)\Sigma^m\subset \Sigma$, hence $\Phi(t)\Sigma\subset \Sigma$. By reversibility, we have $\Phi(t)\Sigma=\Sigma$. Note that the structure of the solution allows us to pass to the limit of the relation
$$ \Phi_N(t+s)=\Phi_N(t)\circ\Phi_N(s),\forall t,s\in\R.
$$
Therefore, the limit flow $\Phi(t)$ satisfies the group property. This completes the proof of Proposition \ref{longtime-2}.
\end{proof}
%%%%%%%%%%%%%%%%%%%%%%
\subsection{Measure invariance}
To prove the measure convergence, by reversibility of the flow $\Phi(t)$ and the reduction argument in \cite{BTT-AIF} (see also \cite{Tz-AIF, Sun-Tz}), it suffices to show that for any $R>0$ any any $H^{\sigma}$ compact subset of $\mathcal{B}_R$, we have
\begin{align}\label{invariance}
\rho(K)\leq\rho(\Phi(t)\rho(K)), 
\end{align} 
where
$$ \mathcal{B}_R:=\big\{\phi\,:\, \|\phi\|_{\mathcal{V}^{q,\epsilon}+H^s}\leq R, \,\,\mathcal{W}_{s,\epsilon}(\phi)\leq R^3  \big\}.
$$
We need the following approximation lemma.
\begin{lemme}\label{approximation}
There exists $C_0>0$, such that the following holds true. For every large $R\geq 1$, small  $\epsilon>0$, and every compact set $K\subset\mathcal{B}_R$ with respect to the $H^{\sigma}$ topology, there exists $N_0\geq 1,\kappa>0,c>0$, such that for all $N\geq N_0, \phi\in K,|t|\leq \tau_R=cR^{-\kappa}$, we have
$$ \|\Phi(t)\phi-\Phi_N(t)\phi \|_{H^{\sigma}}<\epsilon.
$$
\end{lemme}
\begin{proof}
This is a simple consequence of the local well-posedness. Write
$$ \Phi(t)\phi=e^{\frac{it}{\pi}\|\phi\|_{L^2}^2 }\big(\Pi_N\Psi(t)\phi,\Pi_N^{\perp}\Psi(t)\phi \big), \quad \Phi_N(t)\phi=\big(e^{\frac{it}{\pi}\|\Pi_N\phi\|_{L^2}^2 }\Pi_N\Psi_N(t)\phi, \Pi_N^{\perp}\Psi_N(t)\phi \big),
$$
where $\Psi_N(t)\phi$ ($\Psi(t)\phi$) is the local solution of the Wick-ordered truncated (non-truncated) equation. Note that from the compactness of $K$ in $H^{\sigma}$, the convergences of
$ \|\Pi_N\phi\|_{L^2}$ to $ \|\phi\|_{L^2}
$ and $\|\Pi_N^{\perp}\phi\|_{H^{\sigma}}$ to $0$ are uniform. Therefore, it suffices to prove the uniform convergence of $\Psi_N(t)\phi$ to $\Psi(t)\phi$ in $H^{\sigma}$.

From Proposition \ref{LWP-main}, we have, for $|t|\leq\tau_R= cR^{-\kappa}$
$$ \Psi_N(t)\phi=S_{\alpha}(t)\phi+w_N(t), \quad \Psi(t)\phi=S_{\alpha}(t)\phi+w(t),
$$
where the nonlinear parts $w_N(t)\in E_N, w(t)\in H^s$ satisfy the integral equations:
$$ w_N(t)=-i\Pi_N\int_0^tS_{\alpha}(t-t')\mathcal{N}\big(\Psi_N(t')\phi\big)dt',\quad w(t)=-i\int_0^tS_{\alpha}(t-t')\mathcal{N}\big(\Psi(t')\phi\big)dt',
$$
and
$$ \|w_N\|_{X_T^{s,\frac{1}{2}+2\epsilon}}+\|w\|_{X_T^{s,\frac{1}{2}+2\epsilon}}\leq T^{\theta}R^3,
$$
if $T\leq\tau_R$. 
Expanding the trilinear expression $\mathcal{N}(\cdot)$ and using Proposition \ref{multi-linear}, we obtain that
$$ \|w_N(t)-w(t)\|_{X_T^{s_1,\frac{1}{2}+2\epsilon}}\leq \|\Pi_N^{\perp}w\|_{X_T^{s_1,\frac{1}{2}+2\epsilon}}+CT^{\theta}R^3\|w_N-w\|_{X_T^{s_1,\frac{1}{2}+2\epsilon}},
$$
where $s_1\in[\frac{1}{2}-\frac{\alpha}{4},s)$.
Taking $\kappa>0$ large enough and $T\leq T_R=cR^{-\kappa}$, we obtain that
$$ \|w_N-w\|_{X_{T_R}^{s_1,\frac{1}{2}+2\epsilon}}\leq C\|\Pi_N^{\perp}w\|_{X_T^{s_1,\frac{1}{2}+2\epsilon}}\leq CN^{-(s-s_1)}T^{\theta}R^3.
$$
This proves the uniform convergence of $\Psi_N(t)\phi-\Psi(t)\phi$ to $0$ in $H^{s_1}(\T)\hookrightarrow H^{\sigma}(\T)$. The proof of Lemma \ref{approximation} is now complete.
\end{proof}
To finish the proof of the measure invariance, we observe that for any $\epsilon>0$, from Fatou's lemma and the approximation Lemma \ref{approximation}, we have
$$ \rho\big(\Phi(t)(K)+B_{\delta}^{H^{\sigma}} \big)\geq \lim_{N\rightarrow\infty}\rho_N\big(\Phi(t)(K)+B_{\delta}^{H^{\sigma}} \big)\geq \limsup_{N\rightarrow\infty}\rho_N\big(\Phi_N(t)(K)+B_{c\delta}^{H^{\sigma}} \big),
$$
for all $|t|\leq \tau_R$. Thus from invariance of $\rho_N$ under $\Phi_N(t)$, we have 
$$ \limsup_{N\rightarrow\infty}\rho_N(\Phi_N(t)(K)+B_{c\delta}^{H^{\sigma}} )\geq \limsup_{N\rightarrow\infty}\rho_N(K)=\rho(K).
$$
Passing $\delta\rightarrow 0$, we obtain that for $|t|\leq T_R$, we have $\rho(\Phi(t)(K))\geq \rho(K)$. Iterating the argument, we obtain \eqref{invariance} for all $t\in\R$. This proves the invariance of the Gibbs measure. The proof of Theorem \ref{thm6} is then complete.
\begin{proof}[Proof of Corollary \ref{stability}]
From the invariance of the Gibbs measure $d\rho=e^{-V}d\mu$ by $\Phi(t)$, the transported measure $\mu^t=\Phi(t)_*\mu$ is absolutely continuous with respect to $\mu$. 
By the Radon-Nikodym theorem,  for every $t\in\R$ there exists a function $G(t)\in L^1(d\mu)$, $G\geq 0$ such that $\mu^t=G(t)d\mu$. Set 
$$
d\nu_{j}(u)=f_{j}(u)d\mu(u),\quad j=1,2
$$
and $d\nu_j^t(u)=\Phi(t)_*d\nu_j(u)$. Then for a test function $\Psi$, we can write
\begin{eqnarray*}
\int_{H^\sigma} \Psi(u)d\nu^t_{j}(u) & = & \int_{H^\sigma} \Psi(\Phi(t)(u)) d\nu_{j}(t)
\\
& = &
\int_{H^\sigma} \Psi(\Phi(t)(u)) f_j(u)d\mu(u)
\\
& = &
\int_{H^\sigma} \Psi(u) f_j(\Phi(-t)(u))G(t,u) d\mu(u)\,.
\end{eqnarray*}
Therefore $d\nu^t_{j}(u)=F_j(t,u)d\mu(u)$ with   $F_{j}(t,u)= f_j(\Phi(-t)(u))G(t,u)$.
Next, we can write 
\begin{eqnarray*}
\int_{H^\sigma}|F_1(t,u)-F_{2}(t,u)|d\mu(u) & = & \int_{H^\sigma}  |f_1(\Phi(-t)(u))-f_{2}(\Phi(-t)(u))|G(t,u)d\mu(u)
\\
& = & \int_{H^\sigma}|f_1(u)-f_2(u)|d\mu(u).
\end{eqnarray*}
This completes the proof of Corollary \ref{stability}.
\end{proof}
%%%%
\subsection{Almost sure convergence of smooth solutions}
In this section, we prove Theorem~\ref{thm5}. The key point is the following local stability result, which is a version of the enhanced local convergence.
\begin{proposition}[local stability]\label{enhanced-localconvergenec2}
	Assume that $\alpha,q,\epsilon$ be the numerical constants as in Proposition \ref{LWP-main}. Let $(\phi_{k})\subset \mathcal{V}^{q,\epsilon}+H^s$, $\phi\in \mathcal{V}^{q,\epsilon}+H^s$ satisfying
	$$ \|\phi_{k}\|_{\mathcal{V}^{q,\epsilon}+H^s}+\|\phi\|_{\mathcal{V}^{q,\epsilon}+H^s}\leq R,\quad \lim_{k\rightarrow\infty} \|\phi_k-\phi\|_{\mathcal{V}^{q,\epsilon}+H^s}=0.
	$$
	Assume moreover that
	\begin{equation*}
	\lim_{k\rightarrow\infty}\mathbf{d}(\phi_k,\phi)=0
	\end{equation*}
	and for all $k\in\N$,
	\begin{equation*}
	\mathcal{W}_{s,\epsilon}(\phi_k)\leq R^3, \quad \mathcal{W}_{s,\epsilon}(\phi)\leq R^3.
	\end{equation*}
	Then there exist $c>0, \kappa>0$, such that for all $t\in [-\tau_R,\tau_R]$ with $\tau_R=c(R+2)^{-\kappa}$, we have
	\begin{align*}
	\lim_{k\rightarrow\infty}\sup_{|t|\leq \tau_R}\|\Phi(t)\phi_{k}-\Phi(t)\phi \|_{\mathcal{V}^{q,\epsilon}+H^s}=0.
	\end{align*}
	Furthermore, with $\mathcal{J}_{k,t}=\{\Phi(t)\phi_k, \Phi(t)\phi  \}$, we have
	%	\begin{align}\label{ugly4}
	%	\sum_{ \substack{ f_j\in\mathcal{J}_{k,t}\\
	%	i_1,i_2,i_3\in\{0,1\}  } } \Gamma_{N_k,s,\epsilon}^{i_1,i_2,i_3}(f_1,f_2,f_3)\leq 2R,
	%	\end{align}
	\begin{align*}
	\lim_{k\rightarrow\infty}\mathbf{d}\big(\Phi(t)\phi_k,\Phi(t)\phi\big)=0.
	\end{align*}
\end{proposition}
\begin{remarque}
	Comparing with Proposition~\ref{enhanced-localconvergence}, the only difference here is that instead of comparing the flow $\Phi(t)\phi$ with the truncated truncated flow $\Phi_{N_k}(t)\phi_k$, we compare it with the real flow $\Phi(t)\phi_k$.
\end{remarque}
\begin{proof}
The proof is almost the same as the proof of Proposition \ref{enhanced-localconvergence}, and we only give a sketch. 
First we have the same decomposition  $\phi_k=\phi_{0,k}+r_{0,k}, \phi=\phi_0+r_0$ as in \eqref{decomposition} with the same property. Arguing as before, we have
$$ \mathbf{d}(\phi_{0,k},\phi_0)\rightarrow 0.
$$
On the same local existence time interval $[-\tau_R,\tau_R]$ as in Proposition~\ref{enhanced-localconvergence}, we have for any $|t|\leq \tau_R$, the difference of $\Phi(t)\phi_k-\Phi(t)\phi$ can be written as $\varphi_k(t)+\psi_k(t)$, where the $\mathcal{V}^{q,\epsilon}$ part is
$ \varphi_k(t)=b_k(t)S_{\alpha}(t)\phi_{0,k}-b(t)S_{\alpha}(t)\phi_0,$ with $b_k(t)=e^{\frac{it}{\pi}\|\phi_k\|_{L^2}^2 }, b(t)=e^{\frac{it}{\pi}\|\phi\|_{L^2}^2  }.
$  
The $H^s$ part is $\psi_k(t)=b_k(t)S_{\alpha}(t)r_{0,k}-b(t)S_{\alpha}(t)r_0+b_k(t)w_k(t)-b(t)w(t)$, where 
$$ \|r_{0,k}-r_0\|_{H^s}\rightarrow 0,\quad \sup_{|t|\leq \tau_R}\|w_k(t)-w(t) \|_{H^s}\rightarrow 0.
$$
Thus by quasi-invariance of the $\mathcal{V}^{q,\epsilon}$ norm and the quantity $\mathcal{W}_{s,\epsilon}(\cdot)$ under $S_{\alpha}(t)$, we deduce that 
$$ \sup_{|t|\leq \tau_R}\|\varphi_k(t)\|_{\mathcal{V}^{q,\epsilon}}\rightarrow 0, \quad \sup_{|t|\leq\tau_R}\|\psi_k(t)\|_{H^s}\rightarrow 0, \quad \mathbf{d}(\Phi(t)\phi_k, \Phi(t)\phi)\rightarrow 0.
$$
This completes the proof of Proposition~\ref{enhanced-localconvergenec2}.
\end{proof}
Now we prove Theorem \ref{thm5}.
\begin{proof}[Proof of Theorem \ref{thm5}]
We follow the argument in \cite{Sun-Tz}.  By the Borel-Cantelli lemma, it is sufficient to show that for any $T>0$, we have the almost convergence of the smooth solutions on the time interval $[0,T]$. We introduce an extra data set
$$ \widetilde{\Sigma}:=\bigcup_{l=1}^{\infty}\bigcap_{l'=l}^{\infty}\widetilde{\Sigma}_{l'},\quad
$$
where
$$ \widetilde{\Sigma}_{l}:=\big\{\phi\,:\, N^{\epsilon_0(i_1+i_2+i_3)}\mathcal{W}_{s,\epsilon}\big((\Pi_N^{\perp})^{i_1}\phi,(\Pi_N^{\perp})^{i_2}\phi,(\Pi_N^{\perp})^{i_3}\phi \big)\leq l^{\frac{3}{2}},\forall i_1,i_2,i_3\in\{0,1\} \big\},
$$
with $\epsilon_0>0$ as in Corollary \ref{convergence:W}. Consequently,
$$ \mu((\widetilde{\Sigma}_l)^c)<e^{-cl}
$$
and by Borel-Cantelli, $\widetilde{\Sigma}$ has full $\mu$ and $\rho$ measure.
Since $\Sigma$ constructed in Theorem~\ref{thm6} also has full $\rho$ measure, the proof will be finished once we show that for any $\phi\in\Sigma\cap\widetilde{\Sigma}$, the global solution $\Phi(t)(\Pi_N\phi )$ converges to $\Phi(t)\phi$ in $C([0,T];H^{\frac{\alpha-1}{2}-\epsilon}(\T))$. We will in fact prove the convergence in the stronger topology $C\big([0,T];\mathcal{V}^{q,\epsilon}+H^s\big)$.
\\

For any $\phi\in\Sigma\cap\widetilde{\Sigma}$, there exists $m\in\N$, such that $\phi\in\Sigma^m$. By Proposition \ref{longtime-2}, we have for all $|t|\leq T$, with $\Lambda_{m,T}= Cm^{3/2}(1+\log(1+|T|))^{3/2}$, we have
$$ \|\Phi(t)\phi\|_{\mathcal{V}^{q,\epsilon}+H^s}+\big(\mathcal{W}_{s,\epsilon}(\Phi(t)\phi)\big)^{\frac{1}{3}}\leq \Lambda_{m,T}.
$$
Moreover, from the construction of $\widetilde{\Sigma}$, $$\|\phi-\Pi_N\phi\|_{\mathcal{V}^{q,\epsilon}+H^s}\rightarrow 0,\quad  \mathbf{d}(\Pi_N\phi,\phi)\rightarrow 0,\text{ as }N\rightarrow\infty.
$$
Set $\phi_N=\Pi_N\phi$. We have that for $N\geq N_0$, large enough, 
$$ \|\phi_N\|_{\mathcal{V}^{q,\epsilon}}+\mathcal{W}_{s,\epsilon}(\phi_N)\leq 2\Lambda_{m,T}.
$$ 
Let $R=3\Lambda_{m,T}$ and we divide $[0,T]$ into $N_{R}\sim T/\tau_{R}$ intervals of equal length $\tau_{R}$. Applying Proposition \ref{enhanced-localconvergenec2} to $\phi_N,\phi$ and $R$, we obtain that for all $t\in[0,\tau_R]$,
 $$ \mathbf{d}(\Phi(t)\phi_N,\Phi(t)\phi )\rightarrow 0, \quad \sup_{t\in[0,\tau_R]}\|\Phi(t)\phi_N-\Phi(t)\phi\|_{\mathcal{V}^{q,\epsilon}+H^s}\rightarrow 0.
 $$
In particular,
$$ \|\Phi(t)\phi\|_{\mathcal{V}^{q,\epsilon}+H^s}=\lim_{N\rightarrow\infty}\|\Phi(t)\phi_N\|_{\mathcal{V}^{q,\epsilon}+H^s}.
$$
Furthermore, by definition and using the triangle inequality, we have
$$ \mathcal{W}_{s,\epsilon}(\Phi(t)\phi)=\lim_{N\rightarrow\infty}\mathcal{W}_{s,\epsilon}(\Phi(t)\phi_N).
$$
Therefore, for some $N_1\geq N_0$ and for all $N\geq N_1$,
$$ \|\Phi(\tau_R)\phi_N\|_{\mathcal{V}^{q,\epsilon}+H^s}+\big(\mathcal{W}_{s,\epsilon}(\Phi(t)\phi_N)\big)^{\frac{1}{3}}\leq 2\Lambda_{m,T}.
$$
This allows us to apply Proposition \ref{enhanced-localconvergenec2} to $\Phi(\tau_R)\phi_N,\Phi(\tau_R)\phi$ on $[\tau_R,2\tau_R]$. Successively, after $N_R$ steps, we prove the convergence of $\Phi(t)\phi_N$ to $\Phi(t)\phi$ to the whole interval $[0,T]$. 
\end{proof}
%%%%%%%%%%%%%%%%%%%%%%%%%%%%%%%%%%%%%%%%%%%%%%%%%%%%%%%%%%%%%%%%%%%%%%%%%%
%%%%%%%%%%%%%%%%%%%%%%%%%%%%%%%%%%%%%%%%%%%%%%%%%%%%%%%%%%%%%%%%%%%%%%%%%%%%%%%%%%%%%%%%%%%%%%%%%%%%%%%%%%%%%%%%%%%%%%%%%%%%%%%%%%%%%%%%%%%%%%%%%%%%%%%%%%%%%%%%%%%%%%%%%%%%%%%%%%%%%%%%%%%%%%%%%%%%%%%%%%%%%%%%%%%%%%%%%%%%%%%%%%%%%%%
\section{Convergence of the whole sequence of solutions for the truncated equation when $\alpha>1$ }
Recall that we denote by $\Phi_N(t)$ the flow of the truncated equation
$$ i\partial_tu+|D_x|^{\alpha}u+\Pi_N(|\Pi_Nu|^2\Pi_Nu)=0,\quad u|_{t=0}= \phi,
$$
defined on any Sobolev space $H^s(\T)$. The measure $\rho_N$ is invariant under $\Phi_N(t)$ and as a consequence we have the following statement. 
\subsection{New probabilistic a priori estimates}
\begin{lemme}\label{pointwiseinvariance}
	Let $F: H^{s_1}(\T)\rightarrow H^{s_2}(\T)$ ($s_1\geq s_2\geq 0$) be a measurable map with respect to the canonical Borel $\sigma$-algebras on $H^{s}(\T)$. 
	Then for every $t\in\mathbb{R}$,  and almost every $x\in\T$, we have
	$$ \mathbb{E}_{\rho_N}\left[F(\Phi_N(t)\phi)(x)\right]=\mathbb{E}_{\rho_N}\left[F(\phi)(x)\right],
	$$
	as soon as
	$$ \mathbb{E}_{\rho_N}[\|F(\phi)\|_{L^1(\T)}]<\infty.
	$$
	In particular, if for some Fourier multiplier $f(D_x)$ and some $1\leq q,r<\infty$, there holds
	\begin{equation*}
	 \left\|\mathbb{E}_{\rho_N}\left[\left|f(D_x)\phi(x)\right|^q\right]\right\|_{L^r(\T)}<\infty,
\end{equation*}
then we have for $1\leq \nu<\infty$,
	$$ \left\|\mathbb{E}_{\rho_N}\left[\left|f(D_x)(\Phi_N(t)\phi)(x)\right|^q\right]\right\|_{L^{\nu}([0,T];L^r(\T))}=T^{\frac{1}{\nu}}\left\|\mathbb{E}_{\rho_N}\left[\left|f(D_x)\phi(x)\right|^q\right]\right\|_{L^r(\T)}\,.
	$$
\end{lemme}
\begin{proof}
	Actually, the matter is to make the definition of  $x\mapsto \mathbb{E}_{\rho_N}[F(\phi)(x)]$ precise as an $L^1$ function on $\T$. 
	Define a function $\widetilde{F}$ from $H^{s_1}\times \T$ to $\mathbb{C}$ by
	$$ (\phi,x)\mapsto \widetilde{F}(\phi,x):=F(\phi)(x).
	$$
	From the assumption and the Fubini theorem, the function $\widetilde{F}$ is a well-defined $L^1$ function on $H^{s_1}\times \T$. Moreover, for a.e. $x\in \T$, the function
	$$ x\mapsto\mathbb{E}_{\rho_N}[F(\phi)(x)]:=\int_{H^{s_1}}\widetilde{F}(\phi,x)d\rho_N(\phi)
	$$
	is defined as a $L^1$ function on $\T$. 
	
	Now from the invariance of Gibbs measure $\rho_N$ on $H^{s_1}(\T)$ along $\Phi_N(t)$, we have that
	$$ \mathbb{E}_{\rho_N}[\|F(\Phi_N(t)\phi)\|_{L^1(\T)}]=\mathbb{E}_{\rho_N}[\|F(\phi)\|_{L^1(\T)}]<\infty.
	$$
	Thus $\mathbb{E}_{\rho_N}\left[F(\Phi_N(t)\phi)(x)\right] $ is defined for almost every $x\in \T$ as an $L^1$ function. 
	Now it remains to show the desired equality. 
	%Note that under the condition \eqref{finite}, $\mathbb{E}_{\rho_N}\left[F(\Phi_N(t)\phi)(x)\right] $ is a $L^r$ function on $\T$. 
	For any $\theta\in C^{\infty}(\T)$, we have from the Fubini theorem that
	\begin{equation*}
	\begin{split}
	\langle\mathbb{E}_{\rho_N}\left[F(\Phi_N(t)\phi)\right],\theta\rangle=&\int_{\T}\left(\int_{H^{s_1}}F(\Phi_N(t)\phi)(x)d\rho_N\right)\theta(x)dx\\
	=&\int_{H^{s_1}}\left(\int_\T F(\Phi_N(t)\phi)(x)\theta(x)dx\right)d\rho_N\\
	=&\int_{H^{s_1}}\langle F(\Phi_N(t)\phi),\theta\rangle d\rho_N\\
	=&\int_{H^{s_1}}\langle F(\phi),\theta\rangle d\rho_N,
	\end{split}
	\end{equation*}
	where in the last step we have used the invariance property by viewing $\phi\mapsto \langle F(\Phi_N(t)\phi),\theta\rangle$ as a continuous functional on $H^{s_1}(\T)$. Using Fubini again, we obtain that
	$$ \mathbb{E}_{\rho_N}[\langle F(\phi),\theta\rangle]=\langle\mathbb{E}_{\rho_N}[F(\phi)(\cdot)],\theta\rangle.
	$$
	This implies that for any $t\in\mathbb{R}$ and almost every $x\in\T$,
	$$
	\mathbb{E}_{\rho_N}[F(\Phi_N(t)\phi)(x)]=\mathbb{E}_{\rho_N}[F(\phi)(x)].
	$$ 
	Similarly, we define
	$$ \widetilde{G}(\phi,x):=\big(f(D_x)\phi\big)(x)
	$$
	and $x\mapsto \mathbb{E}_{\rho_N}\big[\widetilde{G}(\phi,x)\big]$ as a measurable function on $\T$. The same invariance argument as before yields
	$ \mathbb{E}\big[\widetilde{G}(\Phi_N(t)\phi,x) \big]=\mathbb{E}\big[\widetilde{G}(\phi,x)\big],
	$ for every $t\in\R$ and almost every $x\in\T$. The final conclusion is then immediate.
	This completes the proof of Lemma~\ref{pointwiseinvariance}.
\end{proof}
The following probabilistic estimate uses the invariant of the Gibbs measure for the truncated system.
\begin{lemme}\label{apriori-prob}
Let $T>0,\sigma<\frac{\alpha-1}{2}$ and $2\leq q,r<\infty$. There exist positive constants $C_{\sigma,\alpha, T,q,r}$ and  $c(\sigma,\alpha,T,q,r )$, such that for all $N\in\N$ and $\lambda>0$,
\begin{equation*}
\mu\left(\{\phi:\|\Phi_N(t)\phi\|_{L_t^qW_x^{\sigma,r}([0,T]\times \T) }>\lambda \}\right)<C_{\sigma,\alpha,T,q,r}\exp\big(-\lambda^{c(\sigma,\alpha,T,q,r )} \big).
\end{equation*}	
\end{lemme}
\begin{proof}
To simplify the notation, we will use $L_t^qW_x^{\sigma,r}$ instead of $L_t^qW_x^{\sigma,r}([0,T]\times \T)$ in the argument below.  Let $\lambda_1>0$ to be chosen later, we split
\begin{equation*}
\begin{split}
\mu\big(\{\phi: \|\Phi_N(t)\phi \|_{L_t^qW_x^{\sigma,r}}>\lambda \}\big)
\leq & \mu\big(\{\phi: \|\Phi_N(t)\phi \|_{L_t^qW_x^{\sigma,r}}>\lambda,\|\Pi_N\phi\|_{L_x^{4}}\leq \lambda_1 \}\big)
\\+&\mu\big(\{\phi: \|\Phi_N(t)\phi \|_{L_t^qW_x^{\sigma,r}}>\lambda,\|\Pi_N\phi\|_{L_x^{4}}> \lambda_1 \}\big).
\end{split}
\end{equation*}
Recall that $d\rho_N =\exp\big(-\frac{1}{4}\|\Pi_N\phi\|_{L_x^4}^4\big)d\mu$ is the associated Gibbs measure for the truncated system, 
and the first term on the right side of the last inequality is bounded from above by
$$ 
e^{\frac{1}{4}\lambda_1^4}\rho_N\big(\phi: \|\Phi_N(t)\phi\|_{L_t^qW_x^{\sigma,r} }>\lambda \big),
$$
while the second term can be bounded above by 
$ \exp\left(-c\lambda_1^2\right).$
It remains to estimate 
$$
\rho_N\big(\phi: \|\Phi_N(t)\phi\|_{L_t^qW_x^{\sigma,r} }>\lambda \big).
$$ 
Let $q_1\geq \max\{q,r \}$ which will be fixed later. Using Chebyshev's inequality and then Minkowski's inequality, we have
\begin{equation}\label{interchange}
 \rho_N\big(\big\{\phi: \|\Phi_N(t)\phi\|_{L_t^qW_x^{\sigma,r} }>\lambda \big\}\big)\leq \frac{1}{\lambda^{q_1}}\Big\|\Big(\int_{H^{\sigma}(\T)}\left|D^{\sigma}\Phi_N(t)\phi \right|^{q_1}d\rho_N\Big)^{\frac{1}{q_1}}\Big\|_{L_t^qL_x^r}^{q_1},
\end{equation}
Applying Lemma~\ref{pointwiseinvariance}, the right side of \eqref{interchange} can be bounded above by
$$ \frac{T^{\frac{q_1}{q}}}{\lambda^{q_1}}\Big\|\Big(\int \Big|D^{\sigma}\phi(x) \Big|^{q_1}d\rho_N\Big) \Big\|_{L_x^r}\leq\frac{T^{\frac{q_1}{q}}}{\lambda^{q_1}}\Big\|\Big(\int \Big|D^{\sigma}\phi(x) \Big|^{q_1}d\mu_N\Big) \Big\|_{L_x^r} \leq \frac{C^{q_1} T^{\frac{q_1}{q} }q_1^{\frac{q_1}{2}} }{\lambda^{q_1}},
$$
where we have used the Wiener chaos estimate for the random series
$$ \sum_{|n|\leq N}\frac{g_n(\omega)e^{inx}}{\langle n\rangle^{\alpha-2\sigma}},
$$
and the constant $C$  depends on $\alpha,\sigma,q,r$.
\\

Putting everything together, we obtain that
$$ \mu\big(\{\phi: \|\Phi_N(t)\phi \|_{L_t^qW_x^{\sigma,r}}>\lambda \}\big)\leq e^{\frac{\lambda_1^4}{4}}\Big(\frac{CT^{\frac{1}{q}}\sqrt{q_1} }{
\lambda} \Big)^{q_1}+e^{-c\lambda_1^2}.
$$
We take $q_1=\frac{\lambda^2}{A}$ with $A>C^2T^{\frac{q}{2}}$, then the first term on the right side is majorized by $\exp\left(\lambda_1^4/4-\lambda^2\log(A)/(2A) \right)$. Now we choose $\lambda_1=\lambda^{1/2}\left(\log A/A\right)^{1/4}$, thus
$$ \frac{\lambda_1^4}{4}-\frac{\lambda^2\log A}{2A}=-\frac{\lambda^2\log A}{4A}.
$$
With this choice, the proof of Lemma~\ref{apriori-prob} is now complete.
\end{proof}
The same argument as in the proof of Lemma~\ref{apriori-prob} yields the following statement. 
\begin{corollaire}\label{apriori-prob-cor}
Under the same restriction on the numerologies, we have for all $M<N$ and $\lambda>0$, 
\begin{equation*}
\mu\big(\big\{\phi:\|\Pi_M^{\perp}\Phi_N(t)\phi\|_{L_t^qW_x^{\sigma,r}([0,T]\times \T) }>\lambda \big\}\big)\leq C(\alpha,\sigma,T,q,r)\exp\big(-(\theta\lambda)^{c(\alpha,\sigma,T,q,r)}\big),
\end{equation*}
with $\theta=\theta(T,M)=T^{-\frac{1}{q}}M^{\alpha-1-2\sigma}$.
\end{corollaire}
\subsection{The convergence argument}
In this subsection, we prove the Theorem \ref{thm3}. By a Borel-Cantelli type argument, it is sufficient to prove the convergence of the sequence $(u_N)_{N\in\N}$ of the truncated equations
$$ i\partial_tu+|D_x|^{\alpha}u+\Pi_N(|\Pi_Nu|^2\Pi_Nu)=0, \quad u|_{t=0}=\phi
$$
on $C([0,T];H^{\sigma}(\T))$ for any given $T>0$, where $0<\sigma<\frac{\alpha-1}{2}$. To simplify the notation, we will denote by $v_N(t)=\Pi_N\Phi_N(t)\phi$, which is the low frequency portion of the solution $\Phi_N(t)\phi$. Because $v_N=\Pi_Nv_N$, $v_N(t)$ satisfies the same equation
$$ i\partial_tv_N+|D_x|^{\alpha}v_N+\Pi_N(|v_N|^2v_N)=0.
$$
Since the high frequency part $\Pi_N^{\perp}\Phi_N(t)$ solves the linear equation, 
 it suffices to prove the convergence of the sequence $(v_N)_{N\geq 1}$.
We will simply write $L_t^qW_x^{s,r}$ to stand for the space-time norm $L^q([0,T];W^{s,r}(\T))$, and $L_{t}^qW_x^{s,r}(I)$ the norm $L^q(I;W^{s,r}(\T))$, where $I\subset \R$ is a time interval.

\noi
$\bullet$ \textbf{Step 1: A deterministic estimate.
}\\
Pick $\sigma_1\in\left(\sigma,\frac{\alpha-1}{2}\right)$, $r>\frac{2}{\sigma_1-\sigma}$, $2<q<\infty$, large enough, and $B(N)<N$ to be determined later. For each $N$, we 
associate with a small number $\eta=\eta(N)>0$ and partition the interval $[0,T]$ into $T/\eta$ intervals enabled as $I_j=[t_j,t_{j+1}]$ with length $\eta$. Let $N_1\in [N,2N]$. With $F(v)=|v|^2v$, we write
\begin{equation*}
\begin{split}
v_{N_1}(t)-v_N(t)=&S_{\alpha}(t-t_j)(v_{N_1}(t_j)-v_N(t_j) )
-i\int_{t_j}^tS_{\alpha}(t-t')\Pi_{N}^{\perp}\Pi_{N_1}F(v_{N_1})(t')dt'\\
-&i\int_{t_j}^tS_{\alpha}(t-t')\Pi_N[F(v_{N_1}(t')-F(v_N)(t') ]dt'\\
=&:\mathrm{I}_j+\mathrm{II}_j+\mathrm{III}_j
\end{split}
\end{equation*}
with respectively. For I$_j$, we estimate it simply by
\begin{equation}\label{Ij}
 \|\mathrm{I}_j\|_{L_t^{\infty}H_x^{\sigma}(I_j)}\leq \|v_{N_1}(t_j)-v_N(t_j)\|_{H_x^{\sigma}}.
\end{equation} 
For II$_j$, using H\"older's inequality and the product rule, we have
\begin{equation}\label{II_j}
\begin{split}
\|\mathrm{II}_j\|_{L_t^{\infty}H_x^{\sigma}}\leq& N^{-(\sigma_1-\sigma)}\|F(u_{N_1})\|_{L_t^1H_x^{\sigma_1}(I_j)}\\
\lesssim & N^{-(\sigma_1-\sigma)}\|u_{N_1}\|_{L_t^{2q}H_x^{\sigma_1}(I_j)}\|u_{N_1}\|_{L_t^{2(2q)'}L_x^{\infty}(I_j)}^2.
\end{split}
\end{equation}
To estimate III$_j$, note that by triangle inequality, we have
\begin{equation*}\label{IIIj-1}
\begin{split}
\|\mathrm{III}_j\|_{L_t^{\infty}H_x^{\sigma}}\leq & \||v_{N_1}|^2(v_{N_1}-v_N) \|_{L_t^1H_x^{\sigma}(I_j)}+\|\ov{u}_{N_1}v_N(v_{N_1}-v_N) \|_{L_t^1H_x^{\sigma}(I_j)}\\
+&\|(\ov{u}_{N_1}-\ov{u}_N)v_N^2\|_{L_t^1H_x^{\sigma}(I_j)}
\end{split}
\end{equation*} 
Applying Lemma \ref{nonlinear2}, the right side can be majorized by
$$ \|v_{N_1}-v_N\|_{L_t^{q'}H_x^{\sigma}(I_j)}\left(\||v_{N_1}|^2\|_{L_t^{q}B_{r,2}^{\sigma_2}(I_j) }+\||v_{N}|^2\|_{L_t^{q}B_{r,2}^{\sigma_2}(I_j) } \right)
$$
where $\sigma_2=\frac{\sigma_1+\sigma}{2}$ and $r>\frac{2}{\sigma_1-\sigma}=\frac{1}{\sigma_2-\sigma}$. Applying Lemma \ref{nonlinear1} and using the fact that $W^{\sigma_1,r}$ is embedded into $B_{r,2}^{\sigma_2}$, we have
\begin{equation*}
\begin{split}
 \|\mathrm{III}_j\|_{L_t^{\infty}H_x^{\sigma}}\lesssim &\|v_{N_1}-v_N\|_{L_t^{q'}H_x^{\sigma}(I_j)}\|v_{N_1}\|_{L_t^{2q}L_x^{\infty}(I_j)}\|v_{N_1}\|_{L_t^{2q}W_x^{\sigma_1,r}(I_j)}\\+&\|v_{N_1}-v_N\|_{L_t^{q'}H_x^{\sigma}(I_j)} \|v_{N}\|_{L_t^{2q}L_x^{\infty}(I_j)}\|v_{N}\|_{L_t^{2q}W_x^{\sigma_1,r}(I_j)}.
\end{split}
\end{equation*}
Thus
\begin{equation}\label{III_j}
\|\mathrm{III}_j\|_{L_t^{\infty}H_x^{\sigma}(I_j)}\lesssim \eta^{\frac{1}{q'}}\|v_{N_1}-v_N\|_{L_t^{\infty}H_x^{\sigma}(I_j)}\sum_{\nu=0}^1\|v_{ N_{\nu}}\|_{L_t^{2q}L_x^{\infty}(I_j)}\|v_{ N_{\nu}}\|_{L_t^{2q}W_x^{\sigma_1,r}(I_j)}.
\end{equation}
Note that $W^{\sigma_1,r}$ is embedded into $L^{\infty}$, combing \eqref{Ij},\eqref{II_j} and \eqref{III_j}, we have
\begin{equation}\label{BB-mainestimate}
\begin{split}
\|v_{N_1}-v_N\|_{L_t^{\infty}H_x^{\sigma}(I_j)}\leq &\|v_{N_1}(t_j)-v_N(t_j)\|_{H_x^{\sigma}}+C_TN^{-(\sigma_1-\sigma)}\|v_{N_1}\|_{L_t^{2q}W_x^{\sigma_1,r}}^2\\
+&C\eta^{1/q'}\|v_{N_1}-v_N\|_{L_t^{\infty}H_x^{\sigma}(I_j)}\sum_{\nu=0}^{1}\|v_{N_{\nu}}\|_{L_t^{2q}W_x^{\sigma_1,r}}^2,
\end{split}
\end{equation}
provided that $2(2q)'<2q$, if $q$ is chosen large enough. Note that the constant $C$ depends on $\sigma_1,\sigma,q,r$. 

Assume for the moment that $$\|v_{N}\|_{L_t^{2q}W_x^{\sigma_1,r}}<B(N),\quad \|v_{N_1}\|_{L_t^{2q}W_x^{\sigma_1,r}}<5B(N).
$$
We take $\eta=(8CB(N))^{-q'}$, it follows from \eqref{BB-mainestimate} that
$$ \|v_{N_1}-v_N\|_{L_t^{\infty}H_x^{\sigma}(I_j)}\leq 2\|v_{N_1}(t_j)-v_N(t_j)\|_{H_x^{\sigma}}+C_T'N^{-(\sigma_1-\sigma)}B(N)^2.
$$
Consequently, if 
$$ B(N)<N^{\frac{\sigma_1-\sigma}{4}}, 
$$
by iteration, we obtain that 
\begin{equation*}
\begin{split}
\|v_{N_1}-v_N\|_{L_t^{\infty}H_x^{\sigma}}\leq & 2^{\frac{T}{\eta}+1}\left(\|v_{N_1}(0)-v_N(0)\|_{H_x^{\sigma}}+N^{-\frac{\sigma_1-\sigma}{2}} \right)\\
\leq &\exp\Big(2T\log_2(4CB(N)^2 )^{q'}\Big)N^{-\frac{\sigma_1-\sigma}{2}}.
\end{split}
\end{equation*}
We take
$$ B(N)=(c_1\log N)^{\frac{1}{2q'}},
$$
for some suitable $c_1=c_1(T,\sigma,\sigma_1)$, small enough, the right hand side of the inequality above can be majorized by $N^{-\frac{\sigma_1-\sigma}{4}}$.

\noi
$\bullet$ {\bf Step 2: Good data set.}

For any dyadic number $N$, we define the set
\begin{equation*}
\begin{split}
\Omega_{N}:=&\{\phi: \|\Pi_{N}^{\perp}\phi\|_{H_x^\sigma}<N^{-(\sigma_1-\sigma)},\|\Pi_N\Phi_N(t)\phi\|_{L_t^{2q}W_x^{\sigma_1,r} }+\|\Pi_{2N}\Phi_{2N}(t)\phi\|_{L_t^{2q}W_x^{\sigma_1,r} }<B(N) \}\\
\cap& \{\phi:\max_{N\leq N_1\leq 2N}\|\Pi_{M_0}^{\perp}\Phi_{N_1}(t)\phi\|_{L_t^qW_x^{\sigma_1,r}}\leq 1 \}
\end{split}
\end{equation*}
where $M_0=M_0(N)$ will be chosen later. From Lemma \ref{apriori-prob} and Lemma  \ref{apriori-prob-cor}, we have
$$ \mu(\Omega\setminus\Omega_N)<e^{-B(N)^c}+Ne^{-T^{\frac{c}{2q}}M_0^{(\alpha-1-2\sigma)c}}.
$$
The choice of $B(N)$ and $M_0$ should assure that the series
$$ \sum_{k=0}^{\infty}\mu(\Omega\setminus\Omega_{2^k})
$$
converges.
We first choose
$$ M_0=(\log N)^{C_0}
$$
with $C_0=C_0(q,r,\sigma_1,\sigma,T)$ large enough, such that $\displaystyle{\sum_{k=0}^{\infty}2^k\exp\big(-T^{\frac{c}{2q}}k^{C_0}\big)}<\infty$, while 
$$ B(N)=(c_1\log N)^{\frac{1}{2q'}}
$$
for some small constant $c_1>0$ to be fixed later. The good data set is then chosen as
$$ \mathcal{G}:=\bigcup_{m=0}^{\infty}\bigcap_{k=m}^{\infty}\Omega_{2^k},
$$
which has full $\mu$ measure, thanks to Borel-Cantelli.

\noi
$\bullet$ {\bf Step 3: Continuity argument.}
 
  Fix $\phi\in \mathcal{G}$, our goal is to show that the sequence $(\Phi_N(t)\phi)_N$ is Cauchy in $C([0,T];H^{\sigma}(\T))$.  Recall the notation $v_N(t)=\Pi_N\Phi_N(t)\phi$. By definition, there exists $k_0\in\N$, such that $\phi\in \Omega_{2^k}$ for all $k\geq k_0$. 
 Denote by $N_0=2^k$ for some $k\geq k_0$. We claim that for all large $N_0$ and $N_0\leq N_1\leq 2N_0$, $\|v_{N_1}\|_{L_t^{2q}W_x^{\sigma_1,r}}<4B(N_0)$.
  
  Indeed, for fixed $N_0$ and $N_1$, we define the set
  $$ S:=\{T'\in[0,T]:\|v_{N_1}\|_{L_t^{2q}W_x^{\sigma_1,r}([0,T'])}<4B(N_0) \}.
  $$
  We first show that $S$ is not empty. Note that $v_{N_1}(t)$ takes value in a finite dimensional space and by conservation of $L^2$ norm, $\|v_{N_1}\|_{L_t^{\infty}L_x^2}=\|\Pi_{N_1}\phi\|_{L_x^2}.$ Then by the equivalence of the norm, there exists $K_{N_1}>0$, such that
  $$ \|v_{N_1}\|_{L_t^{\infty}W_x^{\sigma_1,r}([0,\delta])}\leq K_{N_1}\|\Pi_{N_1}\phi\|_{L_x^2}.
  $$
  Coming back to the definition of $\Omega_{N_0}$
  Thus by H\"older's inequality, 
  $$ \|v_{N_1}\|_{L_t^{2q}W_x^{\sigma_1,r}([0,\delta]) }\leq \delta^{\frac{1}{2q}}K_{N_1}\|\Pi_{N_1}\phi\|_{L_x^2}.
  $$
  Hence if $\delta=\delta_N$ is small enough, $\|v_{N_1}\|_{L_t^{2q}W_x^{\sigma_1,r}([0,\delta_N])}<4B(N_0)$. In particular, $S\neq \emptyset$.
  
  Next we show that $S=[0,T]$. We argue by contradiction. Suppose that $T_0=\sup S<T$. By continuity of the function $$t'\mapsto \|v_{N_1}\|_{L_t^{2q}W_x^{\sigma_1,r}([0,t'])},$$
  there exists $\delta'>0$, $T_0+\delta'<T$, such that
  $$ \|v_{N_1}\|_{L_t^{2q}W_x^{\sigma_1,r}([0,T_0+\delta'])}<5B(N_0).
  $$
  Then from the argument in the last part of Step 1, we obtain that
  $$ \|v_{N_1}-v_{N_0}\|_{L_t^{\infty}H_x^{\sigma}([0,T_0+\delta'])}<N_0^{-\frac{\sigma_1-\sigma}{4}}.
  $$
    Notice that if $N_0\leq N_1<2N_0$, we have
  \begin{equation*}
  %\label{boundN1}
  \begin{split}
  \|v_{N_1}\|_{L_t^{2q}W_x^{\sigma_1,r}([0,T_0+\delta'])}\leq & \|\Pi_{M_0}^{\perp}v_{N_1}\|_{L_t^{2q}W_x^{\sigma_1,r}([0,T_0+\delta'])}+ \|\Pi_{M_0}(v_{N_1}-v_{N_0})\|_{L_t^{2q}W_x^{\sigma_1,r}([0,T_0+\delta'])}\\
  +&\|\Pi_{M_0}^{\perp}v_{N_0}\|_{L_t^{2q}W_x^{\sigma_1,r}([0,T_0+\delta'])}+\|v_{N_0}\|_{L_t^{2q}W_x^{\sigma_1,r}([0,T_0+\delta'])} \\
 \leq &2+T^{\frac{1}{2q}}M_0^{\sigma_1-\sigma+\frac{1}{2}-\frac{1}{r}}\|v_{N_1}-v_{N_0}\|_{L_t^{\infty}H_x^{\sigma}([0,T_0+\delta'])}+B(N_0)\\
  \leq &2+2B(N_0)+T^{\frac{1}{2q}}(\log N_0)^{2C_0}\|v_{N_1}-v_{N_0}\|_{L_t^{\infty}H_x^{\sigma}([0,T_0+\delta'])}.
  \end{split}
 \end{equation*}
For $N_0\gg 1$, the first and third terms are strictly smaller than $B(N_0)$, thus
$$ \|v_{N_1}\|_{L_t^{2q}W_x^{\sigma_1,r}([0,T_0+\delta'])}<4B(N_0),
$$
which is a contradiction.
\\

Now since $S=[0,T]$, we have that for any $N_1\in [N_0,2N_0]$, 
$$ \|v_{N_1}-v_{N_0}\|_{L_t^{\infty}H_x^{\sigma}}<N_0^{-\frac{\sigma_1-\sigma}{4}}.
$$
This implies that $(v_{N}(t))_{N}$ is a Cauchy sequence in $C([0,T];H^{\sigma}(\T))$. Since $\Pi_N^{\perp}\Phi_N(t)\phi=\Pi_N^{\perp}S_{\alpha}(t)\phi$ is linear, it is automatically a Cauchy sequence in $C([0,T]; H^{\sigma}(\T))$. The proof of Theorem~\ref{thm3} is now complete.
%%%%%%%%%%%s%%%%%%%%%%%%%%%%%%%%%%%%%%%%%%%%%%%%%%%%%%%%%%%%%%%%
\section{Weak dispersion case: $\alpha< 1$}

\subsection{Definition of Gibbs measure}
Recall that the renormalized Hamiltonian 
$$ H_N(u)=\int_{\T}||D_x|^{\frac{\alpha}{2}}u|^2+\frac{1}{2}\int_{\T}|\Pi_Nu|^4-2\alpha_N\int_{\T}|\Pi_N u|^2+\alpha_N^2,
$$
where
$$ \alpha_N=\mathbb{E}[\|\Pi_Nu\|_{L^2}^2].
$$
Consider the equation
$$  i\partial_tu=\frac{\delta H_N}{\delta \ov{u}},
$$
which reads
\begin{equation*}\label{equation-renormalized}
 i\partial_tu+|D_x|^{\frac{\alpha}{2}}u+F_N(u)=0,
\end{equation*}
where 
$$ F_N(u)=\Pi_N(|\Pi_Nu|^2\Pi_Nu)-2\alpha_N\Pi_Nu.
$$
Let $\mathcal{X}=H^{\frac{\alpha-1}{2}-\epsilon}(\T)$. The first step is to show that the sequence $(F_N(u))_{N\geq 1}$ is a Cauchy sequence in $L^p\big(\mathcal{X},\mathcal{B},\mu; H^{-\sigma}(\T)\big)$.  
We need a large deviation lemma. Let
$$ b_N(u):=\|\Pi_Nu\|_{L^2}^2-\alpha_N.
$$
\begin{lemme}\label{large-deviation}
	There exist $C,c>0$ so that for all $1\leq M<N$ large enough, and all $\lambda>0$, we have
	$$ \mu(\{u: |b_N(u)-b_M(u)|>\lambda  \})\leq Ce^{-c\lambda M^{\alpha}},\text{ if } \lambda\gtrsim M^{1-\alpha},
	$$ 
	and
	$$ \mu(\{u: |b_N(u)-b_M(u)|>\lambda  \})\leq Ce^{-c\lambda^2 M^{2\alpha-1}},\text{ if } \lambda\ll M^{1-\alpha}.
	$$ 
\end{lemme}
\begin{remarque}
{\rm
If we use Lemma~4.8 of \cite{Tz-ptrf}  (based on Wiener chaos estimates) we obtain the rougher bound $Ce^{-\lambda M^{\frac{\alpha}{2}}}$, which is enough for our purposes. Here we give an estimate which is of its own interest.
}
\end{remarque}
\begin{proof}
	Denote by
	$$  R_{M,N}(\omega):=\sum_{M\leq |n|\leq N} \frac{|g_n(\omega)|^2}{([n]^{\frac{\alpha}{2}})^2},
	$$
	where $g_n(\omega)=\frac{h_n(\omega)+il_n(\omega)}{\sqrt{2}}$ and $\mathbb{E}[|g_n|^2]=1$.
	We have
	\begin{equation*}\label{a3}
	\mu\{u: |b_N(u)-b_M(u)|>\lambda  \}=\mathbb{P}\{\omega: |R_{M,N}(\omega)-\mathbb{E}[R_{M,N}]|>\lambda    \},
	\end{equation*}
	where
	$$ R_{M,N}(\omega)-\mathbb{E}[R_{M,N}]=\sum_{M\leq |n|\leq N}a_nX_n(\omega),\quad a_n=([n]^{-\frac{\alpha}{2}})^2,  X_n(\omega)=|g_n(\omega)|^2-1.
	$$
	\begin{align*}
	\mathbb{P}\Big\{\omega: \Big|\sum_{M\leq|n|\leq N} a_nX_n(\omega)\Big|>\lambda \Big\}\leq &\mathbb{P}\Big\{\omega: \sum_{M\leq|n|\leq N} a_nX_n(\omega)>\lambda \Big\}\\
	+&\mathbb{P}\Big\{\omega: \sum_{M\leq |n|\leq N}a_nX_n(\omega)<-\lambda \Big\}.
	\end{align*}
	First we estimate the probability of the event $\{\sum_{M\leq |n|\leq N}a_nX_n>\lambda\}$.
	For any $\theta>0$, we have
	\begin{align}\label{exponentialProb}
	\mathbb{P}\Big\{\omega: \sum_{M\leq|n|\leq N} a_nX_n(\omega)>\lambda \Big\}=\mathbb{P}\Big\{\omega: e^{\sum_{M\leq |n|\leq N}\theta a_nX_n(\omega)}>e^{\theta\lambda} \Big\}
	\end{align}
	Using Chebyshev's inequality, the r.h.s. of \eqref{exponentialProb} can be bounded by
	\begin{align*} e^{-\theta\lambda}\mathbb{E}\Big[e^{\sum_{M\leq|n|\leq N}\theta a_n (|g_n|^2-1) } \Big]\leq e^{-\theta\lambda} e^{-\sum_{M\leq |n|\leq N}\theta a_n }\prod_{M\leq|n|\leq N}\mathbb{E}\big[e^{\theta a_n|g_n|^2} \big],
	\end{align*}
	where we have used the independence. Since each $g_n$ can be identified as a standard two dimensional Gaussian random variable, we have
	\begin{align*}
	\prod_{M\leq |n|\leq N} \mathbb{E}\big[e^{\theta a_n|g_n|^2} \big]=\prod_{M\leq|n|\le N}\Big(\frac{1}{2\pi}\int_{\R^2} e^{-\frac{ |z|^2}{2}(1-\theta a_n) }dz\Big)
	=\prod_{M\leq |n|\leq N}\frac{1}{1-\theta a_n},
	\end{align*}
	provided that $\theta a_n<1$. We will finally choose suitable $\theta$ such that $\theta a_n<\frac{1}{2}$. From the elementary inequality
	$$ -y-\log(1-y)\leq C_0y^2,
	$$
	uniformly in $0<y<\frac{1}{2}$, we deduce that
	\begin{align*}
	e^{-\theta\lambda-\sum_{M\leq |n|\leq N}\theta a_n }\prod_{M\leq |n|\leq N} \mathbb{E}\Big[e^{\theta a_n|g_n|^2} \Big]
	=& e^{-\theta\lambda+\sum_{M\leq |n|\leq N}(-\theta a_n -\log (1-\theta a_n)) }\\
	\leq & e^{-\theta\lambda +C_0\sum_{M\leq |n|\leq N}\theta^2a_n^2 }\leq e^{-\theta \lambda +C_0\theta^2 \epsilon_M},
	\end{align*}
	where 
	$$ \epsilon_M=\sum_{|n|\geq M} a_n^2\sim M^{-(2\alpha-1)}. 
	$$
	Similarly, for the event $\{\omega: \sum_{M\leq |n|\leq N}a_nX_n(\omega)<-\lambda  \}$, we can rewrite it as
	$$ \big\{\omega: e^{\theta \sum_{M\leq |n|\leq N}a_n(1-|g_n|^2) } >e^{\lambda\theta}\big\}.
	$$
	Again by Chebyshev, the probability of this event is bounded by
	$$ e^{-\lambda \theta}\mathbb{E}\Big[e^{\sum_{M\leq |n|\leq N}\theta a_n(1-|g_n|^2) }\Big]=e^{-\lambda\theta+\sum_{M\leq |n|\leq N}\theta a_n } \prod_{M\leq |n|\leq N }\mathbb{E}[e^{-\theta a_n|g_n|^2}]
	$$
	Again from
	$$ \mathbb{E}[e^{-\theta a_n|g_n|^2}]=\frac{1}{2\pi}\int_{\R^2}e^{-\frac{|z|^2}{2}(1+\theta a_n)}dz=\frac{1}{1+\theta a_n},
	$$
	we have
	$$\mathbb{P}\big\{\omega: \sum_{M\leq |n|\leq N}a_nX_n(\omega)<-\lambda  \big\}\leq e^{-\lambda\theta +\sum_{M\leq |n|\leq N}[\theta a_n-\log(1+\theta a_n)] }.
	$$
	From the inequality
	$$ y-\log(1+y)\leq \frac{y^2}{2},\quad \forall \,0<y<1,
	$$
	we have
	$$\mathbb{P}\big\{\omega: \sum_{M\leq |n|\leq N}a_nX_n(\omega)<-\lambda  \big\}\leq e^{-\lambda\theta+\frac{1}{2}\sum_{M\leq|n|\leq N}\theta^2a_n^2 }=e^{-\lambda\theta+\frac{1}{2}\epsilon_M\theta^2}.
	$$
	In summary, we have that, for all $\theta>0,\lambda>0$
	$$ \mathbb{P}\big\{\omega: \big|\sum_{M\leq |n|\leq N}a_nX_n(\omega)\big|>\lambda  \big\}\leq 2e^{-\lambda\theta+C_0\epsilon_M\theta^2}.
	$$
	The function $\theta\mapsto -\lambda\theta+C_0\epsilon_M\theta^2$ attains its minimum at $\theta_0=\frac{\lambda}{2C_0\epsilon_M}\sim \lambda M^{2\alpha-1}$. If
	$$ \frac{\lambda}{2C_0\epsilon_M}\leq \frac{M^{\alpha}}{4},\quad \text{i.e.} \quad\lambda\leq \frac{C_0\epsilon_MM^{\alpha}}{2}\sim M^{1-\alpha},
	$$
	we choose $\theta=\theta_0$ (thus the condition $\theta a_n\leq \frac{1}{2}$ for all $M\leq|n|\leq N$ are satisfied), and we deduce that the desired probability is bounded by
	$ 2e^{-\frac{\lambda^2}{4C_0\epsilon_M}}\leq e^{-c'\lambda^2M^{2\alpha-1}}.
	$
	Otherwise
	$$ \lambda>\frac{C_0\epsilon_M M^{\alpha}}{2},\quad\text{i.e.}\quad \frac{\lambda}{2}>\frac{C_0\epsilon_M M^{\alpha}}{4}\sim M^{1-\alpha}.
	$$
	we take $\theta=\frac{M^{\alpha}}{4}$, and the desired probability is bounded by
	$$ 2e^{-\frac{\lambda M^{\alpha}}{4}+C_0\epsilon_M\frac{M^{2\alpha}}{4^2} }=2e^{-\frac{M^{\alpha}}{4}\big(\lambda-\frac{C_0\epsilon_MM^{\alpha}}{4} \big)}\leq 2e^{-\frac{M^{\alpha}\lambda}{8}}.
	$$
	The proof of Lemma~\ref{large-deviation} is now complete. 
\end{proof}

\begin{proposition}\label{Cauchysequence}
Assume that $\frac{2}{3}<\alpha<1$ and $\sigma>\frac{3(1-\alpha)}{2}$. For all $p\geq 2$, the sequence $(F_N(u))_{N\geq 1}$ is a Cauchy sequence in the space $L^p(\mathcal{X},\mathcal{B},\mu; H^{-\sigma}(\T )$. More precisely, there exists $\epsilon_0>0, C>0$, such that for all $1\leq M<N$, 
$$ \int_{\mathcal{X}} \|F_N(u)-F_M(u)\|_{H^{-\sigma}(\T)}^pd\mu(u)\leq \frac{C}{M^{\epsilon_0}}.
$$ 
\end{proposition}
\begin{proof}
We prove for $p=2$, and the estimate for the other values of $p$ will follow from Wiener chaos estimates. 
Note that $F_N(u)=G_N(u)+2b_N(u)\Pi_Nu$ where
$$ G_N(u)=\Pi_N\big(|\Pi_Nu|^2\Pi_Nu \big)-2\|\Pi_Nu\|_{L^2}^2\Pi_Nu.
$$
 Therefore, from Lemma \ref{Cauchysequence} and Lemma \ref{large-deviation}, it suffices to obtain the same type of estimate for
$$ \int_{\mathcal{X}}\|G_N(u)-G_M(u)\|_{H^{-\sigma}(\T)}^2d\mu(u).
$$
Write
$$ \chi_N:=|\phi_N^{\omega}|^2\phi_N^{\omega}-2\|\phi_N^{\omega}\|_{L^2(\T)}^2\phi_N^{\omega},
$$
and it suffices to show that
\begin{equation*}\label{a}
 \mathbb{E}\left[
\|\chi_N-\chi_M\|_{H^{-\sigma}(\T)}^2 \right]\leq \frac{C}{M^{\epsilon_0}}.
\end{equation*}
From the definition of $\phi_N^{\omega}$, we have
\begin{equation*}
\begin{split}
\chi_N=\sum_{\substack{|n_1|,|n_2|,|n_3|\leq N \\ n_2\neq n_1,n_3} }\frac{g_{n_1}\ov{g}_{n_2}g_{n_3}}{[n_1]^{\frac{\alpha}{2}}[n_2]^{\frac{\alpha}{2}}[n_3]^{\frac{\alpha}{2}}}e^{i(n_1-n_2+n_3)x},
\end{split}
\end{equation*}
and
$$ \chi_N-\chi_M=\sum_{n\in\Z}e^{inx}\sum_{B_{M,N}^{(n)}} \frac{g_{n_1}\ov{g}_{n_2}g_{n_3}}{ [n_1]^{\frac{\alpha}{2}}[n_2]^{\frac{\alpha}{2}}[n_3]^{\frac{\alpha}{2}} },
$$
where
\begin{equation*}
\begin{split}
B_{M,N}^{(n)}=\{(n_1,n_2,n_3)\in \Z^3: & |n_1|,|n_2|,|n_3|\leq N, n_2\neq n_1,n_3\\
\textrm{ and }& |n_1|>M \textrm{ or }|n_2|>M \textrm{ or }|n_3|>M,\\
&n_1-n_2+n_3=n   \}.
\end{split}
\end{equation*}
Since $(g_{n})$ are independent and centered, we deduce that
\begin{equation*}\label{a1}
\begin{split}
\mathbb{E}[ \|\chi_N-\chi_M\|_{H^{-\sigma}(\T)}^2 ]=&\sum_{n\in\Z}\frac{1}{\langle n\rangle^{2\sigma}}\mathbb{E}\Big[ \Big|
\sum_{B_{M,N}^{(n)}}\frac{g_{n_1}\ov{g}_{n_2}g_{n_3}}{[n_1]^{\frac{\alpha}{2}}[n_2]^{\frac{\alpha}{2}}[n_3]^{\frac{\alpha}{2}}   }\Big|^2 \Big]\\
\leq &\sum_{n\in\Z}\frac{C}{\langle n\rangle^{2\sigma}}\sum_{\substack{n_1-n_2+n_3=n\\  M<\max\{|n_1|,|n_2|,|n_3|\}\leq N  \\
n_2\neq n_1, n_3 }}\frac{1}{\langle n_1\rangle^{\alpha}\langle n_2\rangle^{\alpha}\langle n_3\rangle^{\alpha}  }.
\end{split}
\end{equation*}
To estimate the second summation, without loss of generality, we may assume that $|n_1|\geq M$. Then applying Lemma \ref{lemme-summation}, the second summation can be estimated by
$$ \sum_{M<|n_1|\leq N}\frac{C_{\gamma}}{\langle n_1\rangle^{\alpha}\langle n-n_1 \rangle^{\gamma} } 
$$
for some $\gamma<2\alpha-1$. If $\alpha>\frac{2}{3}$, then $3\alpha-2>0$, and we can choose $\gamma>0$ such that $\alpha+\gamma>1$. If $|n|\ll M$, then 
$$ \sum_{|n|\ll M,|n_1|>M}\frac{C_{\gamma}}{\langle n\rangle^{2\sigma}\langle n_1\rangle^{\alpha}\langle n-n_1\rangle^{\gamma} }\leq \sum_{|n|\ll M}\frac{C_{\gamma}}{\langle n\rangle^{2\sigma} M^{3\alpha-2}}\leq \frac{C_{\gamma}}{M^{\epsilon_0}},
$$
provided that $\sigma>\frac{3(1-\alpha)}{2}$. If $|n|\gtrsim M$, we separate the region of summation into $|n-n_1|<\frac{|n_1|}{2}$, $\frac{|n_1|}{2}\leq|n-n_1|<2|n_1|$ and  $|n-n_1|\geq 2|n_1|$. We have
$$ \sum_{|n_1|>M,|n-n_1|<\frac{|n_1|}{2}}\frac{C_{\gamma}}{\langle n\rangle^{2\sigma}\langle n_1\rangle^{\alpha}\langle n-n_1\rangle^{\gamma} }\leq \sum_{|n_1|>M}\frac{C_{\gamma}|n_1|^{1-\gamma}}{\langle n_1\rangle^{\alpha+2\sigma}}\leq \frac{C_{\gamma}}{M^{\epsilon_0}},
$$
provided that $\sigma>\frac{3(1-\alpha)}{2}$. If $\frac{|n_1|}{2}\leq |n-n_1|<4|n_1|$, we have
$$ \sum_{\substack{|n_1|>M,|n|\gtrsim M,\\
\frac{|n_1|}{2}\leq|n-n_1|<4|n_1| } }\frac{C_{\gamma}}{\langle n\rangle^{2\sigma}\langle n_1\rangle^{\alpha}\langle n-n_1\rangle^{\gamma} }\leq \sum_{|n_1|>M}\frac{C_{\gamma}}{\langle n_1\rangle^{\alpha+\gamma}}\sum_{|n|\leq 5|n_1|}\frac{1}{\langle n\rangle^{2\sigma}}\leq \frac{C_{\gamma}}{M^{\epsilon_0}},
$$
provided that $\sigma>\frac{3(1-\alpha)}{2}$. Finally, for $|n-n_1|>4|n_1|$, we have $|n_1|\leq \frac{|n|}{3}$ and $|n-n_1|\geq |n|-|n_1|\geq \frac{2|n|}{3}$, hence
$$ \sum_{|n_1|>M,|n-n_1|>4|n_1| }\frac{C_{\gamma}}{\langle n\rangle^{2\sigma}\langle n_1\rangle^{\alpha}\langle n-n_1\rangle^{\gamma}}\leq \sum_{|n|\gtrsim M, |n_1|\leq \frac{|n|}{3}}\frac{C_{\gamma}}{\langle n\rangle^{2\sigma+\gamma}\langle n_1\rangle^{\alpha}}\leq \frac{C_{\gamma}}{M^{\epsilon_0}},
$$
provided that $\sigma>\frac{3(1-\alpha)}{2}$. This completes the proof of Proposition \ref{Cauchysequence}.
\end{proof}
%%%%%%%%%%%%%%%%%%%%%%%%%%%%%%%%%%%%%%%%%%%%%%%%%%%%%%%%%%%%%%%
Denote by
$$ g_N(u):=\frac{1}{2}\|\Pi_Nu\|_{L^4}^4-\|\Pi_Nu\|_{L^2}^4,\textrm{ then } g_N(u)=f_N(u)-b_N(u)^2.
$$
\begin{lemme}\label{fN}
Assume that $\frac{3}{4}<\alpha\leq 1$, then the sequence $(g_N)_{N\geq 1}$ is a Cauchy sequence in $L^2(\mathcal{X},\mathcal{B};d\mu)$. More precisely, for all $p\geq 2$ and $1\leq M<N$, 
\begin{equation}\label{f_Nlp}
 \|g_N(u)-g_M(u)\|_{L^p(d\mu)}\leq C(p-1)^{2}M^{-\frac{4\alpha-3}{2}}.
\end{equation}
Furthermore, for any $\lambda>0$, 
\begin{equation}\label{distributionalf_N}
\mu\{u\in \mathcal{X}: |g_N(u)-g_M(u)|>\lambda \}\leq Ce^{-c\lambda^{1/2}M^{\frac{4\alpha-3}{4}}}.
\end{equation}
\end{lemme}
\begin{proof}
We prove the estimate for $p=2$, and the general case will follow from Wiener chaos estimates.  Introduce the set
$$ A_N:=\{(n_1,n_2,n_3,n_4)\in \Z^4:|n_1|,|n_2|,|n_3|,|n_4|\leq N, n_1-n_2+n_3-n_4=0,n_2\neq n_1,n_3   \}
$$
and
\begin{equation*}
\begin{split}
A_{M,N}:=\{(n_1,n_2,n_3,n_4)\in \Z^4: & |n_1|,|n_2|,|n_3|,|n_4|\leq N, n_1-n_2+n_3-n_4=0,\\
&n_2\neq n_1,n_3,\max(|n_1|,|n_2|,|n_3|,|n_4|)>M                 \}.
\end{split}
\end{equation*} 
From direct computation, we have
$$ f_N(\phi_N)=-\sum_{A_N} \frac{g_{n_1}\ov{g}_{n_2}g_{n_3}\ov{g}_{n_4}}{[n_1]^{\frac{\alpha}{2}}[n_2]^{\frac{\alpha}{2}}[n_3]^{\frac{\alpha}{2}}[n_4]^{\frac{\alpha}{2}}}
+\sum_{|n|\leq N}\frac{|g_n|^4}{([n]^{\frac{\alpha}{2}})^4},
$$
and
$$ f_N(\phi_N)-f_M(\phi_M)=-\sum_{A_{M,N}}\frac{g_{n_1}\ov{g}_{n_2}g_{n_3}\ov{g}_{n_4}}{[n_1]^{\frac{\alpha}{2}}[n_2]^{\frac{\alpha}{2}}[n_3]^{\frac{\alpha}{2}}[n_4]^{\frac{\alpha}{2}}}
+ \sum_{M\leq |n|\leq N}\frac{|g_n|^4}{([n]^{\frac{\alpha}{2}})^4}.
$$
Now we estimate
\begin{equation}\label{a5}
\begin{split}
&\|f_N(u)-f_M(u)\|_{L^2(d\mu)}^2=\mathbb{E}\left[|f_N(\phi_N)-f_M(\phi_M)|^2\right]\\
\leq &C\sum_{(n_1,n_2,n_3,n_4)\in A_N}\sum_{(m_1,m_2,m_3,m_4)\in A_N }\mathbb{E}\Big[\frac{g_{n_1}\ov{g}_{n_2}g_{n_3}\ov{g}_{n_4}}{[n_1]^{\frac{\alpha}{2}}[n_2]^{\frac{\alpha}{2}}[n_3]^{\frac{\alpha}{2}}[n_4]^{\frac{\alpha}{2}}  }\frac{\ov{g}_{m_1}g_{m_2}\ov{g}_{m_3}g_{m_4}}{[m_1]^{\frac{\alpha}{2}}[m_2]^{\frac{\alpha}{2}}[m_3]^{\frac{\alpha}{2}}[m_4]^{\frac{\alpha}{2}}  }\Big]\\
+& C\sum_{M\leq |n|,|m|\leq N}\mathbb{E}\left[\frac{|g_{n}|^4|g_m|^4}{([n]^{\frac{\alpha}{2}})^4
([m]^{\frac{\alpha}{2}})^4}\right]
\end{split}
\end{equation}
By the independence of the Gaussian variables, $$ \mathbb{E}\Big[\frac{g_{n_1}\ov{g}_{n_2}g_{n_3}\ov{g}_{n_4}}{[n_1]^{\frac{\alpha}{2}}[n_2]^{\frac{\alpha}{2}}[n_3]^{\frac{\alpha}{2}}[n_4]^{\frac{\alpha}{2}}  }\frac{\ov{g}_{m_1}g_{m_2}\ov{g}_{m_3}g_{m_4}}{[m_1]^{\frac{\alpha}{2}}[m_2]^{\frac{\alpha}{2}}[m_3]^{\frac{\alpha}{2}}[m_4]^{\frac{\alpha}{2}}  }\Big]=0
$$
unless $\{n_1,n_2,n_3,n_4\}=\{m_1,m_2,m_3,m_4 \}$. Therefore,
\begin{equation}\label{a6}
\begin{split}
\eqref{a5}\leq &C\sum_{A_N} \frac{1}{([n_1]^{\frac{\alpha}{2}}
	[n_2]^{\frac{\alpha}{2}}
	[n_3]^{\frac{\alpha}{2}}
	[n_4]^{\frac{\alpha}{2}}  )^2  }+C\Big(\sum_{M\leq |n|\leq N}\frac{1}{([n]^{\frac{\alpha}{2}})^4}\Big)^2
\end{split}
\end{equation}
The second term on the right side can be bounded by
$ \frac{C}{M^{2(2\alpha-1)}},
$
provided that $2\alpha>1$. For the first term, by symmetry of the sum, we may majorize it by
\begin{equation}\label{a7}
 C\sum_{n_1,n_2,n_3\in \Z,|n_1|>M}\frac{1}{\langle n_1\rangle^{\alpha}\langle n_2\rangle^{\alpha}\langle n_3\rangle^{\alpha}
\langle n_1-n_2+n_3\rangle^{\alpha} }.
\end{equation}
Applying Lemma \ref{lemme-summation}, we have
\begin{equation*}\label{a8}
\begin{split}
\eqref{a7}\leq &C\sum_{n_1,n_2\in \Z, |n_1|>M}\frac{1}{\langle n_1\rangle^{\alpha}\langle n_2\rangle^{\alpha}\langle n_1-n_2\rangle^{2\alpha-1}  }\\
\leq &C\sum_{|n_1|>M}\frac{1}{\langle n_1\rangle^{4\alpha-2}}\leq \frac{C}{M^{4\alpha-3}},
\end{split}
\end{equation*}
provided that $\alpha>\frac{3}{4}$, where from the first inequality to the second, we divide the region of summation as $|n_1-n_2|\leq\frac{|n_1|}{2}, \frac{|n_1|}{2}\leq |n_1-n_2|<4|n_1|$ and $|n_1-n_2|\geq 4|n_1|$ as in the proof of Proposition \ref{Cauchysequence}. 

To prove \eqref{distributionalf_N}, using Tchebyshev inequality and \eqref{f_Nlp}, for any $p>0$, we have
$$ \mu\{u\in X(\T):|g_N(u)-g_M(u)|>\lambda \}\leq \left(\frac{C}{\lambda M^{\frac{4\alpha-3}{2}}}\right)^p (p-1)^{2p}.
$$
Choosing $p=\big(\frac{\lambda M^{\frac{4\alpha-3}{2}}}{C}\big)^{1/2}e^{-1}$, we obtain that \eqref{distributionalf_N}. This completes the proof.
\end{proof}

Following the argument in \cite{BTT-Toulouse}, we prove Proposition \ref{convergence-Gibbs}.
\begin{proof}[Proof of Proposition \ref{convergence-Gibbs}]
We use Nelson type argument. First we prove the large deviation for $f_N(u)-f_M(u)$. Recall that $f_N(u)=g_N(u)+b_N(u)^2$, we have
$$ f_N(u)-f_M(u)=g_N(u)-g_M(u)+(b_N(u)-b_M(u))(b_N(u)+b_M(u) ).
$$
Therefore,  $\mu\{u:|f_N(u)-f_M(u)|>a  \}$ can be bounded by 
$$ \mu\{u: |g_N(u)-g_M(u)|>a/2 \}+\mu\{u:|b_N(u)-b_M(u)||b_N(u)+b_M(u)|>a/2  \}.
$$
By Lemma \ref{fN}, the first measure can be bounded by $Ce^{-ca^{1/2}M^{\frac{4\alpha-3}{4}}}$. To estimate the second measure, we write
$$ (b_N(u)-b_M(u))(b_N(u)+b_M(u) )=(b_N(u)-b_M(u))^2+2b_M(u)(b_N(u)-b_M(u) ).
$$
From Lemma \ref{large-deviation},
$$ \mu\{u: |b_N(u)-b_M(u)|^2>a/4\}\leq Ce^{-ca^{1/2}M^{\alpha}}.
$$ 
It remains to estimate $\mu\{u:|b_M(u)(b_N(u)-b_M(u))|>a/4 \}$. From Lemma \ref{large-deviation}, we have for any $a'\geq 1$,
$$ \mu\{u: |b_M(u)|>a' \}\leq Ce^{-ca'}.
$$
Therefore, for any $a'>0$, we have 
\begin{align*}
&\mu\{u: |b_N(u)-b_M(u)||b_M(u)|>a/4  \}\\ \leq &\mu\{u:|b_M(u)|>a' \}+\mu\{u:|b_M(u)(b_N(u)-b_M(u))|>a/4,|b_M(u)|\leq a' \}\\
\leq & \mu\{u:|b_M(u)|>a' \}+\mu\{u: |b_N(u)-b_M(u)|>a/(4a') \}\\
\leq &Ce^{-ca'}+Ce^{-c\frac{a}{4a'}M^{\alpha}},
\end{align*}
provided that $\frac{a}{a'}\gtrsim M^{1-\alpha}$, where we have used Lemma \ref{large-deviation}. When $\alpha>\frac{2}{3}$, we must have $M^{\frac{\alpha}{2}}>M^{1-\alpha}$. 
By optimally choosing $a'=a^{1/2}M^{\frac{\alpha}{2}}$, we obtain that
$$ \mu\{u: |b_N(u)-b_M(u)||b_N(u)+b_M(u)|>a/2 \}\leq Ce^{-ca^{1/2}M^{\frac{\alpha}{2}}}<Ce^{-ca^{1/2}M^{\frac{4\alpha-3}{4}}}.
$$
Therefore, for $a\geq 1$,
\begin{align*}
\mu\{u:|f_N(u)-f_M(u)|>a \}\leq Ce^{-c a^{1/2}M^{\frac{4\alpha-3}{4}}}.
\end{align*}
This yields the $L^p$ convergence of $f_N(u)$. To complete the proof, we need show that 
$$ \|e^{-f_N(u)}\|_{L^p(d\mu)}\leq C,
$$
independent of $N$. Since we can write
$$ -f_N(u)=\alpha_N^2-\frac{1}{2}\int_{\T}\big(|\Pi_Nu|^2-2\alpha_N \big)^2,
$$
we have
$$ -f_N(u)\leq \alpha_N^2\leq CM^{2(1-\alpha)}.
$$
For fixed $\lambda\geq 1$ large, we choose $M$ such that $M^{2(1-\alpha)}= \theta\log\lambda$ with $0<\theta\ll 1$ such that
$ \log\lambda-CM^{2(1-\alpha)}\geq \frac{1}{2}\log\lambda,
$ 
thus
$$ -f_N(u)+f_M(u)\geq -f_N(u)-CM^{(1-\alpha)}\geq \frac{1}{2}\log\lambda.
$$
Therefore,
$$ \mu\big\{u: -f_N(u)>\log\lambda \big\}\leq \mu\big\{u: -f_N(u)+f_M(u)>\frac{1}{2}\log\lambda \big\}\leq Ce^{-c(\log\lambda)^{\frac{1}{2}+\frac{4\alpha-3}{8(1-\alpha)} }}\leq C_L\lambda^{-L}
$$
for all $L\in\N$, provided that
$$ \frac{1}{2}+\frac{4\alpha-3}{8(1-\alpha)}>1,\text{ i.e. } \alpha>\frac{7}{8}.
$$
This completes the proof of Proposition \ref{convergence-Gibbs}.
\end{proof}

Finally, the proof of Theorem \ref{thm2} (the same for Theorem \ref{thm1}) follows from the same probabilistic compactness argument as in \cite{BTT-Toulouse}, and we omit the details here. 
%%%%%%%%%%%%%%%%%%%%%%%%%%%%%%%%%%%%%%%%%%%%%%%%%%%%%%%%%%%%%%%%%%%%%%%%%%%%
\section{ Appendix: General convergence theorem and deterministic nonlinear estiamtes on compact manifold}
It turns out that the argument of Bourgain-Bulut also works for the fractional NLS with a quite  general nonlinearity  on any compact Riemannian manifold. 
More precisely, let $(\mathcal{M},g_0)$ be a compact Riemannian manifold (without boundary) of dimension $d$. 
Denote by $\Delta_{g_0}$ the Beltrami-Laplace operator with eigenvalues $(-\lambda_n^2)_{n\in\N}$ and associated eigenfunctions $(\varphi_n(x))_{n\in\N}$
($
-\Delta_{g_0} \varphi_n=\lambda_n^2\varphi_n
$).
Consider the truncated fractional NLS
\begin{equation}\label{general:truncated-NLS}
\begin{cases}
& i\partial_tu+(-\Delta_{g_0} )^{\frac{\alpha}{2}}u+\Pi_N(|u|^{p-1}u)=0,\\
&u|_{t=0}
%=\Pi_N\phi^{\omega}(x)
=\displaystyle{\sum_{\lambda_n\leq N}\frac{g_n(\omega)}{\lambda_n^2+1}}\varphi_n(x),
\end{cases}
\end{equation}
where $\Pi_{N}$ is the orthogonal projection (with respect to the $L^2(\mathcal{M})$ scalar product) on ${\rm span}(\varphi_n)_{1\leq \lambda_n \leq N}$.
We have the following theorem\footnote{ For simplicity we consider only the polynomial nonlinearity here, our argument applies to more general nonlinearities having polynomial growth and defocusing feature.}:
\begin{theorem}\label{thmA}
Assume that $\alpha>d$ and $\sigma<\frac{\alpha-d}{2}$.  The sequence  $(u^{\omega}_{N})_{N\in\N}$ 
of solutions of \eqref{general:truncated-NLS} converges a.s. in $C(\R;H^{\sigma}({\mathcal M}))$  to some limit $u$ which solves 
$$
 i\partial_tu+(-\Delta_{g_0} )^{\frac{\alpha}{2}}u+|u|^{p-1}u=0
$$
in the distributional sense.  
\end{theorem}
The proof of Theorem~\ref{thmA} follows from the same lines as in the proof of Theorem~\ref{thm3}. We only sketch here the main ingredients. For the probabilistic side, to establish the analogues of Lemma \ref{apriori-prob} and Corollary \ref{apriori-prob-cor}, we can not use that fact that $\varphi_n(x)$ are bounded, uniformly in $n$.  We should use instead the following average effect of eigenfunctions due to H\"ormander. 
\begin{lemme}\label{Hormander}
There exists $C=C(\mathcal{M},g_0)>0$, such that for any $N$, we have
$$ C^{-1}N^d\leq \sum_{N\leq \lambda_n\leq 2N}|\varphi_n(x)|^2\leq CN^d
$$	
\end{lemme}
For the deterministic side, we need to prove a relatively standard nonlinear estimate needed in the convergence argument. We present it here for its own interest. 
The following proposition proved in \cite{BGT} allows us to reduce the analysis to paraproduct type arguments in $\mathbb{R}^d$.
\begin{proposition}[\cite{BGT}]\label{BGT}
	Let $P$ be an elliptic self-adjoint differential operator of order $m>0$ on a compact manifold $\mathcal{M}$ of dimension $d$. Let $\psi\in C^{\infty}(\mathbb{R})$, $\kappa:U\subset\mathbb{R}^d\rightarrow V\subset \mathcal{M}$ a coordinate patch, and $\chi_1,\chi_2\in C_c^{\infty}(V)$ such that $\chi_2=1$ near the support of $\chi_1$. Then there exists a sequence $(\psi_j)_{j\geq 0}$ of $C_c^{\infty}(U\times\mathbb{R}^d)$ such that, for every $L\in\mathbb{N}$ and for every $h\in(0,1)$, $\nu\in[0,L]$, $f\in C^{\infty}(\mathcal{M})$, we have
	\begin{equation*}\label{BGT1}
	\Big\|\kappa^*\left(\chi_1\psi(h^mP)f\right)-\sum_{j=1}^{L-1}h^j\psi_j(x,hD_x)\kappa^*(\chi_2f)\Big\|_{H^{\nu}(\mathbb{R}^d)}\leq C_Lh^{L-\nu}\|f\|_{L^2(\mathcal{M})}.
	\end{equation*} 
	Moreover, $\psi_0(x,\xi)=\chi_1(\kappa(x))\psi(p_m(x,\xi))$ and
	$$ \mathrm{supp}(\psi_j)\subset\{(x,\xi)\in U\times\mathbb{R}^d: \kappa(x)\in \mathrm{supp}(\chi_1),p_m(x,\xi)\in\mathrm{supp}(\psi)\},
	$$
	where $p_m$ is the principle symbol of $P$.
\end{proposition}
We will use different notations for the  Littlewood-Paley decomposition in this appendix. We denote by $\Delta_l=\psi(-2^{2l}\Delta_{g_0})$ for $l\geq 1$ and $\Delta_0=\psi_0(-\Delta_{g_0})$, where $\psi_0\in C_c^{\infty}(|\xi|\leq 2)$ and $\psi\in C_c^{\infty}(\frac{1}{2}<|\xi|\leq 2)$. The Besov space $B_{r,q}^{s}(\mathcal{M})$ is defined via the norm
$$ \|f\|_{B_{r,q}^{s}(\mathcal{M})}:=\big\|2^{ls}\|\Delta_l f\|_{L^r(\mathcal{M})} \big\|_{l^q(\mathbb{N})}=\Big(\sum_{l\geq 0}2^{qls}\|\Delta_lf\|_{L^r(\mathcal{M})}^q \Big)^{\frac{1}{q}}.
$$
The Sobolev space $H^s(\mathcal{M})$ in then $B_{2,2}^s(\mathcal{M})$.
\begin{lemme}\label{nonlinear1}
	Let $F:\mathbb{C}\rightarrow C$ satisfies $F(0)=0$ and
	$$ |F(z)|\leq C|z|^{\nu},\quad  |\partial^{l} F(z)|\leq C|z|^{\nu-l},\quad l=1,2,
	$$
	with $\nu\geq 2$. Then for any $\sigma\in(0,1)$, $2\leq r<\infty$ we have
	$$ \|F(u)\|_{B_{r,2}^{\sigma}(\mathcal{M})}\leq C\|u\|_{L^{\infty}(\mathcal{M})}^{\nu-1}\|u\|_{B_{r,2}^{\sigma}(\mathcal{M})}.
	$$	
\end{lemme}

\begin{proof}
It is sufficient to estimate $\|\Delta_lF(u)\|_{L^2(\mathcal{M})}$ in one coordinate patch.
	Applying Proposition~\ref{BGT} to the operator $\Delta_l=\psi(-2^{-2l}\Delta_{g_0})$, we have
	$$ 
	\kappa^*(\chi_1\Delta_lF(u))=\sum_{j=0}^{L-1}\psi_j(x,2^{-l}D_x)\kappa^*(\chi_2F(u))+R_{L,l}
	$$
	with $$\psi_0(x,\xi)=\chi_1(\kappa(x))\chi\left(|\xi|_{g_0}^2\right),
	\quad  |\xi|_{g_0}^2:=\sum_{i,j}g_0^{i,j}(x)\xi_i\xi_j,
	$$
	and
	$$ \|R_{L,l}\|_{L^r(\R^d)}\lesssim \|R_{L,l}\|_{H^{\nu}(\R^d)}\lesssim 2^{-l(L-\nu)}\|F(u)\|_{L^2(\mathcal{M})},
	$$
	where we take $L>\nu$ large enough such that $H^{\nu}(\R^d)\hookrightarrow L^r(\R^d)$.
	Let
	$$ 1=\theta_0(\xi)+\sum_{l=1}^{\infty}\theta(2^{-l}\xi)
	$$
	be a dyadic partition of unity in $\mathbb{R}^d$, $\theta_0\in C_c^{\infty}(\mathbb{R}^d),\theta\in C_c^{\infty}(\mathbb{R}^d\setminus\{0\})$. Denote by
	$\theta_j(\cdot)=\theta(2^{-j}\cdot),$ for $j\geq 1$, we denote by $\widetilde{\Delta}_j=\theta_j(D)$ be the usual Littlewood-Paley dyadic projector in $\mathbb{R}^d$ and 
	$$ \widetilde{S}_j:=\sum_{k\leq j}\widetilde{\Delta}_k.
	$$
	Note that on the support of $\chi_1$,
	$$ a|\xi|^2\leq |\xi|_{g_0}^2\leq b|\xi|^2,
	$$
	in view of the support property of $\psi_j$, the standard pseudodifferential calculus implies, if $|l'-l|\geq \nu_0$ for some fixed positive constant $\nu_0$, we have
	$$\|\theta_{l'}(D)\psi_j(x,2^{-l}D)\kappa^*(\chi_2F(u))\|_{L^r(\mathbb{R}^d)}\lesssim 2^{-l}\|\rho(2^{-l'}D)\kappa^*(\chi_2F(u))\|_{L^2(\mathbb{R}^d)}
	$$
	for some $\rho\in C_c^{\infty}(\mathbb{R}^d\setminus\{0\})$.
	Therefore, we have
	\begin{equation}\label{reductionE}
	\begin{split}
	\|\kappa^*(\chi_1\Delta_lF(u))\|_{L^r(\mathbb{R}^d)}\lesssim & \sum_{|l'-l|\leq \nu_0}\|\widetilde{\Delta}_{l'}(\kappa^*(\chi_2F(u)))\|_{L^r(\mathbb{R}^d)}+2^{-l}\|F(u)\|_{L^2(\mathcal{M})}.
	\end{split}
	\end{equation}
	We could replace the error by $2^{-l}\|F(u)\|_{L^r(\mathcal{M})}$ since $L^r(\mathcal{M})\hookrightarrow L^2(\mathcal{M})$.
	Denote by $v=\kappa^*u=u\circ\kappa$, and $\widetilde{\chi}_j=\chi_j\circ\kappa$, $j=1,2$. Without loss of generality, we may assume that $v$ has compact suppourt in $\mathbb{R}^d$. Observe that 
	$$ \|[\widetilde{\Delta}_l,\widetilde{\chi}_2]\|_{L^r(\mathbb{R}^d)\rightarrow L^r(\mathbb{R}^d)}\lesssim 2^{-l}, 
	$$
	we have
	\begin{equation*}\label{r1}
	\begin{split}
	\|\kappa^*(\chi_1\Delta_lF(u))\|_{L^r(\mathbb{R}^d)}\lesssim & \sum_{|l'-l|\leq \nu_0}\|\widetilde{\chi}_2\widetilde{\Delta}_{l'}(F(v))\|_{L^r(\mathbb{R}^d)}+2^{-l}\|F(u)\|_{L^r(\mathcal{M})}.
	\end{split}
	\end{equation*}
	Now we have reduced all the functions and operators to $\mathbb{R}^d$ and we can perform the standard analysis. We write
	$$ F(v)=\sum_{j\geq 0}\Big[F(\widetilde{S}_jv)-F(\widetilde{S}_{j-1}v)\Big]:=\sum_{j\geq 0}m_j\widetilde{\Delta}_jv,
	$$ 
	with the convention that $\widetilde{S}_{-1}=0$, 
	where
	$$ m_j=\int_0^1F'(\tau\widetilde{S}_jv+(1-\tau)\widetilde{S}_{j-1}v)d\tau.
	$$
	We write the product as
	$$ \sum_{j\geq 0}m_j\widetilde{\Delta_j}v=\sum_{j\geq 0}\widetilde{S}_{j-2}m_j\widetilde{\Delta_j}v+\sum_{k,j:k\geq j-2}\widetilde{\Delta}_km_j\widetilde{\Delta_j}v.
	$$
	The first term on the right hand side can be estimated as
	\begin{equation*}\label{rhs1}
	\begin{split}
	&\Big\|\widetilde{\chi}_2\widetilde{\Delta}_l\Big(\sum_{j\geq 0}\widetilde{S}_{j-2}m_j\Delta_jv\Big)\Big\|_{L^r(\mathbb{R}^d)}\\
	\leq &\Big\|\sum_{|j-l|\leq 2}\widetilde{S}_{j-2}m_j\cdot\widetilde{\chi}_2\widetilde{\Delta}_jv\Big\|_{L^r(\mathbb{R}^d)}
	+2^{-l}\Big\|\sum_{|j-l|\leq 2}\widetilde{S}_{j-2}m_j\cdot\widetilde{\Delta}_jv\Big\|_{L^r(\mathbb{R}^d)}\\
	\lesssim &\sup_{j}\|\widetilde{S}_{j-2}m_j\|_{L^{\infty}(\mathbb{R}^d)}\Big(\sum_{|j-l|\leq 2}\|\widetilde{\chi}_2\widetilde{\Delta}_jv\|_{L^r(\mathbb{R}^d)}+2^{-l}\|v\|_{L^r(\mathbb{R}^d)}\Big)
	\\
	\lesssim &\|u\|_{L^{\infty}(\mathcal{M})}^{p-1}\Big(\sum_{|j-l|\leq 2}\|\Delta_ju\|_{L^r(\mathcal{M})}+C2^{-l}\|u\|_{L^r(\mathcal{M})}\Big),
	\end{split}
	\end{equation*}
	where in the last inequality, we have used the estimates
	$$ \|\widetilde{S}_{j-2}m_j\|_{L^{\infty}(\mathbb{R}^d)}\lesssim \|v\|_{L^{\infty}(\mathbb{R}^d)}^{p-1}\lesssim \|u\|_{L^{\infty}(\mathcal{M})}^{p-1},\quad \|v\|_{L^r(\mathbb{R}^d)}\leq \|u\|_{L^r(\mathcal{M})}.
	$$
	Moreover, we have also applied Proposition \ref{BGT} to replace $\|\widetilde{\chi}_2\widetilde{\Delta}_jv\|_{L^r(\mathbb{R}^d)}$ by $\|\Delta_ju\|_{L^r(\mathcal{M})}$ and an error term absorbed in $2^{-j}\|u\|_{L^r(\mathcal{M})}$, as in the argument we have used just now.  
	
	Therefore,
	\begin{equation*}\label{lowhigh}
	\sum_{l\geq 0 }2^{2l\sigma}\Big\|\widetilde{\chi}_2\widetilde{\Delta}_l\Big(\sum_{j\geq 0 }\widetilde{S}_{j-2}m_j\widetilde{\Delta}_jv\Big)\Big\|_{L^r(\mathbb{R}^d)}^2\lesssim \Big(\|u\|_{L^{\infty}(\mathcal{M})}^{p-1}\|u\|_{B_{r,2}^{\sigma}(\mathcal{M})}\Big)^2.
	\end{equation*}
	To estiamte the other term, we write
	\begin{equation*}
	\begin{split}
	\widetilde{\chi}_2\widetilde{\Delta}_l\Big(\sum_{k,j:k\geq j-2}\wt{\Delta}_km_j\wt{\Delta_j}v\Big)=\wt{\chi}_2\wt{\Delta}_l\Big(\sum_{k\geq l-10}\sum_{j\leq k+2}\wt{\Delta}_km_j\wt{\Delta}_jv\Big).
	\end{split}
	\end{equation*}
	Thanks to Bernstein, we apply the following type of control
	$$ \Big\|\sum_{k\geq l-10}\sum_{j\leq k+2}\wt{\Delta}_kG\cdot\wt{\Delta}_jH\Big\|_{L^r(\mathbb{R}^d)}\lesssim \sum_{k\geq l-10}\sum_{j\leq k+2}2^{-k}\|\nabla_x\left(\wt{\Delta}_kG\cdot\wt{\Delta}_jH\right)\|_{L^r(\mathbb{R}^d)}
	$$
	and obtain that
	\begin{equation}\label{highhigh}
	\begin{split}
	&\Big\|\wt{\chi}_2\wt{\Delta}_l\Big(\sum_{k\geq l-10}\sum_{j\leq k+2}\wt{\Delta}_km_j\wt{\Delta}_jv\Big)\Big\|_{L^r(\mathbb{R}^d)}\\
	\lesssim & \sum_{k\geq l-10}\sum_{j\leq k+2}2^{-k}\Big(\|\wt{\Delta}_k\nabla m_j\|_{L^{\infty}(\mathbb{R}^d)}\|\wt{\Delta}_jv\|_{L^r(\mathbb{R}^d)}+\|\wt{\Delta}_km_j\|_{L^{\infty}(\mathbb{R}^d)}\|\wt{\Delta}_j\nabla  v\|_{L^r(\mathbb{R}^d)}\Big).
	\end{split}
	\end{equation}
	Note that $v=\widetilde{\chi}_3v$ for some $\widetilde{\chi}_3\in C_c^{\infty}(\mathbb{R}^d)$, and we have from commutator estimate that 
	$$ \|\wt{\Delta}_jv\|_{L^r(\mathbb{R}^d)}\lesssim \|\wt{\chi}_3\wt{\Delta}_jv\|_{L^r(\mathbb{R}^d)}+2^{-j}\|u\|_{L^r(\mathcal{M})},
	$$
	$$ \|\wt{\Delta}_j\nabla v\|_{L^r(\mathbb{R}^d)}\lesssim 2^j\|\wt{\chi}_3\wt{\Delta}_jv\|_{L^r(\mathbb{R}^d)}+\|u\|_{L^r(\mathcal{M})}.
	$$
	Now from the pointwise estimate 
	$$ \|\wt{\Delta}_k\nabla m_j\|_{L^{\infty}(\mathbb{R}^d)}\lesssim 2^j\left(\|\wt{S}_{j}v\|_{L^{\infty}(\mathbb{R}^d)}^{p-1}+\|\wt{S}_{j-1}v\|_{L^{\infty}(\mathbb{R}^d)}^{p-1}\right)\lesssim 2^{j}\|u\|_{L^{\infty}(\mathcal{M})}^{p-1},
	$$
we have
	\begin{equation*}
	\begin{split}
	\eqref{highhigh}\lesssim &\|u\|_{L^{\infty}(\mathcal{M})}^{p-1}\sum_{k\geq l-10}\sum_{j\leq k+2}\Big(2^{-(k-j)}\|\wt{\chi}_3\wt{\Delta}_jv\|_{L^r(\mathbb{R}^d)}+2^{-k}\|u\|_{L^r(\mathcal{M})}\Big)\\
	\lesssim &\|u\|_{L^{\infty}(\mathcal{M})}^{p-1}\sum_{k\geq l-10}\sum_{j\leq k+2}\Big(2^{-(k-j)}\|\Delta_ju\|_{L^r(\mathcal{M})}+2^{-k}\|u\|_{L^r(\mathcal{M})}\Big)\\
	\lesssim &\|u\|_{L^{\infty}(\mathcal{M})}^{p-1}\sum_{k\geq l-10}2^{-k\sigma}\sum_{j\leq k+2}2^{-(k-j)(1-\sigma)}2^{j\sigma}\|\Delta_ju\|_{L^r(\mathcal{M})}\\+&l2^{-l}\|u\|_{L^{\infty}(\mathcal{M})}^{p-1}\|u\|_{L^r(\mathcal{M})}.
	\end{split}
	\end{equation*}
	Thus Young's convolution inequality on $l^2$ yields
	$$ \sum_{l\geq 0}2^{2l\sigma}\Big\|\wt{\chi}_2\wt{\Delta}_l\Big(\sum_{k\geq l-10}\sum_{j\leq k+2}\wt{\Delta}_km_j\wt{\Delta}_jv\Big)\Big\|_{L^r(\mathbb{R}^d)}^2\lesssim \|u\|_{B_{r,2}^{\sigma}(\mathcal{M})}^2\|u\|_{L^{\infty}(\mathcal{M})}^{p-1}.
	$$
	This completes the proof of Lemma \ref{nonlinear1}.
\end{proof}
We also need the following type of paraproduct estimate.
\begin{lemme}\label{nonlinear2} 
	We have
	$$ \|fg\|_{H^s(\mathcal{M})}\leq C_{s,\sigma_1,r}\|f\|_{H^s(\mathcal{M})}\|g\|_{B_{r,2}^{\sigma_1}(\mathcal{M})}
	$$	
	for any $0<s<\sigma_1<1$ and $r>\frac{d}{\sigma_1-s}$.
\end{lemme}

\begin{proof}
	Applying \eqref{reductionE} by replacing $F(u)$ to $f\cdot g$, we have
	\begin{equation*}
	\begin{split}
	\|\kappa^*(\chi_1\Delta_l(fg))\|_{L^2(\mathbb{R}^d)}\leq & C\sum_{|l'-l|\leq \nu_0}\|\widetilde{\Delta}_{l'}(\kappa^*(\chi_2(fg)))\|_{L^2(\mathbb{R}^d)}+C2^{-l}\|fg\|_{L^2(\mathcal{M})}.
	\end{split}
	\end{equation*}
	Again, we denote by $\wt{\chi}_1=\chi_1\circ\kappa$,  $v=f\circ\kappa=\wt{\chi}_3v,$ and $ w=g\circ\kappa=\wt{\chi}_3w$ with $\wt{\chi}_3\in C_c^{\infty}(U)$. Now we write
	$$v\cdot w=T_vw+T_wv+R(v,w),
	$$
	with
	$$ T_vw=\sum_{j\geq 0}\wt{S}_{j-2}v\wt{\Delta}_jw,\quad T_wv=\sum_{j\geq 0}\wt{S}_{j-2}w\wt{\Delta}_jv,
	\text{ and }
 R(v,w)=\sum_{|j-k|\leq 2}\wt{\Delta}_jv\wt{\Delta}_kw.
	$$
We estimate
	\begin{equation*}\label{3}
	\begin{split}
	\|\wt{\chi}_1\wt{\Delta}_l(T_wv)\|_{L^2(\mathbb{R}^d)}\leq &\Big\|\wt{\Delta}_l\Big(\sum_{|j-l|\leq 2}\wt{S}_{j-2}w\wt{\Delta}_jv\Big)\Big\|_{L^2(\mathbb{R}^d)}\\
	\leq &\sum_{|j-l|\leq 2}\|\wt{S}_{j-2}w\|_{L^{\infty}(\mathbb{R}^d)}\|\wt{\Delta}_j(\wt{\chi}_3v)\|_{L^2(\mathbb{R}^d)}\\
	\lesssim & \|g\|_{L^{\infty}(\mathcal{M})}\sum_{|j-l|\leq 2}\Big(\|\Delta_j f\|_{L^2(\mathcal{M})}+2^{-l}\|f\|_{L^2(\mathcal{M})}\Big),
	\end{split}
	\end{equation*}
	where in the last inequality, we have used the $\|[\wt{\Delta}_j,\wt{\chi}_3]\|_{L^2\rightarrow L^2}\leq C2^{-j}$ and Proposition \ref{BGT} as in the proof of Lemma \ref{nonlinear1}. Therefore, from the embedding $B_{r,2}^{\sigma_1}\hookrightarrow L^{\infty}$, we have 
	$$ \|\wt{\chi}_1(T_{w}v)\|_{H^{s}(\mathbb{R}^d)}\lesssim \|g\|_{L^{\infty}(\mathcal{M})}\|f\|_{H^s(\mathcal{M})}\lesssim \|g\|_{B_{r,2}^{\sigma_1}(\mathcal{M})}\|f\|_{H^s(\mathcal{M})}.
	$$
	Similarly,
	\begin{equation*}
	\begin{split}
	2^{ls}\|\wt{\chi}_1\wt{\Delta}_lR(v,w))\|_{L^2(\mathbb{R}^d)}\leq &2^{ls}\Big\|\wt{\Delta}_l\Big(\sum_{|j-k|\leq 2,j\geq l-10}\wt{\Delta}_jv\wt{\Delta}_kw\Big)\Big\|_{L^2(\mathbb{R}^d)}\\
	\lesssim &2^{ls}\sum_{|j-k|\leq 2,j\geq l-10}\|\wt{\Delta}_kw\|_{L^{\infty}(\mathbb{R}^d)}\|\wt{\Delta}_j(\wt{\chi}_3v)\|_{L^2(\mathbb{R}^d)}\\
	\lesssim &2^{ls}\|g\|_{L^{\infty}(\mathcal{M})}\sum_{j\geq l-10}\left(\|\Delta_jf\|_{L^2(\mathcal{M})}+2^{-j}\|f\|_{L^2(\mathcal{M})}\right)\\
	\lesssim &\|g\|_{L^{\infty}(\mathcal{M})}\Big(\sum_{j\geq l-10}2^{js}\|\Delta_jf\|_{L^2(\mathcal{M})}\cdot 2^{-(j-l)s}+2^{-(1-s)l}\|f\|_{L^2(\mathcal{M})}\Big).
	\end{split}
	\end{equation*}
	Young's convolution inequality gives
	$$\|\wt{\chi}_1R(v,w)\|_{H^{s}(\mathbb{R}^d)}\leq C\|g\|_{L^{\infty}(\mathcal{M})}\|f\|_{H^s(\mathcal{M})}. 
	$$
	The treatement for the term $T_vw$ is a little different, since we still need put $L^2$ norm on $f$. We estimate
	\begin{equation*}
	\begin{split}
	2^{ls}\|\wt{\chi}_1\wt{\Delta}_l(T_vw)\|_{L^2(\mathbb{R}^d)}\leq &C2^{ls}\sum_{|j-l|\leq 2}\|\wt{S}_{j-2}v\|_{L^2(\mathbb{R}^d)}\|\wt{\Delta}_j(\wt{\chi}_3w)\|_{L^{\infty}(\mathbb{R}^d)}\\
	\leq &C2^{ls}\|f\|_{L^2(\mathcal{M})}\sum_{|j-l|\leq 2}\|\wt{\Delta}_j(\wt{\chi}_3w)\|_{L^{\infty}(\mathbb{R}^d)}\\
	\leq &C2^{ls}\|f\|_{L^2(\mathcal{M})}\sum_{|j-l|\leq 2}2^{\frac{jd}{r}}\|\wt{\Delta}_j(\wt{\chi}_3w)\|_{L^{r}(\mathbb{R}^d)},
	\end{split}
	\end{equation*}
	where we have used Bernstein in the last inequality. Thanks to $s+\frac{d}{r}< \sigma_1<1$, we can bound the right hand side by
	$$ C2^{l\left(s+\frac{d}{r}\right)}\|f\|_{L^2(\mathcal{M})}\sum_{|j-l|\leq 2}\|\Delta_jg\|_{L^r(\mathcal{M})}+C2^{-l(1-\sigma_1)}\|f\|_{L^2(\mathcal{M})}\|g\|_{L^r(\mathcal{M})}.
	$$
	Finally, we complete the proof of Lemma \ref{nonlinear2} by taking the $l^2$ norm of the above quantity. 
\end{proof}
Thanks to the established estimates, the proof of Theorem~\ref{thmA}  can be done exactly as we did in the proof of Theorem~\ref{thm3}.
%%%%%%

\end{document}